\documentclass[11pt]{article}
\usepackage[latin1]{inputenc}
\usepackage[english]{babel}
\newtheorem{theorem}{Theorem}
\newtheorem{lemma}{Lemma}
\newtheorem{definition}{Definition}
\newtheorem{proof}{Proof}
\usepackage{stmaryrd}
\usepackage{enumerate}
\usepackage{algorithm}
\usepackage[noend]{algpseudocode}
\usepackage{calrsfs}
\usepackage{graphicx}
\usepackage{amssymb}
\usepackage{array}
\usepackage{amsmath}
\usepackage{epstopdf}
\usepackage{epsfig}
\usepackage{subcaption}
\usepackage{pstricks}
\usepackage{fancyheadings}
\usepackage{pgfplots}
\usepackage{physics}
\usepackage{stackengine}
\usepackage[margin=1.0in]{geometry}
\DeclareMathAlphabet{\pazocal}{OMS}{zplm}{m}{n}

\usepackage{graphicx}

\makeatletter
\def\BState{\State\hskip-\ALG@thistlm}
\makeatother

\newcommand{\lJump}{[\![}
\newcommand{\rJump}{]\!]}

\newcommand{\el}{\boldsymbol{e}_l}
\newcommand{\er}{\boldsymbol{e}_r}
\newtheorem{remark}{Remark}

\begin{document}

\title{A conservative and energy stable discontinuous  spectral element method for the shifted wave equation in second order form}
\author{Kenneth Duru\thanks{Corresponding author. Mathematical Sciences Institute, The Australian National University, Canberra, Australia.} \and Siyang Wang \thanks{Department of Mathematics and Mathematical Statistics, Ume{\aa} University, Ume\aa,  Sweden.} \and Kenny Wiratama\thanks{Mathematical Sciences Institute, The Australian National University, Canberra, Australia.} }
\pagenumbering{arabic}
\maketitle
\begin{abstract}
In this paper, we develop a provably energy stable and conservative  discontinuous  spectral element method  for the shifted wave equation in second order form.  
The proposed method combines the advantages and central ideas of  very successful numerical techniques, the summation-by-parts  finite difference method, the spectral method and the discontinuous Galerkin method.
We prove energy-stability, discrete conservation principle, and derive error estimates in the energy norm for the (1+1)-dimensions shifted wave equation in second order form. The energy-stability results, discrete conservation principle, and the error estimates generalise to multiple dimensions using tensor products of quadrilateral and hexahedral  elements. Numerical experiments, in (1+1)-dimensions  and (2+1)-dimensions, verify the theoretical results  and demonstrate optimal convergence of $L^2$ numerical errors at subsonic, sonic and supersonic regimes.
\end{abstract}

\section{Introduction}
Second order  systems of hyperbolic partial differential equations (PDEs) often describe problems where wave phenomena are dominant. Typical examples are the acoustic wave equation, the elastic wave equation, and Einstein's equations of general relativity.  However, many solvers for wave equations and  Einstein's equations are designed for first order systems of hyperbolic PDEs \cite{DumbserFambriGaburro2019,DavidBrown2012}. In particular, multi-domain spectral methods which are increasingly becoming attractive because they are optimal in terms of  efficiency and accuracy, and are commonly  implemented  as first-order systems \cite{Boyle_et_al_2007, Lindblom_et_al_2006}. That is, the system of  second order hyperbolic PDEs are first reduced to a system of first order hyperbolic PDEs  before numerical approximations are introduced. The main reason is that the theory and numerical methods to solve  hyperbolic PDEs are well developed for first order hyperbolic systems, and less developed for second order hyperbolic systems. 

There are disadvantages of solving the equations in first order form which can be avoided if the equations are solved in second order form. These include the introduction of (non-physical) auxiliary variables with their constraints and boundary conditions.  For example, in the harmonic description of general relativity, Einstein's equations are a system of  10 curved space second order wave equations, while the corresponding reduction to first order systems will involve around 60 equations. The reduction to first order form is also less attractive from a computational point of view considering the efficiency and accuracy of numerical approximation. 

The main motivation of this work is the development of  \\
$\bullet$ efficient (explicit in time, no auxiliary variables, and  no matrix inversion), \\
$\bullet$ robust (provably stable),  \\
$\bullet$  conservative (constraints preserving) and  \\
$\bullet$  arbitrarily (spectrally)  accurate \\
discontinuous  spectral element methods (DSEM) for  Einstein's equations of general relativity in second order form, without the introduction of auxiliary variables.

It has been a long held ambition of the computational relativity community  to develop the theory, and robust numerical techniques for second order hyperbolic systems such that Einstein's equations can be solved efficiently \cite{HeinzOrtiz2002, OscarSarbach2011,TaylorKidderTeukolky2010}. This has proven to be an incredibly difficult task. Comparing with the standard wave equations in classical mechanics, eg. the acoustic wave equation, the elastic wave equation, etc,  Einstein's equations of general relativity with the space-time metric often results in a non-vanishing shift  and complicates the derivation of well-posed and energy-stable boundary and interface conditions for the continuous problem \cite{Fournodavlos2019OnTI, OscarSarbach2011}, leading to significant challenges in constructing provably stable and high order accurate numerical methods.  In particular, for very high order methods it is more difficult to guarantee stability for naturally second order systems than their corresponding first order forms \cite{Tichy2006,Tichy2009,Fieldetal2010,DavidBrown2012}, see the  progress in this direction  \cite{HeinzOrtiz2002, OscarSarbach2011, MattssonParisi2010,TaylorKidderTeukolky2010}.

In this paper, we take a first but an important step towards designing  provably stable and very high order accurate DSEM for Einstein's equations of general relativity in second order form. In particular, we will consider  the spatial numerical approximations and  focus on accurate and stable interface and boundary treatments. We consider the shifted wave equation in second order form, in one (and two) space dimension, as the suitable model problem which embodies most of the numerical challenges for Einstein's equations of general relativity and minimises technical difficulties. The shifted wave equation is also a prototype for problems in aero-acoustics, where the shift emanates from linearising  Euler equations of compressible fluid dynamics with non-vanishing mean flow. We will consider all flow regimes, namely subsonic,  sonic, and supersonic flow regimes. Similar to Einstein's equations of general relativity, comparing with the classic scalar wave equation, the presence of non-vanishing shift in the shifted wave equation imposes  difficulties to derive well-posed and energy-stable boundary and interface conditions for the continuous problem, leading to significant challenges in constructing robust and efficient numerical methods for all well-posed medium parameters of the shifted wave equation in second order form .

Summation by parts (SBP) finite difference (FD) methods \cite{MattssonNordstrom2004, Mattsson2012} have been developed to solve the wave equation in second order form \cite{Mattsson2008,Mattsson2009}.
Using the so-called compatible SBP FD operators, these methods have been extended to the elastic wave equation in heterogeneous media and complex geometries  \cite{Sjogreen2012,DuruKreissandMattsson2014,DuruVirta2014,Zhang2021}. However, the extension of the SBP FD methods to the  shifted wave equation generates high frequency exponentially growing numerical modes  \cite{MattssonParisi2010,SzilagyiKreissWinicour2005}, which will require artificial numerical dissipation to numerically stabilise the methods.  Artificial numerical dissipation can help in many ways but it can introduce some unwanted numerical artefacts.

DG methods have been developed for classical second order hyperbolic PDEs, such as the symmetric interior penalty discontinuous Galerkin method (SIPDG) for the wave equation \cite{GroteSchneebeliSchotzau2006}. This method is high order accurate, and is geometrically flexible by using unstructured grids. However, it is not straightforward to extend this SIPDG formulation to solve Einstein's equations for gravitational waves, or even its simplified model problem, the shifted wave equation. As above, the main difficulties arise from the presence of the shift, and mixed temporal and spatial derivatives, which make a straightforward application of the SIPDG numerical flux \cite{GroteSchneebeliSchotzau2006} impossible.  The energy-based DG method developed in \cite{Appelo2015} has recently been extended to the shifted wave equation \cite{Zhang2019}.

In this paper, we begin the development of a  robust and arbitrarily accurate multi-domain spectral  method for Einstein's equations of general relativity, with the shifted wave equation as our model problem. We will combine the advantages and central ideas of  three very successful numerical techniques, SBP FD methods, spectral methods  and DG methods. We will  introduce a strict compatibility condition (that we call the ultra-compatible SBP property)   for SBP operators that will enable stable and accurate  numerical treatment of well-posed second order hyperbolic problems with mixed temporal and spatial derivatives. The SBP operators will be derived using a Galerkin spectral approach that is common to DG methods. Then we will design conservative (constraint preserving)  and accurate numerical fluxes to couple locally adjacent spectral elements. We prove numerical stability for the method and derive a priori error estimates. Numerical experiments are presented in (1+1)-dimensions  and (2+1)-dimensions to verify the theoretical results.  The numerical solution is integrated in time using the 4th order accurate classic Runge-Kutta method. The numerical results corroborate the theory   and demonstrate optimal convergence of $L^2$ numerical errors at subsonic, sonic and supersonic flow regimes.

The remaining parts of the paper are organised as follows. In the next section, we introduce the notion of ultra-compatible SBP property for discrete operators approximating the first derivative operator $\partial /\partial x$ and the second derivative operator $\partial /\partial x\left(b  \partial /\partial x\right)$ with variable coefficient. Furthermore, we derive ultra-compatible spectral difference SBP operators and prove their accuracy. In section 3, we introduce the shifted wave equation, derive well-posed interface and boundary conditions, derive the conservation principle and prove that the model problem is well-posed for all possible parameters. In section 4, we present the spatial discretisation and derive energy estimates to prove stability. Numerical error estimates are derived in section 5. In section 6, we present numerical experiments. The numerical experiments corroborate the theoretical results. In section 7, we draw conclusions and suggest directions for future work.
 
\section{SBP spectral difference operators}
In this section, we first introduce the notion of ultra-compatible SBP property for discrete operators approximating the first derivative operator $\partial /\partial x$ and the second derivative operator $\partial /\partial x\left(b  \partial /\partial x\right)$ on a finite number of grid points in a closed interval, where the smooth function $b$ is always positive and describes the material property in the physical model. We will then construct the ultra-compatible SBP spectral difference operators and derive their global accuracy.

To begin, for real functions $u$ and $v$, we define the weighted $L^2$-inner product and the corresponding norm
\begin{align}\label{eq:l2norm}
\left(u, v\right)_{b\Omega} = \int_{\Omega} b(x)u(x)v(x) dx, \quad \|u\|^2_{b\Omega}= \left(u, u\right)_{b\Omega}, \quad b(x) > 0, \quad \forall x\in \Omega.
\end{align}
If $b(x) = 1$ we omit the subscript $b$, and we get the standard $L^2$-inner product $\left(u, v\right)_{\Omega}$ and the corresponding norm $\|u\|_{\Omega}$. 
We will also omit the subscript $\Omega$ when the context is clear. 
The equivalence of norms holds
\begin{align}\label{eq:l2norm_equivalence}
C_1\|u\|_{\Omega} \le \|u\|_{b\Omega}\le C_2\|u\|_{\Omega}, \quad C_1, C_2 > 0,
\end{align}
where
$$
C_1 = \min_{x}\sqrt{b(x)} >0, \quad C_2 = \max_{x}\sqrt{b(x)} >0.
$$
Consider any smooth function $u \in H^{m}(\Omega)$ for $m \ge 2$ in the interval $x\in \Omega = [B_{-}, B_{+}]$, integration-by-parts gives
\begin{align}\label{eq:IBP_dx}
\left(\frac{\partial u}{\partial x}, u\right)_{\Omega} = -\left(u, \frac{\partial u}{\partial x}\right)_{\Omega} + u^2\Big|_{B},
\end{align}
\begin{align}\label{eq:IBP_dxx}
\left(\frac{\partial }{\partial x}\left(b\frac{\partial u}{\partial x}\right), u\right)_{\Omega} = -\left(\frac{\partial u}{\partial x}, \frac{\partial u}{\partial x}\right)_{b\Omega} + \left(b\frac{\partial u}{\partial x}\right)u\Big|_{B},
\end{align}
  where the subscript $B$ for a function $u$ denotes the evaluation of the function at the boundaries as follows,
\begin{align}\label{eq:BD_continuous}
u\Big|_B = u(B_{+}) - u(B_{-}).
\end{align}
Discrete operators approximating the first derivative $\partial /\partial x$ and the second derivative $\partial /\partial x\left(b \partial /\partial x\right)$ on a finite number of grid points in a closed interval are called SBP operators  \cite{BStrand1994, MattssonNordstrom2004, Mattsson2012} if they mimic the integration-by-parts properties \eqref{eq:IBP_dx}--\eqref{eq:IBP_dxx} in a discrete inner product. We present the formal definitions in the following section. 
 
 \subsection{SBP operators}
  To be precise, we introduce $P+1$ grid points, $B_{-}=x_1<x_2<x_3 \cdots <x_{P+1}=B_{+}$, and let $D_x\approx {\partial }/{\partial x}$ and $D_{xx}^{(b)} \approx {\partial }/{\partial x}\left(b(x){\partial }/{\partial x}\right)$ denote the discrete operators approximating the derivatives on the grid.
 \begin{definition}\label{Def:SBP_Operator}
  The discrete derivative operators  $D_x \in \mathbb{R}^{(P+1)\times(P+1)}$ and $D_{xx}^{(b)} \in \mathbb{R}^{(P+1)\times(P+1)}$ are called {\it SBP operators} if for all $\mathbf{u} \in \mathbb{R}^{P+1}$,
  \begin{align*}
  &D_x = H_x^{-1} Q, \quad Q + Q^T = B, \quad H_x=H_x^T, \quad \mathbf{u}^T {H}_x \mathbf{u}>0,\ \forall \mathbf{u}\neq 0,\quad \mathbf{u}^TB\mathbf{u} = u_{P+1}^2 - u_{1}^2,\\
  &D_{xx}^{(b)} = H_x^{-1}\left(- M_x^{(b)} + B {b} S_x\right), \quad \mathbf{u}^TM_x^{(b)}\mathbf{u} \ge 0,\ \forall \mathbf{u}, \quad \mathbf{u}^TB \left({b} S_x\right)\mathbf{u} = u_{P+1}\left({b} S_x\mathbf{u}\right)_{P+1} - u_{1}\left({b} S_x\mathbf{u}\right)_{1},
  \end{align*}
  where $S_x\approx {\partial }/{\partial x}$ at the boundaries.
  \end{definition}
Note that $H_x$ defines a discrete inner product and norm through 
\begin{align}\label{eq:disc_scalar_product_H_x}
\langle\mathbf{u} , \mathbf{v} \rangle_{{H}_x} = \mathbf{v}^T {H}_x \mathbf{u} , \quad \|\mathbf{u} \|_{{H}_x }^2 = \langle\mathbf{u} , \mathbf{u} \rangle_{{H}_x} > 0, \quad \forall \mathbf{u}  \ne 0.
\end{align}
From Definition \ref{Def:SBP_Operator}, we have
\begin{align}\label{eq:SBP_dx}
\langle D_x\mathbf{u}, \mathbf{u} \rangle_{{H}_x}  = -\langle \mathbf{u} , D_x\mathbf{u} \rangle_{{H}_x}  + \underbrace{\mathbf{u}^TB\mathbf{u}}_{u_{P+1}^2 - u_{1}^2} ,
\end{align}
\begin{align}\label{eq:SBP_dxx}
\langle D_{xx}^{(b)}\mathbf{u} , \mathbf{u} \rangle_{{H}_x}   = -\mathbf{u}^TM_x^{(b)}\mathbf{u}  + \underbrace{\mathbf{u}^TB\left( {b} S_x\mathbf{u}\right)}_{u_{P+1}\left({b} S_x\mathbf{u}\right)_{P+1} - u_{1}\left({b} S_x\mathbf{u}\right)_{1}}.
\end{align}
  Note the close similarities between the continuous integration-by-parts properties \eqref{eq:IBP_dx}--\eqref{eq:IBP_dxx} and their discrete analogue \eqref{eq:SBP_dx}--\eqref{eq:SBP_dxx}.
   \begin{remark} 
Traditional finite difference SBP operators \cite{BStrand1994, MattssonNordstrom2004, Mattsson2012} are designed with central difference operators of  even order ($2r$-th, $r = 1, 2, ...$) accuracy  in the interior and lower order ($r$-th) accurate  one-sided  operator close to the boundaries.
\end{remark}
When solving problems with both the first and and second derivatives present, certain compatibility conditions between $D_x$ and $D_{xx}^{(b)}$ are important to derive a stable discretization. Following \cite{DuruVirta2014, MattssonParisi2010}, we introduce the definition of fully compatible SBP operators. 
 \begin{definition}\label{Def:Compatible_SBP_Operator}
  Let $D_x$ and $D_{xx}^{(b)}$  denote SBP operators approximating ${\partial }/{\partial x}$ and ${\partial }/{\partial x}\left(b(x){\partial }/{\partial x}\right)$, respectively. The operators are called {\it fully compatible SBP operators} if
  \begin{align*}
  D_{xx}^{(b)} = H_x^{-1}\left(- M_x^{(b)} + B D_x\right),
  \end{align*}
  where
  \begin{align*}
  M_x^{(b)} = D_x^T H bD_x + R_x^{(b)}, \quad   \mathbf{u}^T R_x^{(b)}\mathbf{u} \ge 0,\ \forall \mathbf{u}.
  \end{align*}
  \end{definition}
     Fully compatible SBP operators enable the design of accurate and provably stable multi-block numerical approximation for problems involving mixed spatial derivatives such as the acoustic and elastic wave equations in complex geometries  \cite{VirtaMattsson2014, DuruVirta2014, DuruKreissandMattsson2014}. For these models with mixed spatial derivatives, the remainder operator $R_x^{(b)} $ enhances numerical accuracy and eliminates poisonous spurious numerical high frequency modes. However, for problems such as the shifted wave equation, where mixed spatial and temporal derivatives are present, the fully compatible properties of first and second derivative operators are not sufficient to guarantee numerical stability for all well-posed coefficients. In fact the opposite is the case, as the remainder operator $R_x^{(b)}$ generates high frequency exponentially growing numerical modes which can destroy the accuracy of the numerical solution \cite{MattssonParisi2010,SzilagyiKreissWinicour2005}. For these problems, artificial numerical dissipation designed to eliminate the unstable numerical mode is necessary,  see for example \cite{MattssonParisi2010,SzilagyiKreissWinicour2005}. Artificial numerical dissipation helps in many ways but it can introduce some unwanted numerical artefacts that are not present in the continuous model. 
  
  We will now introduce a strict compatibility condition that will enable accurate and stable numerical treatment of well-posed second order hyperbolic initial-boundary-value-problems (IBVPs) with mixed temporal and spatial derivatives.
  \begin{definition}
Let $D_x$and $D_{xx}^{(b)}$  denote SBP operators approximating ${\partial }/{\partial x}$, ${\partial }/{\partial x}\left(b(x){\partial }/{\partial x}\right)$. The operators are called {\bf ultra-compatible SBP operators} if they are fully compatible and $R_x^{(b)} \equiv 0$.
  \end{definition}
  A straightforward approach to  derive ultra-compatible SBP operators is to use the first derivative twice in order to construct the second derivative SBP operator. However, for traditional SBP finite difference operators this approach leads to one order loss of accuracy near boundaries and destroys the accuracy of the solutions. In the present study, we will use a spectral approach to derive ultra-compatible SBP operators with full accuracy.
  
\subsection{Spectral difference operators}
 Our desired goal is to construct a multiple element approximation of the shifted wave equation with spectral accuracy. 
To begin, we  discretise the domain $x \in \Omega = [B_{-}, B_{+}]$ into $K$ elements denoting the $k$-th element by $\Omega_k = [x_k, x_{k+1}]$, where $k = 1, 2, \dots, K$, with $x_1 = B_{-}$ and $x_{K+1} = B_{+}$.
Next, we map the element $\Omega_k$ to a reference element $\xi \in \widetilde{\Omega} =  [-1, 1]$ by the linear transformation 
\begin{align}\label{eq:transf_0}
x(\xi) = x_k + \frac{\Delta{x}_k}{2}\left(1 + \xi \right), \quad \Delta{x}_k = x_{k+1} - x_k, \quad \xi \in \widetilde{\Omega} =  [-1, 1].
\end{align}
In the reference element $\widetilde{\Omega} =  [-1, 1]$, the $L^2$-scalar product and norm are given by
\begin{align}\label{eq:l2norm_ref}
 \left(u, v\right)= \int_{-1}^{1} uvd\xi, \quad \|u\|^2 = \left(u, u\right).
\end{align}
Note that for the reference element $\widetilde{\Omega}$, we have omitted the subscript $\widetilde{\Omega}$ in the scalar product.

In order to derive a discrete approximation of the derivative, we will use a spectral approach. 
Let $\mathbb{P}^{P}$ denote the space of polynomials of degree at most $P$.
Now consider the polynomial approximation of $u(\xi)$
\begin{align}\label{eq:polynomial_approximation}
u(\xi) \approx U(\xi) = \sum_{j = 1}^{P+1} \phi_{j}(\xi) U_j, \quad \phi_{j}(\xi) \in \mathbb{P}^{P},
\end{align}
where $U_j$ are degrees of freedom to be determined and $\phi_{j}(\xi) $ are the polynomial basis spanning $\mathbb{P}^{P}$.
For an arbitrary test function $\phi \in \mathbb{P}^{P}$, consider the weak derivative of $U(\xi)$ defined by
\begin{align}\label{eq:weak_derivative}
\left( \phi(\xi), \frac{d U(\xi)}{d \xi}\right) = \sum_{j = 1}^{P+1}U_j\left( \phi(\xi),   \phi_{j}^{\prime}(\xi) \right).
\end{align}
We consider nodal polynomial basis $\phi_i(\xi)$, and in particular Lagrange polynomials of degree $P$  with 
\begin{equation*}
\phi_i(\xi_j) = \begin{cases}
1, \quad i=j,\\
0, \quad i\neq j,
\end{cases}
\end{equation*}
where $\xi_j$ is the $j$-th node of a Gauss-type quadrature rule and
$U_j = u(\xi_j)$. Next, we perform the classical Galerkin approximation by choosing the test functions as the basis functions $\phi(\xi) = \phi_i(\xi)\in \mathbb{P}^{P}$, $i = 1, 2, \cdots P+1$, and
we choose a quadrature rule that is exact  for all polynomial integrand of degree at most $2P-1$,
\begin{align}
\left( \phi_i(\xi), \frac{d U(\xi)}{d \xi}\right) = \sum_{j = 1}^{P+1}U_j\left( \phi_i(\xi),   \phi_{j}^{\prime}(\xi) \right) = \sum_{j = 1}^{P+1} Q_{ij}U_j = Q \mathbf{U}, \quad 
\mathbf{U} 
=
\begin{pmatrix}
U_1\\
U_2\\
\vdots\\
U_{P+1}
\end{pmatrix},
\end{align}
where
\begin{equation}\label{eq:stiffnes_matrix}
 Q_{ij} = \sum_{m = 1}^{P+1} w_m  \phi_i(\xi_m)   \phi_j^{\prime} (\xi_m) = \left( \phi_i(\xi),   \phi_{j}^{\prime}(\xi) \right) = - \left(\phi_j(\xi), \phi_{i}^{\prime}(\xi)  \right)+ \phi_i(1)  \phi_j(1) -  \phi_i(-1)  \phi_j(-1) .
\end{equation}
Here, $\xi_j$ are nodes of a Gauss-type quadrature  with weights $w_j>0$ for $j=1,2,\cdots,P+1$.
Note that 
$$
 Q_{ij} +  Q_{ij}^T = B_{ij}, \quad Q_{ij}^T =  \left(\phi_j(\xi), \phi_{i}^{\prime}(\xi)  \right), \quad B_{ij} = \phi_i(1)  \phi_j(1) -  \phi_i(-1)  \phi_j(-1) ,
$$
and
\begin{align}\label{eq:derivative}
\begin{pmatrix}
\frac{d u(\xi_1)}{d \xi}\\
\frac{d u(\xi_2)}{d \xi}\\
\vdots\\
\frac{d u(\xi_{P+1})}{d \xi}
\end{pmatrix}
\approx
\begin{pmatrix}
\frac{d U(\xi_1)}{d \xi}\\
\frac{d U(\xi_2)}{d \xi}\\
\vdots\\
\frac{d U(\xi_{P+1})}{d \xi}
\end{pmatrix}
= 
D\mathbf{U},
\quad
D = W^{-1}Q, \quad W =   \mathrm{diag}\left([w_1, w_2, \cdots, w_{P+1}]\right).
\end{align}

 We consider the Gauss-Legendre-Lobatto (GLL) quadrature  which is  exact for polynomials up to degree $2P-1$.  We note that GLL nodes include both endpoints $\xi_1 = -1$, $\xi_{P+1} = 1$, with $B = \mathrm{diag}\left([-1, 0, 0,\cdots 0, 1]\right)$, and will not  require projections/interpolations to evaluate numerical fluxes when imposing interface and boundary conditions.
As in \eqref{eq:disc_scalar_product_H_x}, we note that $W$ defines a discrete inner product and norm
\begin{align}\label{eq:disc_scalar_product_H}
\langle\mathbf{u} , \mathbf{v} \rangle_{{W}} = \mathbf{v}^T {W} \mathbf{u} , \quad \|\mathbf{u} \|_{{W} }^2 = \langle\mathbf{u} , \mathbf{u} \rangle_{{W}} > 0, \quad \forall \mathbf{u}  \ne 0.
\end{align}
The discrete norm $\|\mathbf{u} \|_{{W} }$ is equivalent to the continuous norm $\|{u}\|$ (see \cite{Canuto_eta2006}, after (5.3.2)),  that is for all ${u}\in \mathbb{P}^P$ we have
\begin{align}\label{eq:norm_equiv}
\|{u}\| \le \|\mathbf{u}\|_W \le \sqrt{3}\|{u}\|.
\end{align}
The following Lemma establishes the accuracy of the spectral difference operator $D = H^{-1}Q$.
\begin{lemma}
Consider the spectral difference operator $D = W^{-1}Q$ in \eqref{eq:derivative} in a reference element $\widetilde{\Omega}=[-1, 1]$ . 
The norm of the approximation error converges to zero spectrally
fast as \cite{Canuto_eta2006}  (5.4.33)--(5.4.34).
More precisely, 
\begin{equation}\label{err_D}
\left\|D^l\mathbf{U} -\frac{d^lu}{d\xi^l}\right\| \le C P^{l-m}  |u|_{H^{m;P}\left(\widetilde{\Omega}\right)}, \quad |u|_{H^{m;P}\left(\widetilde{\Omega}\right)}^2  = \sum_{n = \min\left(m, P+1\right)}^{m} \left\| \frac{d^n u}{d \xi^n}\right\|^2,
\end{equation}
where $l \ge 0$ and $m\ge l$.
\end{lemma}

 \subsection{Spectral difference SBP operators}
 In a physical element, as opposed to a reference element, introduce the matrices 
 $ H_x, D_x \in \mathbb{R}^{(P+1)\times(P+1)}$  defined by
\begin{equation}\label{sdoD}
D_x = H_x^{-1} Q \approx \frac{\partial}{\partial x}, \quad H_x = \frac{\Delta{x}}{2} W,
\end{equation}
where $Q$ is elemental stiffness matrix defined in \eqref{eq:stiffnes_matrix} and $H$ is the diagonal matrix containing the quadrature weights $w_j >0$. The discrete first derivative operator $D_x$  is a  spectral difference approximation of the first derivative in one space dimension and satisfies the SBP property, that is 
\begin{equation}\label{SBP_1st}
Q + Q^T = B = \mathrm{diag}\left([-1, 0, 0,\cdots 0, 1]\right), \quad  H_x = H_x^T > 0.
\end{equation}

We derive an approximation of the second derivative by using the first derivative twice, that is 
\begin{equation}\label{sdoDD}
D_{xx}^{(b)}= D_x\mathbf{b}D_x \approx \frac{\partial }{\partial x}\left(b(x)\frac{\partial }{\partial x}\right),
\end{equation}
where $\mathbf{b}$ is a diagonal matrix with elements $b_{ii}:=b(x_i)$.
Note that
\begin{equation}\label{SBP_2nd}
D_{xx}^{(b)} = H_x^{-1} \left( -D_x^TH_x\mathbf{b}D_x + B \mathbf{b}D_x\right).
\end{equation}
We will now make the discussion more formal. 
\begin{lemma}\label{lem:SBP_Ultra}
Let $D_x$ and $D_{xx}^{(b)}$ denote the discrete approximations of ${\partial }/{\partial x}$ and ${\partial }/{\partial x}\left(b(x){\partial }/{\partial x}\right)$ defined in \eqref{sdoD} and \eqref{sdoDD}, respectively. The operators  $D_x$ and $D_{xx}^{(b)}$  are { ultra-compatible SBP operators}.   \end{lemma}
  \begin{proof}
  The proof of Lemma \ref{lem:SBP_Ultra} follows from \eqref{SBP_1st}  and \eqref{SBP_2nd}.
  \end{proof}
  
  \subsection{Accuracy of the spectral difference SBP operators}
We present the accuracy property of the spectral difference operators in the following lemma.
\begin{lemma}
Consider the spectral difference operator $D_x$ in \eqref{sdoD} and $D_{xx}^{(b)}$  in \eqref{sdoDD}. In a single element $\Omega_k=[x_k, x_{k+1}]$ with length $\Delta x_k =  x_{k+1}- x_{k}$, the truncation error of the first derivative approximation $D_x$ is $\mathcal{O}({\Delta x_k^P})$, and  the truncation error of the second derivative approximation $D_{xx}^{(b)}$ is $\mathcal{O}({\Delta x_k^{P-1}})$. More precisely, 
\begin{equation}\label{err_Dx}
D_x\mathbf{U} \Big|_{j}= \frac{\partial u}{\partial x}\Big|_{x=x_k^{(j)}}+ C_1 {\Delta x_k^P} \left|\frac{\partial^{P+1}u}{\partial x^{P+1}}(z_1)\right|,\quad D_{xx}^{(b)}\mathbf{U} \Big|_{j}=  \frac{\partial }{\partial x}\left(b(x)\frac{\partial u}{\partial x}\right) \Big|_{x=x_k^{(j)}}+ C_2 {\Delta x_k^{P-1}} \left|\frac{\partial ^{P+1}u}{\partial x^{P+1}}(z_2)\right|,  
\end{equation}
where $x_k^{(j)} =  x_k + \frac{\Delta{x}_k}{2}\left(1 + \xi_j \right),\ j = 1, 2, \cdots P+1$  are the quadrature nodes, $z_1,z_2\in\Omega_k$, and $\mathbf{U}$ is any smooth function $u(x)$ evaluated on the quadrature nodes, i.e. $\mathbf{U}_{j}:=u(x_k^{(j)})$. The constant $C_1$ and $C_2$ are independent of $\Delta x_k$. 

\end{lemma}
\begin{proof}
The proof can easily be adapted from the error bound for the Lagrange interpolation, see Theorem 3 in \cite{Howell1991}.
\end{proof}

\section{The shifted wave equation}\label{sec:shifted_wave_equation}
We consider the  (1+1)-dimensions shifted wave equation
\begin{align}\label{1deqn}
\frac{\partial }{\partial t}\left(\frac{\partial u}{\partial t} - a\frac{\partial u }{\partial x} \right) - \frac{\partial }{\partial x}\left(a\left(\frac{\partial u}{\partial t} - a\frac{\partial u}{\partial x} \right) + b\frac{\partial u}{\partial x} \right) = 0, 
\end{align}
where $u$ is the unknown field, $a$ and $b>0$ are smooth real-valued functions in the spatial domain $x\in\Omega \subset \mathbb{R}$. As will be seen, the parameter $c=b-a^2$ plays an important role in the well-posedness of \eqref{1deqn} and the construction of numerical methods. 

\subsection{Conservation principle and energy estimate}
The PDE \eqref{1deqn} is in the conservative form and satisfies a conservation principle. To see this, we introduce the  flux function $F(u)$ given by
\begin{align}\label{pde_flux}
F(u)=a\left(\frac{\partial u}{\partial t} - a\frac{\partial u}{\partial x} \right) + b\frac{\partial u}{\partial x}.
\end{align}
We can then write \eqref{1deqn} as
\begin{align}\label{eq:1deqn_flux}
\frac{\partial }{\partial t}\left(\frac{\partial u}{\partial t} - a\frac{\partial u }{\partial x} \right) - \frac{\partial F(u)}{\partial x} = 0.
\end{align}
The PDE flux $F(u)$ has two important components, $a\left({\partial u}/{\partial t} - a{\partial u}/{\partial x} \right)$ the contribution from the advective transport when $ a \ne 0$, and $b{\partial u}/{\partial x}$ the contribution from the expanding pressure wave.

Multiplying \eqref{eq:1deqn_flux} by a smooth test function $\phi$ and integrating over the domain $\Omega$, we have
\begin{align}
\left( \phi, \frac{\partial }{\partial t}\left(\frac{\partial u}{\partial t} - a\frac{\partial u }{\partial x} \right)\right)_{\Omega} - \left(\phi, \frac{\partial F(u)}{\partial x}\right)_{\Omega} = 0.
\end{align}
Integration by parts yields
\begin{align}\label{eq:integrate_by_parts}
\left( \phi, \frac{\partial }{\partial t}\left(\frac{\partial u}{\partial t} - a\frac{\partial u }{\partial x} \right)\right)_{\Omega} + \left(\frac{\partial \phi}{\partial x}, F(u)\right)_{\Omega} - \phi F(u)\Big|_B  = 0.
\end{align}
Substituting a particular test function $\phi=1$ in \eqref{eq:integrate_by_parts}, we obtain
\begin{align}\label{eq:conservative_principle}
\int_{\Omega}\frac{\partial }{\partial t} \left(\frac{\partial u}{\partial t} - a\frac{\partial u }{\partial x}\right)dx =  F(u)\Big|_B,
\end{align}
which leads to the following theorem for the conservation principle. 
\begin{theorem}\label{theo:conservative_principle}
The PDE \eqref{eq:1deqn_flux} satisfies the conservation principle
\begin{align*}
\frac{d}{dt}\left(1, \left(\frac{\partial u}{\partial t} - a\frac{\partial u }{\partial x}\right)\right)_{\Omega} = \frac{d}{dt}\int_{\Omega}\left(\frac{\partial u}{\partial t} - a\frac{\partial u }{\partial x}\right)dx = 0,
\end{align*}
 if $F(u)\Big|_{B } = 0$.
\end{theorem}
Theorem \ref{theo:conservative_principle} holds for any periodic data $u(x,t)$, or a Cauchy data $u(x,t)$ with compact support in $\Omega$.

If a numerical method satisfies a discrete equivalence of Theorem \ref{theo:conservative_principle}, we say that the method is conservative.
This will be useful in preserving constraints imposed by the PDE. Although we consider linear problems in this paper, a conservative scheme will be important in proving the convergence of the numerical method for nonlinear problems with weak solutions \cite{LaxWendroff1960}.

Next, we derive an energy estimate for the continuous problem \eqref{1deqn} and identify boundary and interface conditions that lead to a well-posed problem. 

We introduce the constant 
\begin{align}\label{eq:constant}
C_{a,b} = \max_x \left(\left| \frac{\partial a}{\partial x} \right| + \left| \frac{\partial a}{\partial x} - \frac{a}{b}\frac{\partial b}{\partial x}  \right|\right),
\end{align}
and define the energy
\begin{align}\label{eq:energy_continuous}
 E_b(t) =    \left(\left(\frac{\partial u}{\partial t} - a\frac{\partial u }{\partial x} \right) , \left(\frac{\partial u}{\partial t} - a\frac{\partial u }{\partial x} \right) \right)_{\Omega} 
          + \left(\frac{\partial u }{\partial x}, \frac{\partial u }{\partial x}\right)_{b\Omega}.
\end{align}
Note that for constants $a$ and $b$ we have $C_{a,b} =0$.
The energy $ E_b(t) \geq 0$ given in \eqref{eq:energy_continuous} defines a semi-norm.
We have
\begin{theorem}\label{theo:energy_estimate_continuous}
The PDE \eqref{eq:1deqn_flux}, with $a, b \in \mathbb{R}$ and $b>0$, satisfies the following estimate for the energy change rate 
\begin{align}\label{eq:energy_estimate_continuous}
\frac{d}{dt} E_b(t)  \le C_{ab} E_b(t)+ \mathrm{BT}_s(t),
\end{align}
where the continuous energy $E_b(t)$ is defined  in \eqref{eq:energy_continuous}, and the boundary term defined by
 $$
 \mathrm{BT}_s(t)=\left[\left(\frac{\partial u}{\partial t} - a\frac{\partial u }{\partial x} \right)\left(a\left(\frac{\partial u}{\partial t} - a\frac{\partial u }{\partial x} \right)+b\frac{\partial u}{\partial x}\right)\right]\Big|_B+b\frac{\partial u }{\partial x}\frac{\partial u }{\partial t}\Big|_B.
 $$
\end{theorem}

\begin{proof}
Consider \eqref{eq:integrate_by_parts} and set $\phi = \left({\partial u}/{\partial t} - a{\partial u }/{\partial x} \right)$, we obtain
\begin{align}\label{eq:integrate_by_parts_0}
\left( \left(\frac{\partial u}{\partial t} - a\frac{\partial u }{\partial x} \right), \frac{\partial }{\partial t}\left(\frac{\partial u}{\partial t} - a\frac{\partial u }{\partial x} \right)\right)_{\Omega} + \left(\frac{\partial }{\partial x}\left(\frac{\partial u}{\partial t} - a\frac{\partial u }{\partial x} \right), F(u)\right)_{\Omega} - \left(\frac{\partial u}{\partial t} - a\frac{\partial u }{\partial x} \right) F(u)\Big|_B  = 0,
\end{align}
where $F(u)$ is defined in  \eqref{pde_flux}.  Note that 
\begin{equation}\label{eq:integrate_by_parts_1}
\begin{split}
\left(\frac{\partial }{\partial x}\left(\frac{\partial u}{\partial t} - a\frac{\partial u }{\partial x} \right), F(u)\right)_{\Omega} &=  \left(\frac{\partial }{\partial x}\left(\frac{\partial u}{\partial t} - a\frac{\partial u }{\partial x} \right), a\left(\frac{\partial u}{\partial t} - a\frac{\partial u }{\partial x} \right)\right)_{\Omega}  \\
&+ \left(\frac{\partial }{\partial x}\left(\frac{\partial u}{\partial t}\right), b\frac{\partial u }{\partial x} \right)_{\Omega} - \left(\frac{\partial }{\partial x}\left(a\frac{\partial u }{\partial x} \right), b\frac{\partial u }{\partial x} \right)_{\Omega}.
\end{split}
\end{equation}
Using
\begin{align*}
 \left(\frac{\partial }{\partial x}\left(\frac{\partial u}{\partial t} - a\frac{\partial u }{\partial x} \right), a\left(\frac{\partial u}{\partial t} - a\frac{\partial u }{\partial x} \right)\right)_{\Omega} =  
  \frac{1}{2} \left(\frac{\partial a}{\partial x}\left(\frac{\partial u}{\partial t} - a\frac{\partial u}{\partial x} \right), \left(\frac{\partial u}{\partial t} - a\frac{\partial u}{\partial x} \right)\right)_{\Omega} + \frac{1}{2}a\left(\frac{\partial u}{\partial t} - a\frac{\partial u}{\partial x} \right)^2 \Big|_B
\end{align*}
and
\begin{align*}
\left(\frac{\partial }{\partial x}\left( a\frac{\partial u}{\partial x} \right),  \left(b\frac{\partial u}{\partial x} \right)\right)_{\Omega} = \frac{1}{2} \left( \left(\frac{\partial a}{\partial x}-\frac{a}{b}\frac{\partial b}{\partial x}  \right)\frac{\partial u }{\partial x}, b\frac{\partial u }{\partial x}\right)_{\Omega} + \frac{1}{2}ab\left(\frac{\partial u}{\partial x}\right)^2 \Big|_B,
\end{align*}
in \eqref{eq:integrate_by_parts_1} we have
\begin{equation}\label{eq:integrate_by_parts_4}
\begin{split}
&\left(\frac{\partial }{\partial x}\left(\frac{\partial u}{\partial t} - a\frac{\partial u }{\partial x} \right), F(u)\right)_{\Omega} =  \left(\frac{\partial }{\partial x}\left(\frac{\partial u}{\partial t}\right), b\frac{\partial u }{\partial x} \right)_{\Omega} + \frac{1}{2} \left( \left(\frac{\partial a}{\partial x}-\frac{a}{b}\frac{\partial b}{\partial x}  \right)\frac{\partial u }{\partial x}, b\frac{\partial u }{\partial x}\right)_{\Omega} \\
+&  \frac{1}{2} \left(\frac{\partial a}{\partial x}\left(\frac{\partial u}{\partial t} - a\frac{\partial u}{\partial x} \right), \left(\frac{\partial u}{\partial t} - a\frac{\partial u}{\partial x} \right)\right)_{\Omega} + \frac{1}{2}a\left(\frac{\partial u}{\partial t} - a\frac{\partial u}{\partial x} \right)^2 \Big|_B +  \frac{1}{2}ab\left(\frac{\partial u}{\partial x}\right)^2 \Big|_B.
\end{split}
\end{equation}
Using \eqref{eq:integrate_by_parts_4} in \eqref{eq:integrate_by_parts_0} and adding the transpose of the product gives
\begin{equation}\label{eq:energy_continuous_derivative}
 \begin{split}
 &\frac{d}{dt}\left( \left(\left(\frac{\partial u}{\partial t} - a\frac{\partial u }{\partial x} \right) , \left(\frac{\partial u}{\partial t} - a\frac{\partial u }{\partial x} \right) \right)_{\Omega} 
          + \left(\frac{\partial u }{\partial x}, \frac{\partial u }{\partial x}\right)_{b\Omega}\right) =   \left(\frac{\partial a}{\partial x}\left(\frac{\partial u}{\partial t} - a\frac{\partial u}{\partial x} \right), \left(\frac{\partial u}{\partial t} - a\frac{\partial u}{\partial x} \right)\right)_{\Omega} \\
          &+\left( \left(\frac{\partial a}{\partial x}-\frac{a}{b}\frac{\partial b}{\partial x}  \right)\frac{\partial u }{\partial x}, b\frac{\partial u }{\partial x}\right)_{\Omega}+ a\left(\frac{\partial u}{\partial t} - a\frac{\partial u}{\partial x} \right)^2 \Big|_B +  ab\left(\frac{\partial u}{\partial x}\right)^2 \Big|_B +2\left(\frac{\partial u}{\partial t} - a\frac{\partial u }{\partial x} \right) F(u)\Big|_B.
 \end{split}
\end{equation}
On the left hand side of \eqref{eq:energy_continuous_derivative} we recognise the time derivative of the energy. Introducing $\mathrm{BT}_s$ for the boundary terms and using the Cauchy-Schwartz inequality on the right hand side of \eqref{eq:energy_continuous_derivative} gives 
\begin{align}\label{ce5}
\frac{d}{dt} E_b(t)   \leq  C_{a,b} E_b(t) + \mathrm{BT}_s(t),
\end{align}
where the constant $C_{a,b}$ depends on the material property 
\begin{align*}
C_{a,b} = \max_x \left(\left| \frac{\partial a}{\partial x} \right| + \left| \frac{\partial a}{\partial x} - \frac{a}{b}\frac{\partial b}{\partial x}  \right|\right).
\end{align*}
\end{proof}
Theorem \ref{theo:energy_estimate_continuous} holds for  \eqref{eq:1deqn_flux}, with $a, b \in \mathbb{R}$ and $b>0$.
We note however, if in particular {$c =b-a^2 > 0$}, from \eqref{1deqn} we have
\begin{align}\label{1deqn_c0}
\frac{\partial^2 u}{\partial t^2} - \frac{\partial }{\partial t}\left(a\frac{\partial u }{\partial x} \right) - \frac{\partial }{\partial x}\left(a\frac{\partial u }{\partial t} \right)  - \frac{\partial }{\partial x}\left(c\frac{\partial u}{\partial x} \right) = 0. 
\end{align}
We introduce the energy 
\begin{align}\label{eq:energy_continuous_c0}
 E_c(t) =    \left(\frac{\partial u}{\partial t}  , \frac{\partial u}{\partial t} \right)_{\Omega} + \left(\frac{\partial u }{\partial x}, \frac{\partial u }{\partial x}\right)_{c\Omega}.
\end{align}
The energy $ E_c(t)$ defines a space-time weighted $H^{1}(\Omega)$ norm.
\begin{theorem}\label{theo:energy_estimate_continuous_c0}
The PDE \eqref{eq:1deqn_flux}, with $a, b \in \mathbb{R}$, $b>0$ and {$c=b-a^2 > 0$}, satisfies the following estimate for the energy change rate 
\begin{align}\label{ce5_c0}
\frac{d}{dt} E_c(t)    =    2\frac{\partial u}{\partial t}\left(a \frac{\partial u}{\partial t} + c\frac{\partial u}{\partial x}\right)\Big|_{B}.
\end{align}
\end{theorem}
\begin{proof}
Consider 
$$
\left(\phi, \frac{\partial^2 u}{\partial t^2} \right)_{\Omega}- \left(\phi, \frac{\partial }{\partial t}\left(a\frac{\partial u }{\partial x} \right)\right)_{\Omega} - \left(\phi, \frac{\partial }{\partial x}\left(a\frac{\partial u }{\partial t} \right)\right)_{\Omega}  - \left(\phi, \frac{\partial }{\partial x}\left(c\frac{\partial u}{\partial x} \right)\right)_{\Omega} = 0. 
$$
Using the integration-by-parts and setting $\phi = {\partial u}/{\partial t} $, we obtain
\begin{align}\label{eq:integrate_by_parts_0_c0}
\left(\frac{\partial u}{\partial t}, \frac{\partial^2 u}{\partial t^2} \right)_{\Omega}  + \left(\frac{\partial^2 u}{\partial x\partial t}, \frac{\partial u}{\partial x} \right)_{c\Omega} = \frac{\partial u}{\partial t}\left(a \frac{\partial u}{\partial t} + c\frac{\partial u}{\partial x}\right)\Big|_{B}. 
\end{align}
 Adding the transpose of the product gives
\begin{align}\label{ce5_c00}
\frac{d}{dt} E_c(t)    =    2\frac{\partial u}{\partial t}\left(a \frac{\partial u}{\partial t} + c\frac{\partial u}{\partial x}\right)\Big|_{B}.
\end{align}
\end{proof}
 For a Cauchy problem or periodic boundary conditions, we have $\mathrm{BT}_s(t) = 0$ and 
 \begin{align}\label{ce5_0}
\frac{d}{dt} E_b(t) \leq  C_{a,b} E_b(t), \qquad \frac{d}{dt} E_c(t)  =0\iff E_c(t) = E_c(0).
\end{align}
 Note that $C_{a,b}=0$ for constants $a, b$ and ${dE_b(t)}/{dt}  \leq  0\iff E_b(t) = E_b(0)$.

 For IBVPs, in order to obtain a continuous energy estimate, the physical boundary or interface conditions must be such that the boundary contribution in the energy change rate \eqref{ce5} is negative semi-definite $\mathrm{BT}_s(t) \le 0$.
\begin{remark}
The boundary contribution terms $\mathrm{BT}_s(t)$ in \eqref{eq:energy_estimate_continuous}, 
\begin{equation}\label{BC}
 \mathrm{BT}_s(t) = \left(\frac{\partial u}{\partial t} - a\frac{\partial u }{\partial x} \right)\left(a\left(\frac{\partial u}{\partial t} - a\frac{\partial u }{\partial x} \right)+b\frac{\partial u}{\partial x}\right)\Big|_B+b\frac{\partial u }{\partial x}\frac{\partial u }{\partial t}\Big|_B,
\end{equation}
can be equivalently written as
\begin{equation}\label{BC_lambda}
\mathrm{BT}_s(t) = \frac{\lambda_1}{2}\left(\frac{\partial u}{\partial t} - \lambda_2\frac{\partial u}{\partial x} \right)^2 \Big|_B + \frac{\lambda_2}{2}\left(\frac{\partial u}{\partial t} - \lambda_1\frac{\partial u}{\partial x} \right)^2  \Big|_B,
\end{equation}
where 
\begin{equation}\label{lambda}
\lambda_1 = a + \sqrt{b}, \quad \lambda_2 = a - \sqrt{b}.
\end{equation}
We also have the relation $c=b-a^2=-\lambda_1\lambda_2$.
\end{remark}

\subsection{Well-posed interface conditions}
We will now split our domain into two $\Omega = \Omega_{-} \cup \Omega_{+}$ with an interface at $x_I\in\Omega$, where $\Omega_{-} =\left\{x| x\le x_I\right\}$  and $\Omega_{+} =\left\{x| x\ge x_I\right\}$.
Let $u^+$ denote the solution in the positive subdomain $x\in \Omega_{+}$ and  $u^-$ denote the solution in the negative subdomain $x\in \Omega_{-}$.
We have
\begin{align}\label{eq:subdomain_1}
\frac{\partial }{\partial t}\left(\frac{\partial u^-}{\partial t} - a\frac{\partial u^- }{\partial x} \right) - \frac{\partial }{\partial x}\left(a\left(\frac{\partial u^-}{\partial t} - a\frac{\partial u^-}{\partial x} \right) + b\frac{\partial u^-}{\partial x} \right) = 0, \quad x < x_I,
\end{align}
\begin{align}\label{eq:subdomain_2}
\frac{\partial }{\partial t}\left(\frac{\partial u^+}{\partial t} - a\frac{\partial u^+ }{\partial x} \right) - \frac{\partial }{\partial x}\left(a\left(\frac{\partial u^+}{\partial t} - a\frac{\partial u^+}{\partial x} \right) + b\frac{\partial u^+}{\partial x} \right) = 0,  \quad x > x_I.
\end{align}
At the interface $x_I\in\Omega$, we define the jump in $u$ as $\lJump u \rJump = u^+(x_I)-u^-(x_I)$, where the superscripts $+$ and $-$ denote the quantity on the right and left sides of the interface. 

Our primary objective here is to derive interface conditions that will be used to couple locally adjacent spectral elements together.
To do this, we multiply  equation \eqref{eq:subdomain_1}--\eqref{eq:subdomain_2} by a sufficiently smooth function $\phi$ that vanishes at the boundaries, $\phi(B_{\pm}) = 0$, and integrate over the domain, we obtain \eqref{eq:integrate_by_parts}. Collecting contributions from both sides of the interface gives
\begin{equation}\label{eq:two_domain_weak_form}
\begin{split}
&\left( \phi, \frac{\partial }{\partial t}\left(\frac{\partial u^{-}}{\partial t} - a\frac{\partial u^{-} }{\partial x} \right)\right)_{\Omega_{-}} + \left(\frac{\partial \phi}{\partial x}, F(u^{-})\right)_{\Omega_{-}} \\
 &+\left( \phi, \frac{\partial }{\partial t}\left(\frac{\partial u^{+}}{\partial t} - a\frac{\partial u^{+} }{\partial x} \right)\right)_{\Omega_{+}} + \left(\frac{\partial \phi}{\partial x}, F(u^{+})\right)_{\Omega_{+}}  \\
 & -\phi \lJump F(u) \rJump\Big|_{x =x_I}  = 0.
 \end{split}
\end{equation}

The conservation principle  \eqref{eq:conservative_principle} requires that the jump in the flux vanishes $\lJump F(u) \rJump = 0$. As we will see later this will be useful in designing  conservative numerical interface treatment. 

To obtain a continuous energy estimate, we need the boundary terms in \eqref{eq:energy_estimate_continuous} from both sides to cancel out at the interface. Assuming the coefficients $a$ and $b$ are continuous at the interface,  we have the following interface conditions 
\begin{align}\label{eq:interface_condition}
\lJump \frac{\partial u}{\partial t} - a\frac{\partial u }{\partial x}  \rJump = 0, \quad \text{and} \quad  \lJump F(u) \rJump = 0. 
\end{align}
Also note that for smooth (or constant) coefficients, the interface condition \eqref{eq:interface_condition} is equivalent to
\begin{align}\label{eq:interface_condition_proper}
\lJump u \rJump = 0 \implies \lJump \frac{\partial u }{\partial t} \rJump  = 0, \quad \text{and} \quad \lJump \frac{\partial u }{\partial x} \rJump = 0 .
\end{align}
\begin{theorem}\label{theo:stable_interface}
Consider the two-domain formulation, \eqref{eq:subdomain_1}--\eqref{eq:subdomain_2}, with the interface condition \eqref{eq:interface_condition} or \eqref{eq:interface_condition_proper}.
Let $E_b^{\pm}(t)$ denote the energy in $\Omega_{\pm}$. We have
\begin{align*}
\frac{d}{dt} E_b(t)  \le C_{ab} E_b(t) , \quad E_b(t) = E_b^-(t) + E_b^+(t).
\end{align*}
For constant coefficients $a,b$, we have $C_{ab} =0$ and the energy is conserved,
\begin{align*}
\frac{d}{dt} E_b(t)   = 0.
\end{align*}
\end{theorem}
\begin{proof}
Theorem \ref{theo:stable_interface} is equivalent to Theorem \ref{theo:energy_estimate_continuous}. The proof of Theorem \ref{theo:stable_interface} follows similar steps where the interface condition \eqref{eq:interface_condition} or \eqref{eq:interface_condition_proper} is used to eliminate the contributions from the interface.

That is, using \eqref{eq:energy_continuous_derivative} in each subdomain $\Omega_{\pm}$, we have
\begin{equation}\label{eq:energy_continuous_derivative_pm}
 \begin{split}
 &\frac{d}{dt}\left( \left(\left(\frac{\partial u}{\partial t} - a\frac{\partial u }{\partial x} \right) , \left(\frac{\partial u}{\partial t} - a\frac{\partial u }{\partial x} \right) \right)_{\Omega_{\pm}} 
          + \left(\frac{\partial u }{\partial x}, \frac{\partial u }{\partial x}\right)_{b\Omega_{\pm}}\right) =   \left(\frac{\partial a}{\partial x}\left(\frac{\partial u}{\partial t} - a\frac{\partial u}{\partial x} \right), \left(\frac{\partial u}{\partial t} - a\frac{\partial u}{\partial x} \right)\right)_{\Omega_{\pm}} \\
          &+\left( \left(\frac{\partial a}{\partial x}-\frac{a}{b}\frac{\partial b}{\partial x}  \right)\frac{\partial u }{\partial x}, b\frac{\partial u }{\partial x}\right)_{\Omega_{\pm}} \mp a\left(\frac{\partial u}{\partial t} - a\frac{\partial u}{\partial x} \right)^2 \Big|_{x_I} \mp  ab\left(\frac{\partial u}{\partial x}\right)^2 \Big|_{x_I} \mp 2\left(\frac{\partial u}{\partial t} - a\frac{\partial u }{\partial x} \right) F(u)\Big|_{x_I}.
 \end{split}
 \end{equation}
 Summing contributions from both sides of the interface and using the interface condition \eqref{eq:interface_condition} or \eqref{eq:interface_condition_proper}, the interface terms vanish, and we obtain
 \begin{equation}\label{eq:energy_continuous_derivative_pm_0}
 \begin{split}
 &\frac{d}{dt} E_b(t)=   \sum_{\eta = -, +}\left(\left(\frac{\partial a}{\partial x}\left(\frac{\partial u}{\partial t} - a\frac{\partial u}{\partial x} \right), \left(\frac{\partial u}{\partial t} - a\frac{\partial u}{\partial x} \right)\right)_{\Omega_{\eta}} 
          +\left( \left(\frac{\partial a}{\partial x}-\frac{a}{b}\frac{\partial b}{\partial x}  \right)\frac{\partial u }{\partial x}, b\frac{\partial u }{\partial x}\right)_{\Omega_{\eta}} \right).
 \end{split}
\end{equation}
By using the Cauchy-Schwarz inequality on the right hand side of \eqref{eq:energy_continuous_derivative_pm_0}, we have the desired estimate.
Note in particular with constant coefficients $a,b$, we have ${\partial a}/{\partial x} =0$, ${\partial b}/{\partial x} =0$ and the right hand side of \eqref{eq:energy_continuous_derivative_pm_0} vanishes identically.
\end{proof}
As above, Theorem \ref{theo:stable_interface} holds for  any $a, b \in \mathbb{R}$ with $b>0$.
We note however, if in particular {$c=b-a^2 > 0$},  we also have the following energy estimate. 
\begin{theorem}\label{theo:stable_interface_c0}
Consider the two-domain formulation, \eqref{eq:subdomain_1}--\eqref{eq:subdomain_2} with $a, b \in \mathbb{R}$, $b>0$ and {$c=b-a^2 > 0$}, 
with the interface condition \eqref{eq:interface_condition} or \eqref{eq:interface_condition_proper}.
 We have
\begin{align*}
\frac{d}{dt} E_c(t) =0, \quad E_c(t) = E_c^-(t) + E_c^+(t).
\end{align*}
\end{theorem}
\begin{proof}
Consider 
\begin{align}\label{eq:integrate_by_parts_0_c0_pm}
\left(\frac{\partial u}{\partial t}, \frac{\partial^2 u}{\partial t^2} \right)_{\Omega_{\pm}}  + \left(\frac{\partial^2 u}{\partial x\partial t}, \frac{\partial u}{\partial x} \right)_{c\Omega_{\pm}} = \frac{\partial u}{\partial t}\left(a \frac{\partial u}{\partial t} + c\frac{\partial u}{\partial x}\right)\Big|_{B}. 
\end{align}
 Add the transpose of the product and collect contributions from both sides of the interface. Enforcing the interface condition  \eqref{eq:interface_condition} or \eqref{eq:interface_condition_proper} gives
\begin{align*}
\frac{d}{dt} E_c(t)    =   0 .
\end{align*}
\end{proof}

\subsection{Well-posed boundary conditions}
We will now consider the bounded domain $\Omega = [B_{-}, B_{+}]$ and analyse well-posed boundary conditions.
When analysing boundary conditions, it is convenient to use the form \eqref{BC_lambda} for the boundary contribution. 
We are primarily interested in energy stable boundary conditions. That is, with homogeneous boundary data, the boundary conditions should be such that the boundary term is never positive, $\mathrm{BT}_s(t) \le 0$.  Expanding the boundary terms in  \eqref{BC_lambda}, we have
\begin{equation}\label{BC_lambda_0}
\begin{split}
\mathrm{BT}_s(t) &= \left(\frac{\lambda_1}{2}\left(\frac{\partial u}{\partial t} - \lambda_2\frac{\partial u}{\partial x} \right)^2 + \frac{\lambda_2}{2}\left(\frac{\partial u}{\partial t} - \lambda_1\frac{\partial u}{\partial x} \right)^2  \right)\Big|_{x =B_{+}}\\
&-\left(\frac{\lambda_1}{2}\left(\frac{\partial u}{\partial t} - \lambda_2\frac{\partial u}{\partial x} \right)^2 + \frac{\lambda_2}{2}\left(\frac{\partial u}{\partial t} - \lambda_1\frac{\partial u}{\partial x} \right)^2\right)\Big|_{x =B_{-}}
\end{split}
\end{equation}
Below we consider two different cases, depending on the sign of $c=b-a^2$.

\subsubsection{Case 1:  $ c=-\lambda_1\lambda_2  > 0$}
When $c> 0$, we have $\lambda_1 > 0$ and $\lambda_2 < 0$ for any $a \in \mathbb{R}$. Consequently, for $\mathrm{BT}_s(t) \leq 0$, we have the boundary conditions
\begin{align}\label{eq:BC_case1}
\frac{\partial u}{\partial t} - \lambda_1\frac{\partial u}{\partial x}  = g_1(t), \quad x = B_{-}, \qquad \frac{\partial u}{\partial t} - \lambda_2\frac{\partial u}{\partial x}  = g_2(t), \quad x = B_{+}.
\end{align}
\begin{theorem}\label{theo:energy_estimate_continuous_1}
The IVBP \eqref{1deqn} and \eqref{eq:BC_case1} satisfies the energy estimate
\begin{align*}
&\frac{d}{dt} E_b(t)  - C_{ab} E_b(t) + \mathrm{BT}_1 \le  \frac{\lambda_1}{2}  g^2_2(t)-\frac{\lambda_2}{2}  g^2_1(t), \\
&  \mathrm{BT}_1= \left(\frac{\lambda_1}{2}\left(\frac{\partial u}{\partial t} - \lambda_2\frac{\partial u}{\partial x} \right)^2\Big|_{B_{-}} -\frac{\lambda_2}{2}\left(\frac{\partial u}{\partial t} - \lambda_1\frac{\partial u}{\partial x} \right)^2\Big|_{B_{+}} \right)\ge 0.
\end{align*}
\end{theorem} 

Similarly, we also have the energy estimate with the discrete energy $E_c$. 
\begin{theorem}\label{theo:energy_estimate_continuous_1_c0}
The IVBP \eqref{1deqn} and \eqref{eq:BC_case1} satisfies the energy estimate
\begin{align*}
&\frac{d}{dt} E_c(t)  + \mathrm{BT}_1 =  \frac{\sqrt{b}}{2}g_1^2(t) + \frac{\sqrt{b}}{2}g_2^2(t), \\
&\mathrm{BT}_1= \frac{\sqrt{b}}{2}\left(\left(\frac{\partial u}{\partial t}(B_{-}, t)\right)^2 +\left(\frac{\partial u}{\partial t}(B_{-}, t)-g_1(t)\right)^2\right) + \frac{\sqrt{b}}{2}\left(\left(\frac{\partial u}{\partial t}(B_{+}, t)\right)^2 +\left(\frac{\partial u}{\partial t}(B_{+}, t)-g_2(t)\right)^2 \right).
\end{align*}
\end{theorem} 

Using the form \eqref{BC_lambda_0}, the proofs of Theorem \ref{theo:energy_estimate_continuous_1} and \ref{theo:energy_estimate_continuous_1_c0} follow the same steps as the proofs for Theorem \ref{theo:energy_estimate_continuous} and \ref{theo:energy_estimate_continuous_c0}, respectively.

\subsubsection{Case 2: $ c< 0$}
We consider $a > 0$ and $ c< 0$, which implies $\lambda_1 > \lambda_2 > 0$.  For $\mathrm{BT}_s(t) \leq 0$, we  need to impose two boundary conditions at $x = B_{+}$, and  no boundary condition at $x = B_{-}$.  The boundary conditions are
\begin{align}\label{eq:BC_case2}
\frac{\partial u}{\partial t} - \lambda_1\frac{\partial u}{\partial x}  = g_1(t), \quad x = B_{+}, \qquad \frac{\partial u}{\partial t} - \lambda_2\frac{\partial u}{\partial x}  = g_2(t), \quad x = B_{+},
\end{align}
or equivalently
\begin{align}\label{eq:BC_case2_0}
u = f_1(t), \quad x = B_{+}, \qquad \frac{\partial u}{\partial x}  = f_2(t), \quad x = B_{+},
\end{align}
where
$$
g_1(t) = \frac{d}{dt} f_1(t) - \lambda_1f_2(t), \quad g_2(t) = \frac{d}{dt}  f_1(t) - \lambda_2f_2(t).
$$
We have
\begin{theorem}\label{theo:energy_estimate_continuous_2}
The IVBP  \eqref{1deqn} and \eqref{eq:BC_case2} satisfies the energy estimate
\begin{align*}
&\frac{d}{dt} E_b(t)  - C_{ab} E_b(t) + \mathrm{BT}_2 \le G(t), \\
&G(t) =  \frac{\lambda_2}{2}g_1^2(t) + \frac{\lambda_1}{2}g_2^2(t) , \quad \mathrm{BT}_2=  \left(\frac{\lambda_1}{2}\left(\frac{\partial u}{\partial t} - \lambda_2\frac{\partial u}{\partial x} \right)^2 + \frac{\lambda_2}{2}\left(\frac{\partial u}{\partial t} - \lambda_1\frac{\partial u}{\partial x} \right)^2  \right)\Big|_{x=B_{-}} \ge 0.
\end{align*}
\end{theorem}

\begin{remark}
For the other case  $a < 0$ and $ c< 0$, the boundary conditions shall be reversed: two boundary conditions at $x = B_{-}$, and there is no boundary condition at $x = B_{+}$. That is
\begin{align}\label{eq:BC_case3}
\frac{\partial u}{\partial t} - \lambda_1\frac{\partial u}{\partial x}  = g_1(t), \quad x = B_{-}, \qquad \frac{\partial u}{\partial t} - \lambda_2\frac{\partial u}{\partial x}  = g_2(t), \quad x = B_{-},
\end{align}
or equivalently
\begin{align}\label{eq:BC_case3_0}
u = f_1(t), \quad x = B_{-}, \qquad \frac{\partial u}{\partial x}  = f_2(t), \quad x = B_{-},
\end{align}
where
$$
g_1(t) = \frac{d}{dt} f_1(t) - \lambda_1f_2(t), \quad g_2(t) = \frac{d}{dt}  f_1(t) - \lambda_2f_2(t).
$$
\end{remark}

The above theorems state relations between the energy change rate and the energy. To derive a bound for the energy itself, we apply Gronwall's lemma to obtain the following result. 

\begin{theorem}\label{theo:energy_estimate_continuous_2_other}
Consider  the continuous energy $E(t)>0$  and the energy estimate 
\begin{align}\label{eq:energy_estimate_continuous_00}
&\frac{d}{dt} E(t)  - C_{ab} E(t) + \mathrm{BT}(t) \le G(t), \quad \mathrm{BT}(t) \ge 0, \quad C_{ab} \ge 0.
\end{align}
We have
\begin{align}\label{eq:energy_estimate_continuous_01}
&\frac{d}{dt} E(t)  - C_{ab} E(t) +  \eta(t)E(t) \le G(t), \quad \eta(t) = \frac{\mathrm{BT}(t)}{E(t)} \ge \eta_0 > 0,
\end{align}
and
\begin{align}\label{eq:energy_estimate_continuous_02}
& E(t)  \le e^ {\gamma t} E(0) + \int_0^te^ {\gamma (t-\tau)}  G(\tau) d\tau, \quad \gamma = C_{ab} -\eta_0.
\end{align}
\end{theorem}
Theorem \ref{theo:energy_estimate_continuous_2_other} proves the strong well-posedness of the IBVP, where $E$ could either be $E_b$ or $E_c$.

\section{Discretisation in space}
 In this section we present a multi-element semi-discrete discontinuous spectral element approximation of the shifted wave equation in a bounded domain, and prove numerical stability and conservation.
 We will first present the numerical interface treatment and proceed later to numerical treatment of external boundaries.

\subsection{Numerical interface treatments}\label{sec:numerical_interface_treatment}
Here, we will derive a conservative and energy stable interface treatment for the shifted wave equation.
For simplicity we will focus on one interface shared by two spectral elements, but the method and analysis can be easily extended to more than two elements and multiple  interfaces. 
We consider the two elements model  $\Omega = \Omega_{-} \cup \Omega_{+}$ with an interface at $x_I$.
We map each element to a reference element, $\Omega_{\pm} \to \widetilde{\Omega} =[-1, 1]$.

Let $\mathbf{U} = \begin{pmatrix}  \mathbf{u}^-  \\  \mathbf{u}^+   \end{pmatrix}$ denote the degrees of freedom to be evolved. We introduce the weight matrix and the discrete derivative operator
$$
\mathbf{H} =  \begin{pmatrix} H_x & \mathbf{0} \\ \mathbf{0} & H_x \end{pmatrix}, \quad {\mathbf{D}} = \begin{pmatrix} D_x& \mathbf{0} \\ \mathbf{0} & D_x \end{pmatrix},
$$
and the discrete scalar product and the discrete norm defined by $\mathbf{H}$
\begin{align}\label{eq:disc_scalar_product}
\langle\mathbf{U} , \mathbf{V} \rangle_{\mathbf{H}} = \mathbf{V}^T \mathbf{H} \mathbf{U} , \quad \|\mathbf{U} \|_{\mathbf{H} }^2 = \langle\mathbf{U} , \mathbf{U} \rangle_{\mathbf{H}}.
\end{align}
The Galerkin spectral element approximation of the two-domain formulation, \eqref{eq:subdomain_1}--\eqref{eq:subdomain_2}  is
\begin{equation}
\begin{split}
        \label{eq:InterfaceSemiDiscrete_shift_no_penalty}
           \frac{{d}}{{dt}}\left( \frac{{d}}{{dt}} \mathbf{U}   - \mathbf{a}  {\mathbf{D}}\mathbf{U}  \right) &- \mathbf{D}\left(\mathbf{a}\left( \frac{{d}}{{dt}} \mathbf{U}   - \mathbf{a} {\mathbf{D}}\mathbf{U}  \right) + \mathbf{b} \mathbf{D}\mathbf{U} \right) = 0,
\end{split}
\end{equation}
where $\mathbf{a}$ and $\mathbf{b}$ are the coefficients $a(x)$ and $b(x)$ evaluated on the quadrature nodes. 
In \eqref{eq:InterfaceSemiDiscrete_shift_no_penalty}, note that we have only replaced the continuous derivative operators with spectral difference operators and we  are yet to implement the interface conditions  \eqref{eq:interface_condition} or \eqref{eq:interface_condition_proper}. The numerical solutions in the two elements are independent and  unconnected. The solutions will be connected across the element interface through a numerical flux. 
We introduce the interface matrices
\begin{align*}
  \mathbf{\widehat{B}}  = \begin{pmatrix} 
                         \er \er^T & -\er \el^T \\
                          \el \er^T  & -\el \el^T 
                         \end{pmatrix},
  \quad
\el =\left(1, 0, \cdots, 0, 0\right)^T, 
\quad 
\er =\left(0, 0, \cdots, 0, 1\right)^T.
\end{align*}
Note that
\begin{align*}
  \mathbf{\widehat{B}} \mathbf{U} 
  =
  \begin{pmatrix} 
  0\\
  \vdots\\
  0\\
  u_{P+1}^{-}- u_{1}^{+}\\
   u_{P+1}^{-}- u_{1}^{+}\\
    0\\
      \vdots\\
      0
  \end{pmatrix}.
  \end{align*}
  If the solution is continuous across the interface $\lJump u\rJump: = u_{1}^{+}- u_{P+1}^{-} = 0$, then we have $\mathbf{\widehat{B}} \mathbf{U} = \mathbf{0}$.

To begin, we introduce the penalised difference operator $\widetilde{\mathbf{D}} $ and the global  boundary operator  $\widetilde{\mathbf{B}}$ defined by 
\begin{align}\label{eq:penalise_differential}
\widetilde{\mathbf{D}} = {\mathbf{D}} - \gamma_0 \mathbf{H}^{-1} \mathbf{\widehat{B}}, \quad \widetilde{\mathbf{B}} =\begin{pmatrix} -\el \el^{T} & \mathbf{0} \\ \mathbf{0} & \er\er^{T} \end{pmatrix}.
\end{align}
If the solution is continuous across the interface, we have $\widetilde{\mathbf{D}}\mathbf{U} = {\mathbf{D}} \mathbf{U}$ for any $\gamma_0$.
The following lemma shows that the particular choice $\gamma_0=1/2$ makes $\widetilde{\mathbf{D}}$  anti-symmetric in the discrete scalar product defined by \eqref{eq:disc_scalar_product}.
\begin{lemma}\label{lem:anti_symmetry_D}
Consider the penalised difference operator $\widetilde{\mathbf{D}}$ defined in \eqref{eq:penalise_differential} and the grid functions $\mathbf{V}, \mathbf{U} \in\mathbb{R}^{2P+2}$. If $\gamma_0 = 1/2$ the modified operator $\widetilde{\mathbf{D}}$ satisfies the SBP property
\begin{align}
\langle\mathbf{V} , \widetilde{\mathbf{D}} \mathbf{U} \rangle_{\mathbf{H}} + \langle  \widetilde{\mathbf{D}}\mathbf{V}, \mathbf{U} \rangle_{\mathbf{H}} = \mathbf{V}^T\widetilde{\mathbf{B}}\mathbf{U} =u_{P+1}^{+}v_{P+1}^{+} - u_{1}^{-}v_{1}^{-} = u(B_+)v(B_+) - u(B_-)v(B_-),
\end{align}
and if in particular we ignore contributions from external boundaries at $x =B_-, B_+$, ($v(B_{\pm}) =0$) we have
\begin{align}
\langle\mathbf{V} , \widetilde{\mathbf{D}} \mathbf{U} \rangle_{\mathbf{H}} + \langle  \widetilde{\mathbf{D}}\mathbf{V}, \mathbf{U} \rangle_{\mathbf{H}} = 0.
\end{align}
\end{lemma}
\begin{proof}
Consider $\widetilde{\mathbf{D}}$ defined in \eqref{eq:penalise_differential} and simplify using the SBP property \eqref{sdoD}
\begin{equation}\label{eq:Discrete_operatorD}
\begin{split}
\widetilde{\mathbf{D}} &= \begin{pmatrix} D_x& \mathbf{0} \\ \mathbf{0} & D_x \end{pmatrix} - \frac{1}{2}\begin{pmatrix} H_x & \mathbf{0} \\ \mathbf{0} & H_x \end{pmatrix}^{-1} \mathbf{\widehat{B}} 
= \begin{pmatrix} H_x & \mathbf{0} \\ \mathbf{0} & H_x \end{pmatrix}^{-1}\left(\begin{pmatrix} -D_x^TH_x& \mathbf{0} \\ \mathbf{0} & -D_x^TH_x \end{pmatrix}+\begin{pmatrix} B& \mathbf{0} \\ \mathbf{0} & B \end{pmatrix} - \frac{1}{2} \mathbf{\widehat{B}}  \right)\\
&= \begin{pmatrix} H_x & \mathbf{0} \\ \mathbf{0} & H_x \end{pmatrix}^{-1}\left(\begin{pmatrix} -D_x^TH_x& \mathbf{0} \\ \mathbf{0} & -D_x^TH_x \end{pmatrix}+\begin{pmatrix} -\el \el^{T} & \mathbf{0} \\ \mathbf{0} & \er\er^{T} \end{pmatrix} + \frac{1}{2} \mathbf{\widehat{B}}^T  \right).
\end{split}
\end{equation}
Therefore, we have
\begin{align}
\langle\mathbf{V} , \widetilde{\mathbf{D}} \mathbf{U} \rangle_{\mathbf{H}} + \langle  \widetilde{\mathbf{D}}\mathbf{V}, \mathbf{U} \rangle_{\mathbf{H}} = u_{P+1}^{+}v_{P+1}^{+} - u_{1}^{-}v_{1}^{-} = u(B_+)v(B_+) - u(B_-)v(B_-) =0.
\end{align}
\end{proof}
 %
A consistent DSEM approximation of the two-domain formulation, \eqref{eq:subdomain_1}--\eqref{eq:subdomain_2} with the interface conditions \eqref{eq:interface_condition} or \eqref{eq:interface_condition_proper} is
\begin{equation}
\begin{split}
        \label{eq:InterfaceSemiDiscrete_shift}
           \frac{{d}}{{dt}}\left( \frac{{d}}{{dt}} \mathbf{U}   - \mathbf{a}  \widetilde{\mathbf{D}}\mathbf{U}  \right) &- \mathbf{D}\left(\mathbf{a}\left( \frac{{d}}{{dt}} \mathbf{U}   - \mathbf{a}  \widetilde{\mathbf{D}}\mathbf{U}  \right) + \mathbf{b} \mathbf{D}\mathbf{U} \right) + \tau_N \mathbf{H}^{-1}\mathbf{\widehat{B}} \left(\mathbf{a}\left( \frac{{d}}{{dt}} \mathbf{U}   - \mathbf{a}  \widetilde{\mathbf{D}}\mathbf{U}  \right) + \mathbf{b} \mathbf{D}\mathbf{U} \right) \\
          &+  \mathbf{H}^{-1} \left( \gamma_N\mathbf{D}^T \mathbf{b}\mathbf{\widehat{B}}  +  \tau_0\mathbf{\widehat{B}}^T \mathbf{b}\mathbf{H}^{-1}\mathbf{\widehat{B}}  \right)\mathbf{U} = 0.
\end{split}
\end{equation}
Note that for the exact solutions of the IBVP the interface conditions  \eqref{eq:interface_condition} or \eqref{eq:interface_condition_proper} are satisfied exactly, and the penalty terms in \eqref{eq:InterfaceSemiDiscrete_shift} vanish and we recover \eqref{eq:InterfaceSemiDiscrete_shift_no_penalty}. 
The  penalty terms with coefficients $\gamma_0, a$ (note that $\gamma_0$ is hidden in the operator $\widetilde{\mathbf{D}}$) are flux corrections due to advective transport, when $a\ne 0$, and the remaining penalty terms are flux corrections due to expanding pressure waves. 
Note that when $a \equiv 0$ we obtain the standard interior penalty method for the classical scalar wave equation. The interface treatment \eqref{eq:InterfaceSemiDiscrete_shift} can be viewed as  the extension of the classical interior penalty method to the shifted wave equation where the interactions of expanding pressure waves and transport phenomena are prominent.
We note however, unlike  the classical interior penalty method,  the penalty parameters derived in this study are non-dimensional and independent of the mesh and material parameters of the medium.
\begin{remark}
We note that the numerical method \eqref{eq:InterfaceSemiDiscrete_shift} does not require the introduction of any auxiliary variables. 
As opposed to the energy DG method \cite{Zhang2019} the numerical method \eqref{eq:InterfaceSemiDiscrete_shift} is purely explicit and does not require the inversion of any matrix.
In contrast to classical SBP finite difference methods the numerical method \eqref{eq:InterfaceSemiDiscrete_shift} is arbitrarily and spectrally accurate and maintains full accuracy within the element, that is, there is no loss of accuracy close to the boundaries \cite{Wang2017}.
\end{remark}

Next, we will determine the penalty parameters such that the scheme is conservative and energy stable.
We will prove that the penalty terms  $\tau_N = 1/2$, $\gamma_N = -1/2$,  and $\gamma_0 = 1/2$  ensure a conservative and stable numerical approximation.
To do this, we introduce the auxiliary variable $\mathbf{W}$, use the ultra-compatibiltiy property of the SBP operators  $D_x$, $D_x\left(bD_x\right)$,  and rewrite the semi-discrete approximation \eqref{eq:InterfaceSemiDiscrete_shift} as
\begin{align}
        \label{eq:InterfaceSemiDiscrete3}
          \frac{{d}}{{dt}} \mathbf{U} = \mathbf{a} \widetilde{\mathbf{D}} \mathbf{U}  +  \mathbf{W}, \quad \frac{{d}}{{dt}} \mathbf{W}  =   \widetilde{\mathbf{D}}\left( \mathbf{a} \mathbf{W}  \right)  +  \widetilde{\mathbf{D}}_2 \mathbf{U} , 
\end{align}
where 
\begin{align}
        \label{eq:Discrete_operatorD2}
 \widetilde{\mathbf{D}}_2=  \mathbf{H}^{-1} \left( -\mathbf{A}_b + \widetilde{\mathbf{B}}  \mathbf{b} \widetilde{\mathbf{D}} \right),
\end{align}
and
{\small
\begin{align}\label{eq:Discrete_operatorA}
\mathbf{A}_{b}&=\begin{pmatrix} D_x^T\mathbf{b}H_xD_x & \mathbf{0} \\ \mathbf{0} & D_x^T\mathbf{b}H_xD_x \end{pmatrix}  
          -\frac{1}{2}  \left(\mathbf{\widehat{B}}^T  \begin{pmatrix} \mathbf{b}D_x & \mathbf{0} \\ \mathbf{0} &  \mathbf{b}D_x \end{pmatrix}  +  \begin{pmatrix} D_x^T\mathbf{b}  & \mathbf{0} \\ \mathbf{0} & D_x^T\mathbf{b}  \end{pmatrix} \mathbf{\widehat{B}}  - 2\tau_0 \mathbf{\widehat{B}}^T\mathbf{b}\begin{pmatrix} H_x & \mathbf{0} \\ \mathbf{0} & H_x \end{pmatrix}^{-1}\mathbf{\widehat{B}} \right).
\end{align}
}

The operator $\mathbf{A}_{b}$ is related to $\widetilde{\mathbf{D}}$ by the following lemma.
\begin{lemma}\label{lem:symmetry_A}
Consider the discrete operators  $\widetilde{\mathbf{D}}_2$, $\mathbf{A}_{b}$ and $\widetilde{\mathbf{D}}$ defined in \eqref{eq:Discrete_operatorD2}--\eqref{eq:Discrete_operatorA} and \eqref{eq:penalise_differential}. For all $b(x) >0$, if $\gamma_0 = 1/2$
and  $\tau_0 = 1/4$, then the operators $\widetilde{\mathbf{D}}_2$ and $\widetilde{\mathbf{D}}$ are ultra-compatible SBP operators, and we have
 \begin{align*}
\widetilde{\mathbf{D}}_2=  \mathbf{H}^{-1} \left( -\mathbf{A}_b + \widetilde{\mathbf{B}}  \mathbf{b} \widetilde{\mathbf{D}} \right), \quad  \mathbf{A}_b = {\widetilde{\mathbf{D}}}^T \mathbf{b}\mathbf{H} \widetilde{\mathbf{D}}.
  \end{align*}
\end{lemma}
\begin{proof}
It suffices to prove that $ \mathbf{A}_b = {\widetilde{\mathbf{D}}}^T \mathbf{b}\mathbf{H} \widetilde{\mathbf{D}}$. We consider the discrete operator $\widetilde{\mathbf{D}}$
\begin{equation*}
\begin{split}
\widetilde{\mathbf{D}} &= \begin{pmatrix} D_x& \mathbf{0} \\ \mathbf{0} & D_x \end{pmatrix} - \frac{1}{2}\begin{pmatrix} H_x & \mathbf{0} \\ \mathbf{0} & H_x \end{pmatrix}^{-1} \mathbf{\widehat{B}},
\end{split}
\end{equation*}
and expand the product $ {\widetilde{\mathbf{D}}}^T \mathbf{b}\mathbf{H} \widetilde{\mathbf{D}}$, simplifying further gives the desired result.
\end{proof}
To analyse conservation and stability of the interface treatment we will ignore contributions from the external boundaries and consider
$ \widetilde{\mathbf{D}}_2=  -\mathbf{H}^{-1}\mathbf{A}_b $.

The following theorem states that the discrete interface treatment is conservative,  which is a discrete analogue of Theorem \ref{theo:conservative_principle}.
\begin{theorem}\label{theo:conservative_interface_treatment}
The discrete interface treatment  \eqref{eq:InterfaceSemiDiscrete_shift} or \eqref{eq:InterfaceSemiDiscrete3} is conservative and satisfies  
\begin{align*}
\frac{d}{dt}\Big\langle\mathbf{1}, \mathbf{W} \Big\rangle_\mathbf{H}  = \frac{d}{dt}\Big\langle\mathbf{1},   \left(\frac{\mathrm{d}}{\mathrm{dt}} \mathbf{U} -\mathbf{a} \widetilde{\mathbf{D}} \mathbf{U}\right)\Big\rangle_\mathbf{H} = 0, \quad \mathbf{1} = \left(1, 1, \cdots, 1\right)^T \in \mathbb{R}^{2P+2}.
\end{align*}
\end{theorem}
\begin{proof}
Consider 
\begin{align*}
\frac{d}{dt}\Big\langle\mathbf{1}, \mathbf{W} \Big\rangle_\mathbf{H}  &= \Big\langle\mathbf{1}, \frac{d}{dt}\mathbf{W} \Big\rangle_\mathbf{H} \\
&= \Big\langle\mathbf{1},    \widetilde{\mathbf{D}}(\mathbf{a} \mathbf{W})\Big\rangle_\mathbf{H} - \Big\langle\mathbf{1},   \mathbf{H}^{-1}  \mathbf{A}_b  \mathbf{U}\Big\rangle_\mathbf{H} .
\end{align*}
Using Lemma \ref{lem:anti_symmetry_D} and Lemma \ref{lem:symmetry_A} gives the desired result
\begin{align*}
\frac{d}{dt}\Big\langle\mathbf{1}, \mathbf{W} \Big\rangle_\mathbf{H}  &= - \Big\langle \widetilde{\mathbf{D}}\mathbf{1},    \mathbf{a}\mathbf{W}\Big\rangle_\mathbf{H} - \Big\langle \widetilde{\mathbf{D}}\mathbf{1},    \mathbf{b}\widetilde{\mathbf{D}}\mathbf{U}\Big\rangle_\mathbf{H} = 0,
\end{align*}
since $\widetilde{\mathbf{D}}\mathbf{1} = \mathbf{0}$.
\end{proof}
We have now proven that the interface treatment \eqref{eq:InterfaceSemiDiscrete_shift} or \eqref{eq:InterfaceSemiDiscrete3} satisfies a discrete analogue of the conservative principle Theorem \ref{theo:conservative_principle}. Next we will prove that the interface treatment \eqref{eq:InterfaceSemiDiscrete_shift} is energy stable.

We define the discrete energy
\begin{equation}\label{eq:EnergyNorm_discrete_2_def}
\begin{split}
\mathcal{E}_b\left(t\right) &= \Big\langle\mathbf{W}, \mathbf{W} \Big\rangle_\mathbf{H} +   \Big\langle\mathbf{U}, \mathbf{U} \Big\rangle_{A_b}  \ge 0, \quad \Big\langle\mathbf{U}, \mathbf{U} \Big\rangle_{A_b}= \mathbf{U}^T\mathbf{A}_{b} \mathbf{U} =\left({\widetilde{\mathbf{D}}}\mathbf{U}\right)^T \mathbf{b}\mathbf{H} \left({\widetilde{\mathbf{D}}}\mathbf{U}\right)  \ge  0. 
\end{split}
\end{equation}

We will prove 
\begin{theorem}\label{theorem:energy_stable}
For constant coefficient problem with  $a\in \mathbb{R}$ and $b> 0$,  
the discrete interface treatment  \eqref{eq:InterfaceSemiDiscrete_shift} or \eqref{eq:InterfaceSemiDiscrete3}  satisfies the energy estimate 
\begin{equation}\label{eq:EnergyNorm_discrete_2}
\begin{split}
\frac{d}{dt}\mathcal{E}_b\left(t\right) &= 0 .
\end{split}
\end{equation}
\end{theorem}
\begin{proof}
We take the time derivative of the energy, and obtain
\begin{align}\label{eq:time_derivative_energy}
\frac{d}{dt}\mathcal{E}_b\left(t\right)  &= \Big\langle\mathbf{W}, \frac{d}{dt}\mathbf{W} \Big\rangle_\mathbf{H} +  \Big\langle\frac{d}{dt}\mathbf{W}, \mathbf{W}  \Big\rangle_\mathbf{H} + \Big\langle\mathbf{U}, \frac{d}{dt}\mathbf{U} \Big\rangle_{A_b} +  \Big\langle\frac{d}{dt}\mathbf{U}, \mathbf{U}  \Big\rangle_{A_b}.
\end{align}
Using \eqref{eq:InterfaceSemiDiscrete3}, we replace the time derivatives in the right hand side of \eqref{eq:time_derivative_energy} and obtain 
\begin{equation}\label{eq:time_derivative_energy_0}
\begin{split}
\frac{d}{dt}\mathcal{E}_b\left(t\right)  
& = a\Big\langle\mathbf{W}, \widetilde{\mathbf{D}}\mathbf{W} \Big\rangle_\mathbf{H} +  a\Big\langle \widetilde{\mathbf{D}}\mathbf{W}, \mathbf{W}  \Big\rangle_\mathbf{H}
    -\Big\langle\mathbf{W}, \mathbf{U} \Big\rangle_{A_b} -  \Big\langle \mathbf{U}, \mathbf{W}  \Big\rangle_{A_b}\\
& + a\Big\langle\mathbf{U}, \widetilde{\mathbf{D}}\mathbf{U} \Big\rangle_{A_b} +  a\Big\langle \widetilde{\mathbf{D}}\mathbf{U}, \mathbf{U}  \Big\rangle_{A_b}
   + \Big\langle\mathbf{W}, \mathbf{U} \Big\rangle_{A_b} +  \Big\langle \mathbf{U}, \mathbf{W}  \Big\rangle_{A_b}.
\end{split}
\end{equation}
By Lemma \ref{lem:anti_symmetry_D}, the first two terms in \eqref{eq:time_derivative_energy_0} vanish, 
and by inspection the last two terms in \eqref{eq:time_derivative_energy_0} also cancel out.  Using  Lemma \ref{lem:symmetry_A} 
 and Lemma \ref{lem:anti_symmetry_D} gives the desired result 
\begin{equation*}
\begin{split}
\frac{d}{dt}\mathcal{E}_b \left(t\right)  
& = ab\Big\langle\widetilde{\mathbf{D}}\mathbf{U}, \widetilde{\mathbf{D}}\left(\widetilde{\mathbf{D}}\mathbf{U}\right) \Big\rangle_\mathbf{H} +  ab\Big\langle\widetilde{\mathbf{D}}\left(\widetilde{\mathbf{D}}\mathbf{U}\right), \widetilde{\mathbf{D}}\mathbf{U} \Big\rangle_\mathbf{H} = 0.
\end{split}
\end{equation*}
\end{proof}
In the absence of external boundaries, Theorems \ref{theo:conservative_interface_treatment} and \ref{theorem:energy_stable} prove that the numerical interface treatment \eqref{eq:InterfaceSemiDiscrete_shift} or \eqref{eq:InterfaceSemiDiscrete3} is conservative, and for constant coefficient ptoblems energy stable. The results easily extend to periodic (external) boundary conditions.
For IBVPs where non-periodic boundary conditions are present, we will argument Theorems \ref{theorem:energy_stable} with the analysis below.

Theorem \ref{theorem:energy_stable} holds for  any constant $a \in \mathbb{R}$ and $b>0$. In the following, we show that in the case {$c=b-a^2 > 0$}, the discrete interface treatment  \eqref{eq:InterfaceSemiDiscrete_shift} also satisfies another energy estimate even for variable coefficients $a,b$.

 We note that \eqref{eq:InterfaceSemiDiscrete_shift} can be rewritten as
\begin{align}
        \label{eq:InterfaceSemiDiscrete3_c0}
          \frac{\mathrm{d}}{\mathrm{dt}} \mathbf{U} =   \mathbf{W}, \quad \frac{\mathrm{d}}{\mathrm{dt}} \mathbf{W}  =   \mathbf{a}\widetilde{\mathbf{D}}\mathbf{W} +\widetilde{\mathbf{D}}\left(\mathbf{a}\mathbf{W} \right)  + \widetilde{\mathbf{D}}_2\mathbf{U}, \quad  \widetilde{\mathbf{D}}_2=- \mathbf{H}^{-1} \widetilde{\mathbf{D}}^T\mathbf{c}\mathbf{H}\widetilde{\mathbf{D}}.
\end{align}
%
%
By a new discrete energy $\mathcal{E}_c$ defined as
\begin{equation}\label{eq:EnergyNorm_discrete_2_defp}
\begin{split}
\mathcal{E}_c\left(t\right) &= \Big\langle\mathbf{W}, \mathbf{W} \Big\rangle_\mathbf{H} +   \Big\langle\mathbf{U}, \mathbf{U} \Big\rangle_{A_c}  \geq 0, \quad \Big\langle\mathbf{U}, \mathbf{U} \Big\rangle_{A_c}  =  \mathbf{U}^T\widetilde{\mathbf{D}}^T\mathbf{c}\mathbf{H}\widetilde{\mathbf{D}}\mathbf{U} \ge 0,
\end{split}
\end{equation}
 we have the following theorem for energy stability. 
\begin{theorem}\label{theorem:energy_stable_p}
For any $b>0$, if $c = b-a^2 >0$, then the discrete interface treatment  \eqref{eq:InterfaceSemiDiscrete3}  satisfies  the energy estimate
\begin{equation}\label{eq:EnergyNorm_discrete_2_p}
\begin{split}
\frac{d}{dt}\mathcal{E}_c\left(t\right) &= 0 .
\end{split}
\end{equation}
\end{theorem}
\begin{proof}
We take the time derivative of the energy, and we have
\begin{align}\label{eq:time_derivative_energy_c0}
\frac{d}{dt}\mathcal{E}_c\left(t\right)  &= \Big\langle\mathbf{W}, \frac{d}{dt}\mathbf{W} \Big\rangle_\mathbf{H} +  \Big\langle\frac{d}{dt}\mathbf{W}, \mathbf{W}  \Big\rangle_\mathbf{H} + \Big\langle\mathbf{U}, \frac{d}{dt}\mathbf{U} \Big\rangle_{A_c} +  \Big\langle\frac{d}{dt}\mathbf{U}, \mathbf{U}  \Big\rangle_{A_c}.
\end{align}
Using \eqref{eq:InterfaceSemiDiscrete3_c0}, we replace the time derivatives in the right hand side of \eqref{eq:time_derivative_energy}, having 
\begin{equation}\label{eq:time_derivative_energy_0_c0}
\begin{split}
\frac{d}{dt}\mathcal{E}_c\left(t\right)  
& = \Big\langle\mathbf{W}, \mathbf{a}\widetilde{\mathbf{D}}\mathbf{W} \Big\rangle_\mathbf{H} +\Big\langle\mathbf{W}, \widetilde{\mathbf{D}}(\mathbf{a}\mathbf{W}) \Big\rangle_\mathbf{H}   -\Big\langle\mathbf{W}, \mathbf{U} \Big\rangle_{A_c} \\
&+\Big\langle \mathbf{a}\widetilde{\mathbf{D}}\mathbf{W},\mathbf{W} \Big\rangle_\mathbf{H} +\Big\langle \widetilde{\mathbf{D}}(\mathbf{a}\mathbf{W}),\mathbf{W} \Big\rangle_\mathbf{H} -  \Big\langle \mathbf{U}, \mathbf{W}  \Big\rangle_{A_c}\\
 &  + \Big\langle\mathbf{W}, \mathbf{U} \Big\rangle_{A_c} +  \Big\langle \mathbf{U}, \mathbf{W}  \Big\rangle_{A_c}.
\end{split}
\end{equation}
Using Lemma \ref{lem:anti_symmetry_D} yields the desired result \eqref{eq:EnergyNorm_discrete_2_p}.
\end{proof}
\begin{remark}
The ultra-compatible properties of the element local spectral difference operators and the modified global spectral difference operators  are critical for the development of a conservative and provably stable numerical interface treatment for the shifted wave equation in second order form, for all well-posed medium parameters, $a, b \in \mathbb{R}$, $b>0$.
\end{remark}

The motivation of deriving two energy estimates for the numerical interface treatment is to be able to obtain an energy estimate when generalising to multiple elements with physical boundary conditions. As will be shown in Sec.~\ref{sec_nbt}, the energy analysis for the numerical boundary treatment also uses two different discrete energies, depending on the sign of the parameter $c$. 


\subsection{Numerical boundary treatments}\label{sec_nbt}
We will now consider numerical enforcements of physical boundary conditions.
In particular, we will numerically impose the boundary conditions derived in Section \ref{sec:shifted_wave_equation}.
As above we will consider $c> 0$ and $c<0$ separately. 
The boundary conditions will be implemented weakly using penalties and we will prove numerical stability by deriving discrete energy estimates.
For simplicity we will consider numerical approximation in a single element with homogeneous boundary data ($g_1(t) =g_2(t)= 0$) and focus on the numerical boundary treatments. The analysis can be extended to multiple elements using theory developed in Section \ref{sec:numerical_interface_treatment}.

\subsubsection{Case 1:  $c=b-a^2>0$}
The semi-discrete approximation of the IVBP \eqref{1deqn} and \eqref{eq:BC_case1} in a single element can be written as 
\begin{equation}\label{eq:BoundarySemiDiscrete_shift_Case1}
\begin{split}
&\frac{d}{d t}\left(\frac{d \mathbf{u}}{d t} -\mathbf{a}D_x\mathbf{u}\right)-D_x\left(\mathbf{a}\left(\frac{d \mathbf{u}}{d t} -\mathbf{a}D_x\mathbf{u} \right)+\mathbf{b}D_x\mathbf{u}\right) \\
& = \tau_0 H_x^{-1}  \el \el^T\left( \frac{d \mathbf{u}}{d t} - \boldsymbol{\lambda}_1 D_x \mathbf{u}\right)+ \tau_N H_x^{-1}  \er \er^T\left(\frac{d \mathbf{u}}{d t} - \boldsymbol{\lambda}_2 D_x \mathbf{u}\right),
\end{split}
\end{equation}
where $\tau_0$ and $\tau_N$ are mesh independent penalty parameter to be determined by requiring stability. The diagonal matrices $\boldsymbol{\lambda}_1$, $\boldsymbol{\lambda}_2$, $\mathbf{a}$, $\mathbf{b}$ and $\mathbf{c}$ are the variables $a$, $b$, $c$, $\lambda_1$, $\lambda_2$ from \eqref{lambda} evaluated on the quadrature nodes, respectively.   We also define quantities corresponding to the first diagonal element $a_1$, $b_1$, $c_1$, $\lambda_{11}$, $\lambda_{21}$, and the last diagonal element as $a_N$, $b_N$, $c_N$, $\lambda_{1N}$, $\lambda_{2N}$.
Let 
\begin{equation}\label{eq:Energy_IBVP_Case1}
\begin{split}
\mathbf{A}_c = D_x^T\mathbf{c}H_xD_x, \quad \Big\langle\mathbf{u}, \mathbf{u} \Big\rangle_{A_c} =   \mathbf{u}^T\mathbf{A}_c \mathbf{u}\ge 0,
\end{split}
\end{equation}
and define the discrete energy
\begin{equation}\label{eq:EnergyNorm_IBVP_Case1}
\begin{split}
\mathcal{E}_c\left(t\right) &= \Big\langle\frac{d\mathbf{u}}{dt}, \frac{d\mathbf{u}}{dt} \Big\rangle_{H_{x}} +   \Big\langle\mathbf{u}, \mathbf{u} \Big\rangle_{A_c}  \geq 0 .
\end{split}
\end{equation}
We have the following theorem for the stability of \eqref{eq:BoundarySemiDiscrete_shift_Case1}.
\begin{theorem}\label{theorem:energy_stable_IBVP_case1}
The discrete boundary treatment  \eqref{eq:BoundarySemiDiscrete_shift_Case1} with the penalty parameters $\tau_0=\lambda_{21}$ and $\tau_N=-\lambda_{1N}$   satisfies  the energy estimate 
\begin{equation}\label{eq:energy_stable_IBVP_case1}
\begin{split}
\frac{d}{dt}\mathcal{E}_c\left(t\right)=-2\sqrt{b_1}\left(\frac{d\mathbf{u}}{d t}\right)^T  \el \el^T\left(\frac{d\mathbf{u}}{d t}\right)-2\sqrt{b_N}\left(\frac{d\mathbf{u}}{d t}\right)^T  \er \er^T\left(\frac{d\mathbf{u}}{d t}\right)\leq 0.
\end{split}
\end{equation}
\end{theorem}
\begin{proof}
We rewrite \eqref{eq:BoundarySemiDiscrete_shift_Case1}  as
\begin{equation}\label{eq:BoundarySemiDiscrete_shift_Case1_proof_0}
\begin{split}
\frac{d^2 \mathbf{u}}{d t^2} =& \mathbf{a} D_x\frac{d\mathbf{u}}{d t} + D_x\mathbf{a}\frac{d\mathbf{u}}{d t} +D_x\mathbf{c}D_x\mathbf{u}  \\
&+ \tau_0 H_x^{-1}  \el \el^T\left( \frac{d \mathbf{u}}{d t} - \boldsymbol{\lambda}_1 D_x \mathbf{u}\right)+ \tau_N H_x^{-1}  \er \er^T\left(\frac{d \mathbf{u}}{d t} - \boldsymbol{\lambda}_2 D_x \mathbf{u}\right).
\end{split}
\end{equation}
 By taking the time derivative of the discrete energy $\mathcal{E}_c$,  we obtain
\begin{align}\label{eq:time_derivative_energy_proof}
\frac{d}{dt}\mathcal{E}_c\left(t\right)  &= \Big\langle\frac{d\mathbf{u}}{d t},  \frac{d^2\mathbf{u}}{d t^2} \Big\rangle_{H_{x}} +  \Big\langle \frac{d^2\mathbf{u}}{d t^2} , \frac{d\mathbf{u}}{d t} \Big\rangle_{H_{x}} + \Big\langle\mathbf{u}, \frac{d}{dt}\mathbf{u} \Big\rangle_{A_c} +  \Big\langle\frac{d}{dt}\mathbf{u}, \mathbf{u}  \Big\rangle_{A_c}.
\end{align}
In \eqref{eq:time_derivative_energy_proof}, we replace the second time derivative with the right hand side of \eqref{eq:BoundarySemiDiscrete_shift_Case1_proof_0}, and use the ultra-compatibility of the SBP operators \eqref{sdoD}--\eqref{SBP_1st} and  \eqref{sdoDD}--\eqref{SBP_2nd}.
This gives, 
\begin{align*}
\frac{d}{dt}\mathcal{E}_c\left(t\right)= &2 \left(\frac{d\mathbf{u}}{d t}\right)^T  B\mathbf{a} \frac{d\mathbf{u}}{d t}  +2\left(\frac{d\mathbf{u}}{d t}\right)^T B\mathbf{c}D\mathbf{u}\\
&+ 2\tau_0 \left(\frac{d\mathbf{u}}{d t}\right)^T  \el \el^T\left( \frac{d \mathbf{u}}{d t} - \boldsymbol{\lambda}_1 D \mathbf{u}\right)+ 2\tau_N \left(\frac{d\mathbf{u}}{d t}\right)^T \er \er^T\left(\frac{d \mathbf{u}}{d t} - \boldsymbol{\lambda}_2 D \mathbf{u}\right)\\
=& 2\left(\frac{d\mathbf{u}}{d t}\right)^T  \el \el^T\left( (\tau_0-a_1)\frac{d \mathbf{u}}{d t} +(-\tau_0 \lambda_{11}-c_1) D \mathbf{u}\right)\\
&+2\left(\frac{d\mathbf{u}}{d t}\right)^T \er \er^T\left((\tau_N+a_N)\frac{d \mathbf{u}}{d t} +(-\tau_N \lambda_{2N}+c_N) D \mathbf{u}\right),
\end{align*}
where $a_1$, $\lambda_{11}$ and $c_1$ are the first diagonal element of $ \mathbf{a}$, $\lambda_1$ and $ \mathbf{c}$, respectively. The quantities with subscript $N$ are defined analogously for the last diagonal element.  
Using $c=-\lambda_1\lambda_2$ and $a=\frac{\lambda_1+\lambda_2}{2}$,  we obtain the energy estimate \eqref{eq:energy_stable_IBVP_case1} with $\tau_0=\lambda_{21}$ and $\tau_N=-\lambda_{1N}$. 
\end{proof}
We note that the stability result can be extended to multiple elements by combining Theorem \ref{theorem:energy_stable_IBVP_case1}  with the numerical interface treatment Theorem \ref{theorem:energy_stable_p} for general problems with variable coefficients.


\subsubsection{Case 2: $a > 0$ and $ c< 0$}
The relation $a > 0$ and $ c< 0$   implies $\lambda_1 > \lambda_2 > 0$. A consistent  semi-discrete approximation of the IVBP \eqref{1deqn} and \eqref{eq:BC_case2} can be written as 
\begin{align}
        \label{eq:BoundarySemiDiscrete_shift_Case2}
          & \frac{{d}}{{dt}}\left( \frac{d \mathbf{u}}{d t} - \mathbf{a}\left(D_x-   \gamma_0{H_{x}^{-1}  \er \er^T} \right) \mathbf{u}  \right) - D_x 
          \left(\mathbf{a}\left(\frac{d \mathbf{u}}{d t} - \mathbf{a}\left(D_x-   \gamma_0{H_{x}^{-1}  \er \er^T} \right) \mathbf{u} \right) + \mathbf{b}D_x  \mathbf{u}  \right)  
            \\ \notag
          & + {\tau_N H_{x} ^{-1}  \er \er^T\mathbf{a}\left(  \frac{d \mathbf{u}}{d t} - \mathbf{a}\left(D_x-   \gamma_0{H_{x}^{-1}  \er \er^T} \right) \mathbf{u} \black \right) \black }   \\ \notag
          &+ { H_{x}^{-1} \left(\gamma_N D_x^T \mathbf{b}  \er \er^T + \tau_0    \er \er^T \mathbf{b}H_{x}^{-1}\er \er^T  \right)   \mathbf{u}}  = 0,
\end{align}
where $\gamma_0, \tau_0, \tau_N$ and  $\gamma_N $ are penalty  to be determined by requiring stability.
We will show how to choose the penalty parameters and prove energy stability. 

To begin, set $\gamma_0 = \tau_0 = \tau_N = 1$ and  $\gamma_N = -1$, from \eqref{eq:BoundarySemiDiscrete_shift_Case2} we have
\begin{align}
        \label{eq:BC_R0_SemiDiscrete_shift_0}
          & \frac{{d}}{{dt}}\left( \frac{d \mathbf{u}}{d t} - \mathbf{a}\left(D_x-   {H_{x}^{-1}  \er \er^T } \right) \mathbf{u}  \right) - \left(D_x -H_{x}^{-1}  \er \er^T  \right) 
          \left(\mathbf{a}\left(\frac{d \mathbf{u}}{d t} - \mathbf{a}\left(D_x-   {H_{x}^{-1}  \er \er^T } \right) \mathbf{u} \right)  \right)  
            \\ \notag
          & -D_x\mathbf{b}D_x  \mathbf{u}  - { H_{x}^{-1} \left( D_x^T \mathbf{b}  \er \er^T -      \er \er^T \mathbf{b}H_{x}^{-1}\er \er^T  \right)   \mathbf{u}}  = 0.        
\end{align}
With the modified operators 
\begin{align}\label{eq:modified_SBP_D}
\widetilde{D}_x = \left(D_x-   {H_x^{-1}  \er \er^T} \right), \quad {A}_b = \widetilde{D}_x ^TbH_x\widetilde{D}_x ,
\end{align}
we have the following lemma. 
\begin{lemma}\label{lem:dissipative_D}
Consider the modified operator $\widetilde{D}_x$ defined in \eqref{eq:modified_SBP_D} and the grid function $\mathbf{u} \in\mathbb{R}^{P+1}$. We have
\begin{align}
\langle\mathbf{u} , \widetilde{D}_x \mathbf{u} \rangle_{{H_{x}}} + \langle  \widetilde{D}_x\mathbf{u}, \mathbf{u} \rangle_{{H_{x}}} =  -\left(u_{1}^2 + u_{P+1}^2 \right) \le 0.
\end{align}
\end{lemma}
\begin{proof}
Consider $\widetilde{D}_x$ defined \eqref{eq:modified_SBP_D} and simplify using the SBP property \eqref{sdoD},
\begin{align*}
\widetilde{D}_x &= \left(D_x-   {H_x^{-1}  \er \er^T } \right)\\
                        &= H_x^{-1} \left(-D_x^TH_x-    \el \el^T \right)
\end{align*}
Therefore, we have
\begin{align*}
\langle\mathbf{u} , \widetilde{D}_x \mathbf{u} \rangle_{{H}} + \langle  \widetilde{D}_x\mathbf{u}, \mathbf{u} \rangle_{{H}}  =-\mathbf{u}^T \left(\el \el^T    +  \er \er^T \right) \mathbf{u}   = -\left(u_{1}^2 + u_{P+1}^2 \right) \le 0.
\end{align*}
\end{proof}

We introduce the auxiliary variable $\mathbf{w}$ such that \eqref{eq:BC_R0_SemiDiscrete_shift_0} can be written as
\begin{align}
        \label{eq:BC_R0_SemiDiscrete_shift2_BC2}
          \frac{d \mathbf{u}}{d t}  =  \mathbf{a}\widetilde{D}_x \mathbf{u}  + \mathbf{w}, \quad  \frac{\mathrm{d}}{\mathrm{dt}}\mathbf{w} =  \widetilde{D}_x\left(\mathbf{a}\mathbf{w}\right) + \widetilde{D}_{xx}  \mathbf{u}, \quad \widetilde{D}_{xx}  =-  H_{x}^{-1}{A}_b,
%
\end{align}
where we have used the ultra-compatibility of the SBP operators \eqref{sdoD}--\eqref{SBP_1st} and  \eqref{sdoDD}--\eqref{SBP_2nd}.
With the discrete energy
\begin{equation}\label{eq:EnergyNorm_discrete_2_bc2}
\begin{split}
\mathcal{E}_b\left(t\right) &=  \langle\mathbf{w} , \mathbf{w} \rangle_{{H_{x}}}  +   \langle\mathbf{u} , \mathbf{u} \rangle_{{A}_b} >0, \quad  \langle\mathbf{u} , \mathbf{u} \rangle_{{A}_b}:=\mathbf{u} ^T {A}_b  \mathbf{u} \geq 0 ,
\end{split}
\end{equation}
we have the following theorem for the stability of \eqref{eq:BC_R0_SemiDiscrete_shift_0}.
\begin{theorem}\label{theorem:energy_stable_IBVP_case2}
Consider the discrete boundary treatment \eqref{eq:BoundarySemiDiscrete_shift_Case2} for constants $a >0$, $b>0$ and $c = b-a^2< 0$. The numerical boundary treatment  \eqref{eq:BoundarySemiDiscrete_shift_Case2} with the penalty parameters $\gamma_0 = \tau_0 = \tau_N = 1$ and  $\gamma_N = -1$  satisfies the energy estimate  
\begin{equation}\label{eq:energy_stable_IBVP_case2}
\begin{split}
\frac{d}{dt}\mathcal{E}_b\left(t\right) = -a(w_1^2 + w_{P+1}^2 ) -ab\left(\left(\widetilde{D}_x\mathbf{u}\right)_{1}^2 + \left(\widetilde{D}_x\mathbf{u}\right)_{P+1}^2\right) \le 0. 
\end{split}
\end{equation}
\end{theorem}

\begin{proof}
We take the time derivative of the energy, and we have
\begin{align}\label{eq:time_derivative_energy_BC2}
\frac{d}{dt}\mathcal{E}_b\left(t\right)  &= \Big\langle\mathbf{w}, \frac{d}{dt}\mathbf{w} \Big\rangle_{H_{x}} +  \Big\langle\frac{d}{dt}\mathbf{w}, \mathbf{w}  \Big\rangle_{H_{x}} + \Big\langle\mathbf{u}, \frac{d}{dt}\mathbf{u} \Big\rangle_{{A}_b } +  \Big\langle\frac{d}{dt}\mathbf{u}, \mathbf{u}  \Big\rangle_{{A}_b }.
\end{align}
Using \eqref{eq:BC_R0_SemiDiscrete_shift2_BC2}, we replace the time derivatives in the right hand side of \eqref{eq:time_derivative_energy_BC2}, and obtain 
\begin{equation}\label{eq:time_derivative_energy_0_BC2}
\begin{split}
\frac{d}{dt}\mathcal{E}_b\left(t\right)  
& = a\Big\langle\mathbf{w}, \widetilde{D}_x\mathbf{w} \Big\rangle_{H_{x}} +  a\Big\langle \widetilde{D}_x\mathbf{w}, \mathbf{w}  \Big\rangle_{H_{x}}
    -\Big\langle\mathbf{w}, \mathbf{u} \Big\rangle_{{A}_b } -  \Big\langle \mathbf{u}, \mathbf{w}  \Big\rangle_{{A}_b }\\
& + a\Big\langle\mathbf{u}, \widetilde{D}_x\mathbf{u} \Big\rangle_{{A}_b } +  a\Big\langle \widetilde{D}_x\mathbf{u}, \mathbf{u}  \Big\rangle_{{A}_b }
   + \Big\langle\mathbf{w}, \mathbf{u} \Big\rangle_{{A}_b } +  \Big\langle \mathbf{u}, \mathbf{w}  \Big\rangle_{{A}_b }.
\end{split}
\end{equation}
By Lemma \ref{lem:dissipative_D} and \eqref{eq:modified_SBP_D}, we have
\begin{equation*}
\begin{split}
\frac{d}{dt}\mathcal{E}_b\left(t\right)  
& = -a(w_1^2 + w_{P+1}^2 )+ab\Big\langle\widetilde{D}_x\mathbf{u}, \widetilde{D}_x\left(\widetilde{D}_x\mathbf{u}\right) \Big\rangle_{H_{x}} +  ab\Big\langle\widetilde{D}_x\left(\widetilde{D}_x\mathbf{u}\right), \widetilde{D}_x\mathbf{u} \Big\rangle_{H_{x}}.
\end{split}
\end{equation*}
Using  Lemma \ref{lem:dissipative_D} again gives the desired result \eqref{eq:energy_stable_IBVP_case2}.
\end{proof}

To obtain an energy estimate for the semidiscretisation with multiple elements and physical boundary conditions, we shall combine the analysis in Theorem \ref{theorem:energy_stable} for the numerical interface treatment and  \ref{theorem:energy_stable_IBVP_case2} for the numerical boundary treatment. 

\begin{remark}\label{remark_case3}
We note that similar analysis holds for {Case 3: $a < 0$ and $ c< 0$}. In Case 3, we have $\lambda_2 < \lambda_1 < 0$, and  need to impose the  two boundary conditions in \eqref{eq:BC_case3} or  \eqref{eq:BC_case3_0} at $x = B_{-}$, and  no boundary condition  at $x = B_{+}$. 
\end{remark}

\section{Error estimates}

In this section, we derive error estimates for the semidiscretisation with multiple elements and physical boundary conditions, and focus on $h$-convergence in energy norms. In the stability analysis, two discrete energy norms are used for positive and negative $c$. In the following error analysis, we also consider these two cases separately.    For the sake of simplied notation, we use the two elements model $\Omega=\Omega_-\cup\Omega_+$ with both interface conditions at $\Omega_-\cap\Omega_+$ and physical boundary conditions at $\partial\Omega$.

Let the vector $\mathbf{U}_{ex}$ be the exact solution evaluated on the quadrature points, then the pointwise error $\mathbf{E}=\mathbf{U}_{ex}-\mathbf{U}$ satisfies the error equation

\begin{equation}
\begin{split}
        \label{EE1}
           \frac{{d}}{{dt}}\left( \frac{{d}}{{dt}} \mathbf{E}   - \mathbf{a} \widetilde{\mathbf{D}}\mathbf{E}  \right) &- \mathbf{D}\left(\mathbf{a}\left( \frac{{d}}{{dt}} \mathbf{E}   - \mathbf{a}\widetilde{\mathbf{D}}\mathbf{E}  \right) + \mathbf{b}\mathbf{D}\mathbf{E} \right) = \mathbf{P_i} +\mathbf{P_b} +\mathbf{T}. 
\end{split}
\end{equation} 
The term $\mathbf{P_i}$ takes the form 
\begin{equation*}
\mathbf{P_i}=- \frac{1}{2} \mathbf{H}^{-1}\mathbf{\widehat{B}} \left(\mathbf{a}\left( \frac{{d}}{{dt}} \mathbf{E}   - \mathbf{a}\widetilde{\mathbf{D}}\mathbf{E}  \right) + \mathbf{b} \mathbf{D}\mathbf{E} \right) -  \mathbf{H}^{-1} \left( -\frac{1}{2}\mathbf{D}^T \mathbf{b}\mathbf{\widehat{B}}  +  \frac{1}{4}\mathbf{\widehat{B}}^T \mathbf{b}\mathbf{H}^{-1}\mathbf{\widehat{B}}  \right)\mathbf{E},
\end{equation*} 
and corresponds to the numerical interface treatment \eqref{eq:InterfaceSemiDiscrete_shift}. 
The precise form of $\mathbf{P_b}$ depends on the sign of the variable $c$. When $c>0$, we have  
\begin{equation*}
\mathbf{P_b}=\begin{bmatrix}
          \Lambda_2 H_x^{-1}\el \el^T\left( \frac{d}{dt}{\mathbf{E}^{-}} - \lambda_1 D_x \mathbf{E}^{-}\right) \\
          -\Lambda_1 H_x^{-1}  \er \er^T\left(\frac{d}{dt}\mathbf{E}^{+} - \lambda_2 D_x \mathbf{E}^{+}\right)
          \end{bmatrix},
\end{equation*}
and corresponds to the numerical boundary treatment   \eqref{eq:BoundarySemiDiscrete_shift_Case1}. When $c<0$, for constant coefficient problems, we instead have
\footnotesize
\begin{equation*}
\mathbf{P_b}=\begin{bmatrix}
          0 \\
            { H_{x} ^{-1}  \er \er^Ta\left(  \frac{d \mathbf{E^+}}{d t} - a\left(D_x-   {H_{x}^{-1}  \er \er^T} \right) \mathbf{E^+} \black \right) \black } + { H_{x}^{-1} \left(- D_{x}^T b  \er \er^T +  \er \er^T bH_{x}^{-1}\er \er^T \right)   \mathbf{E^+}}
          \end{bmatrix},
\end{equation*}
\normalsize
which corresponds to the numerical boundary treatment  \eqref{eq:BoundarySemiDiscrete_shift_Case2}.

The last term $\mathbf{T}$ contains  the truncation error on the quadrature points. We now determine the order of $\mathbf{T}$ according to the accuracy property of the spectral difference operators in \eqref{EE1}. 
On the left-hand side of \eqref{EE1}, the operator $\widetilde{\mathbf{D}}$ in the first term leads to a truncation error $\mathcal{O}(\Delta x^P)$, where $\Delta x$ is the element width. In the second term, applying the spectral difference operator twice results in a truncation error $\mathcal{O}(\Delta x^{P-1})$. On the right-hand side of  \eqref{EE1}, the penalty terms also contribute to the truncation error. In $\mathbf{P_i}$, the spectral difference operator in the first term again leads to a truncation error $\mathcal{O}(\Delta x^P)$. In combination with the $\Delta x^{-1}$ factor in $\mathbf{H}^{-1}$, the truncation error becomes $\mathcal{O}(\Delta x^{P-1})$. The operators in the second term in $\mathbf{P_i}$ introduce no truncation error. Similarly, the truncation error by the operators in $\mathbf{P_b}$ is $\mathcal{O}(\Delta x^{P-1})$ for both positive and negative $c$. As a consequence, the dominating truncation error of the semidiscretisation is $\mathcal{O}(\Delta x^{P-1})$, i.e. each component of $\mathbf{T}\sim\mathcal{O}(\Delta x^{P-1})$.

Next, we define the quantities
\[
\mathbf{E_w}=\begin{cases}
\frac{d}{dt}\mathbf{E}-a\widehat{\mathbf{D}}\mathbf{E},\quad c<0\\
\frac{d}{dt}\mathbf{E},\quad c>0
\end{cases},\quad
A = \begin{cases}
A_b,\quad c<0\\
A_c,\quad c>0
\end{cases},
\]
where 
\[
\widehat{\mathbf{D}}=\widetilde{\mathbf{D}}-\begin{pmatrix}
0 & 0 \\
0  & {H_x^{-1}  \er \er^T }
\end{pmatrix}.
\]
Here again, we consider constant coefficient problems when $c<0$. We define the error in the energy norm  
\begin{equation}\label{deltac}
\delta= \sqrt{\Big\langle\mathbf{E_w}, \mathbf{E_w} \Big\rangle_{H} +   \Big\langle\mathbf{E}, \mathbf{E} \Big\rangle_{A}},
\end{equation}
which can be bounded by the following theorem.
\begin{theorem}\label{thm_err}
For the semidiscretisation \eqref{eq:InterfaceSemiDiscrete_shift} and \eqref{eq:BoundarySemiDiscrete_shift_Case1} with $c>0$ (or \eqref{eq:BoundarySemiDiscrete_shift_Case2} when $c<0$), the error in the energy norm \eqref{deltac} satisfies 
\[
\delta \leq  \widetilde{C} \Delta x^{P-1} \int_0^t |u(\cdot,\tau)|_{H^{p+1}(\Omega)} d\tau,
\]
where $\widetilde{C} >0$ is independent of $\Delta x$.
\end{theorem}
\begin{proof}
We apply the energy analysis to the error equation \eqref{EE1},
\[
\frac{d}{dt}\delta^2 \leq \Big\langle \mathbf{E_w}, \mathbf{T} \Big\rangle_{H}\leq C {\delta}\Delta x^{P-1} \left|\frac{d^{P+1}u}{dx}\right|_\infty.
\]
We then have 
\[
\frac{d}{dt}{\delta} \leq  \widetilde{C} \Delta x^{P-1} \left|\frac{d^{P+1}u}{dx}\right|_\infty,
\]
and consequently
\[
{\delta} \leq  \widetilde{C} \Delta x^{P-1} \int_0^t \left|\frac{d^{P+1}u}{dx}(\cdot,\tau)\right|_\infty  d\tau.
\]
\end{proof}
The above error analysis can easily be generalised to the case with multiple elements, and the error in the energy norm also converges with rate $P-1$. We thus denote this rate as optimal. 
\begin{remark}
The theory and numerical method can be extended to higher space dimensions (2D and 3D) using tensor products of quadrilateral and hexahedral  elements. 
\end{remark}

\section{Numerical experiments}
In this section we present numerical experiments, in (1+1)-dimensions and (2+1)-dimensions,  to verify the theory presented in this study.
We will demonstrate spectral accuracy and verify both $h$- and $p$-convergence. We will consider first  (1+1)-dimensions (one space dimension) and proceed later to (2+1)-dimensions (two space dimensions).
\subsection{Numerical examples in one space dimension}
We consider the  shifted wave equation \eqref{1deqn}.
The numerical experiments here are designed to verify the discrete conservative principle Theorem \ref{theo:conservative_interface_treatment}, investigate numerical stability and accuracy.
\paragraph{Conservation and efficacy test}
We demonstrate the effectiveness of high order accuracy and verify that our method is conservative in the sense of Theorem \ref{theo:conservative_interface_treatment}.
 We consider the initial conditions 
\begin{equation}\label{eq:init_condition}
u(x,0) =u_0(x) =10e^{-\frac{(x - 10)^2}{0.08}}, \quad \frac{\partial u}{\partial t} \Big|_{t=0} = 0.
\end{equation}
The constant coefficients Cauchy problem \eqref{1deqn} with \eqref{eq:init_condition} has the analytical solution
\begin{equation}
u(x,t) = \frac{\lambda_1}{\lambda_1-\lambda_2}u_0(x-\lambda_2 t) - \frac{\lambda_2}{\lambda_1-\lambda_2}u_0(x-\lambda_1 t),
\end{equation}
where $\lambda_1 = a + \sqrt{b}, \quad \lambda_2 = a - \sqrt{b}.$
We consider two different material constants $a =  -0.5, b = 1$ and $a = -1, b = 0.25$. These two choices  correspond to $c = b - a^2$ positive and negative, respectively, and are referred to as subsonic  and supersonic  regimes. 
We consider the bounded domain $x\in [0,20]$ with the boundary condition \eqref{eq:BC_case1} for $c\ge 0$, \eqref{eq:BC_case3_0} for $c<0$, and homogeneous boundary data $g_1(t) = g_2(t) =0$.

To demonstrate the effectiveness of spectral accuracy, we consider two different discretisation parameters, polynomial approximation of degree $P = 100$ with $N =5$ uniform elements, and $P =  4$ with $N =101$ uniform elements. Both choices of numerical approximation parameters yield the same amount of degrees of freedom (DoF), $\mathrm{DoF} =(P+1)\times N=505$,  to be evolved.
We set the constant time-step $\Delta{t} = 5.4954\times 10^{-4}$ and evolve the solutions until the final time, $T = 20$. The snapshots of the solutions are displayed in Figures \ref{fig:supersonic_u_1D}--\ref{fig:subsonic_u_1D}, showing the evolutions of the initial Gaussian profile. 

In Figure \ref{conservationplot_00} we display the numerical errors.  For the same $\mathrm{DoF}=505$, the numerical error for $P = 100$  is significantly smaller than the numerical error for lower order polynomial approximation $P = 4$ for all $t>0$. At the final time $T =20$, the numerical errors differ by several orders of magnitude $\sim 10^{9}$, since $\|e\|_{L^2} \sim 10^{-11}$ for $P = 100$ and  $\|e\|_{L^2} \sim 10^{-2}$ for $P = 4$. In contrast, the computational time of $P=100$ is only four times the computational time of $P=4$ due to the difference in the sparsity of the discretisation operators. This demonstrates the remarkable efficiency of stable very high order methods.

To verify discrete conservation principle, that is Theorem \ref{theo:conservative_interface_treatment}, we consider the spatial domain $x\in [0,20]$ with periodic  boundary conditions $u(0,t) = u(20,t)$, and  the final time $T = 20$. The discretisation parameters are the same as above. For each time-step, we compute $w = \Big\langle\mathbf{1}, \mathbf{W} \Big\rangle_{H}$. The  evolution of $w(t + \Delta t) - w(t)$  is depicted in Figure \ref{conservationplot}. We observe that for each $t$ the quantity $w(t + \Delta t) - w(t)$ is zero up to machine precision, thus the method is conservative. Note that the numerical conservation principle is independent of the polynomial degree $P$ and number of elements $N$.

\begin{figure}[h!]
{\includegraphics[width=0.245\textwidth]{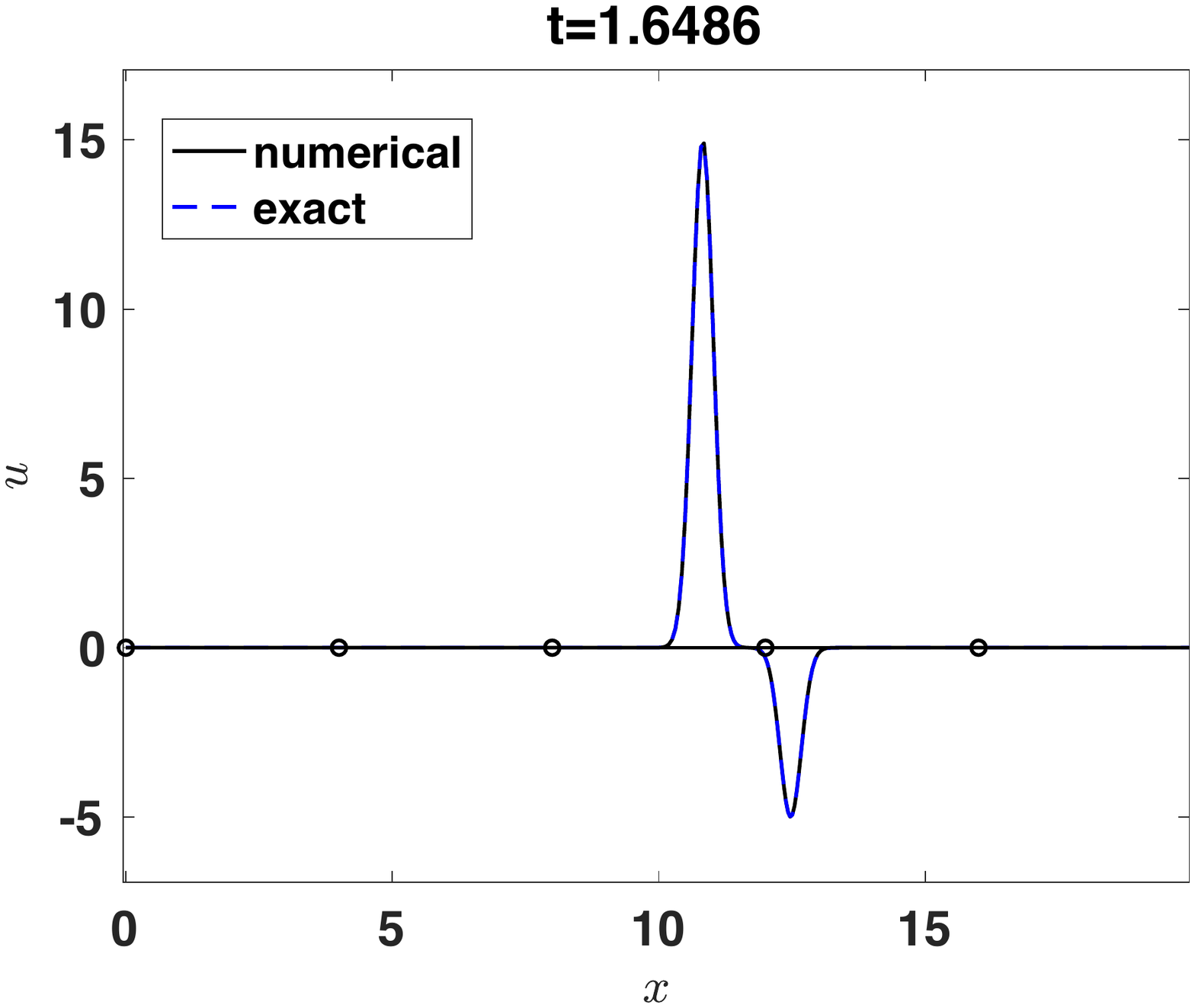}}
{\includegraphics[width=0.245\textwidth]{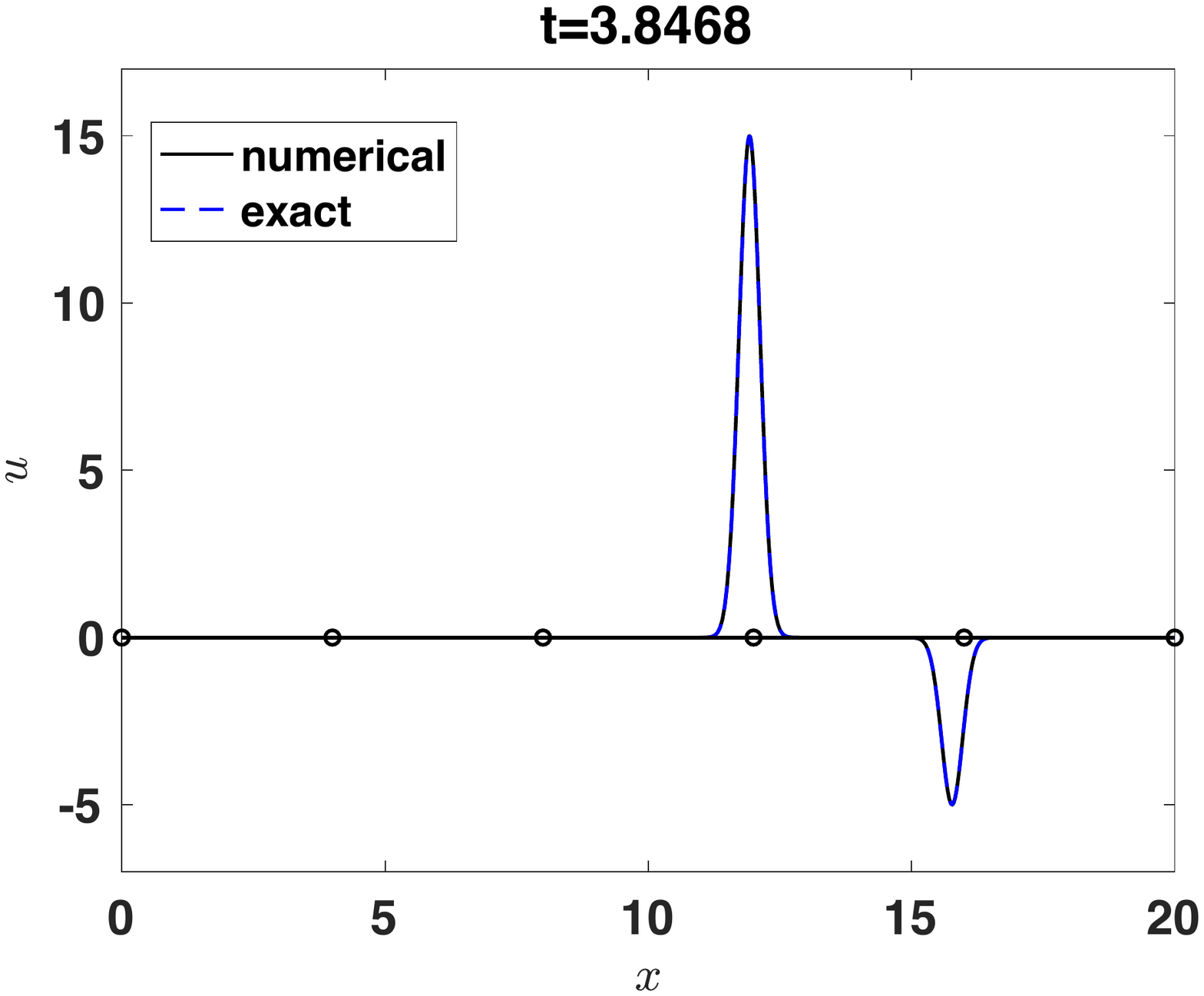}}
{\includegraphics[width=0.245\textwidth]{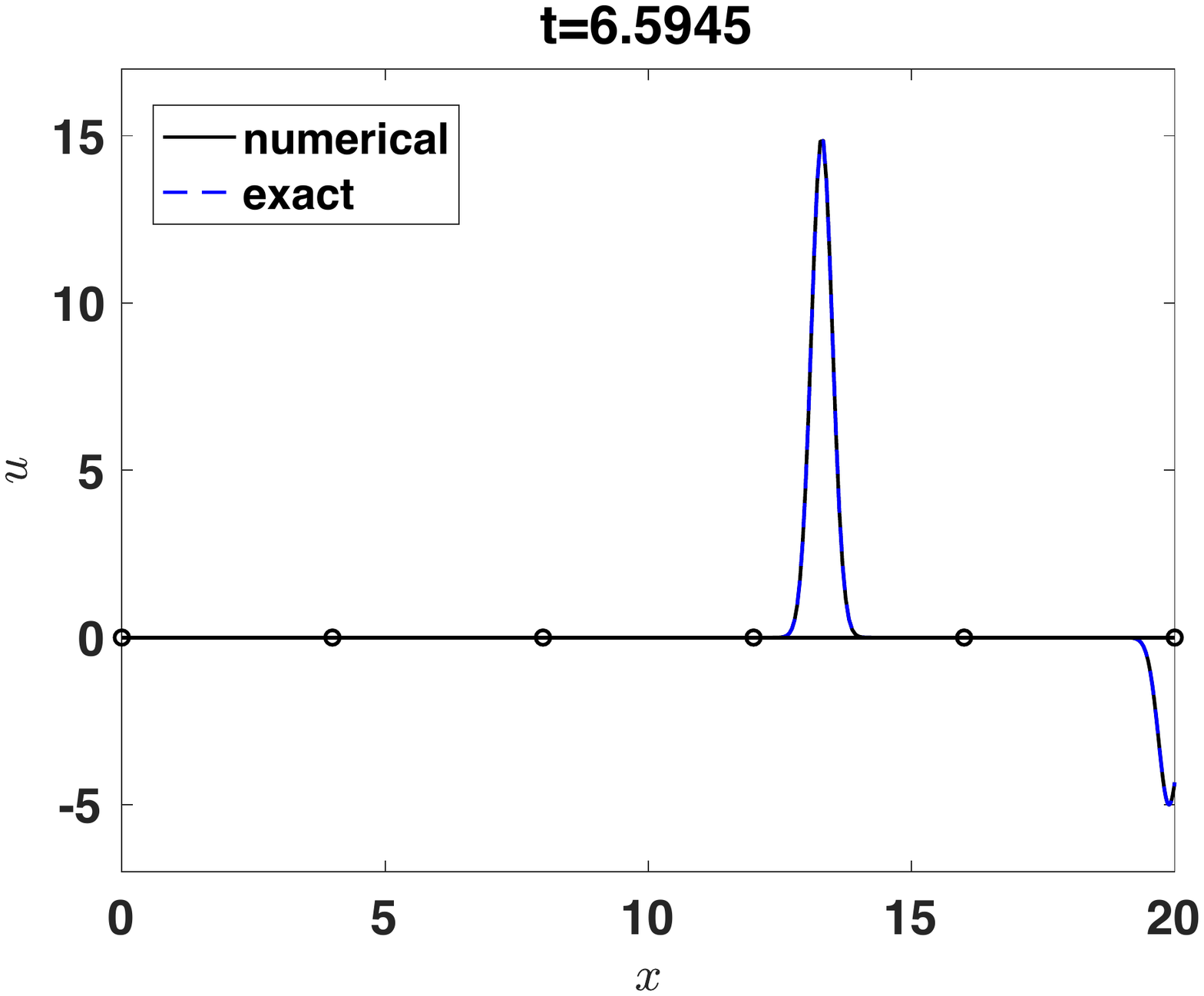}}
{\includegraphics[width=0.245\textwidth]{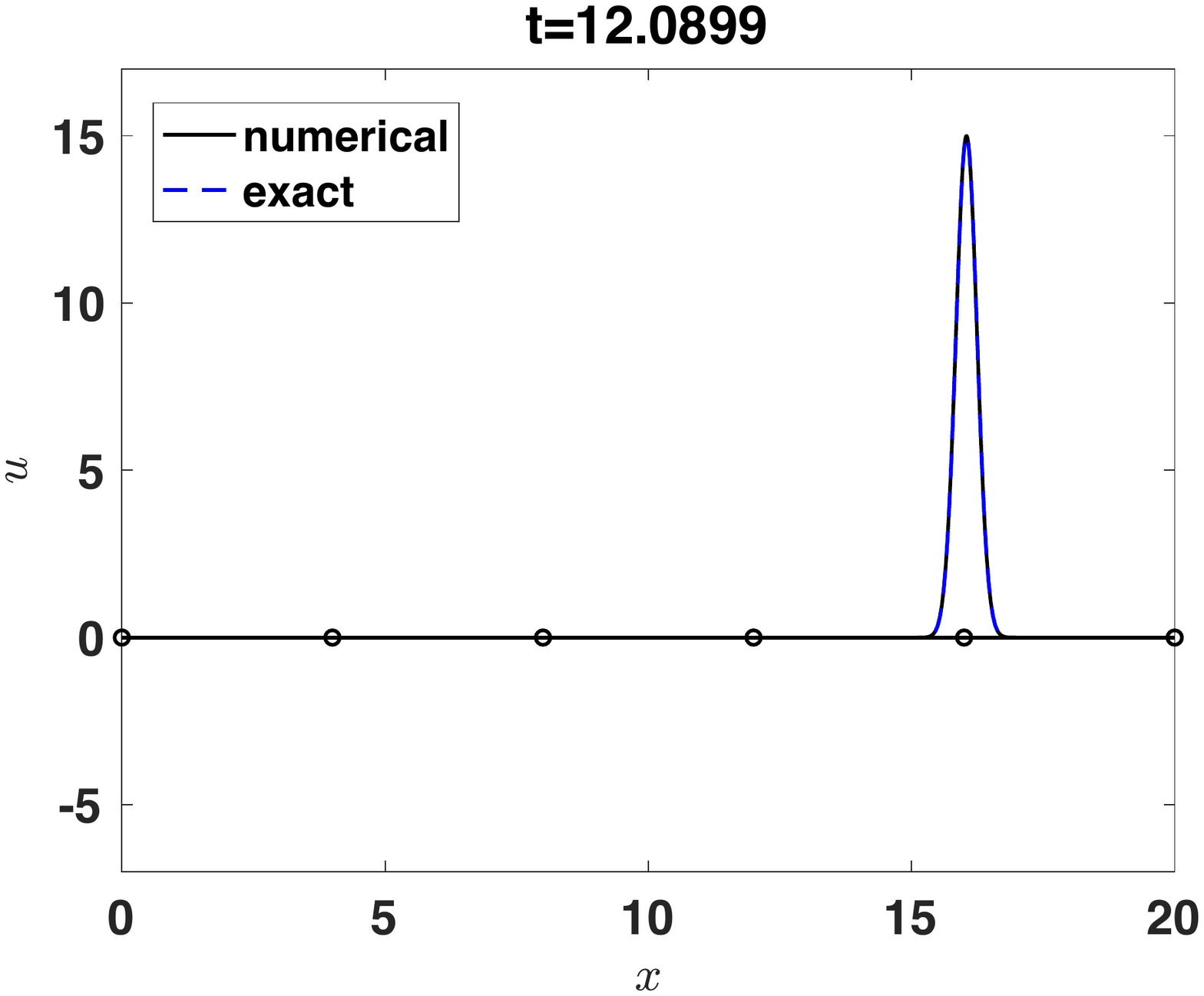}}
    \caption{Snapshots of the numerical and exact solution $u(x,t)$ for a supersonic regime for $a = -1, b = 0.25$.}
    \label{fig:supersonic_u_1D}
\end{figure}
\begin{figure}[h!]
{\includegraphics[width=0.245\textwidth]{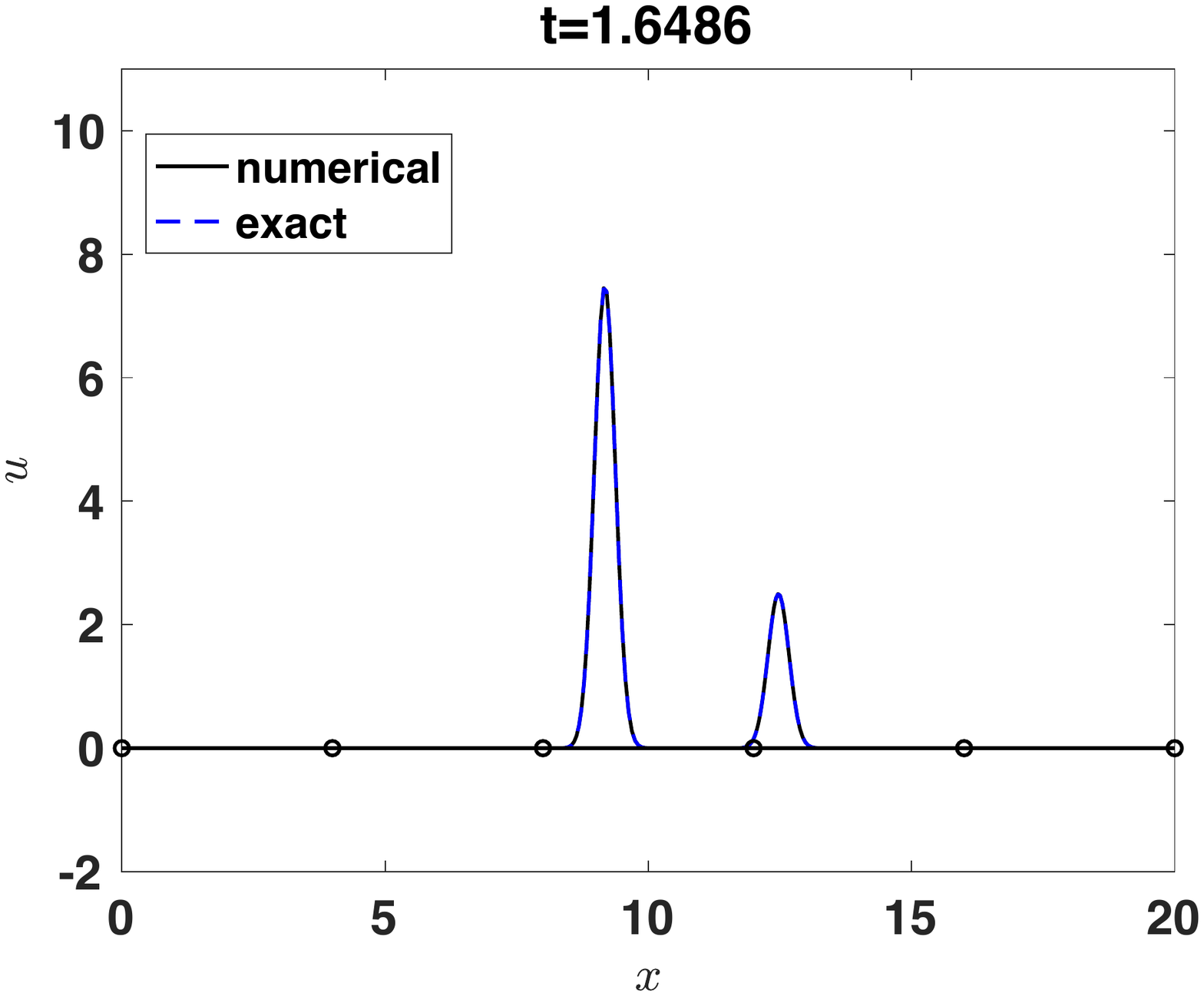}}
{\includegraphics[width=0.245\textwidth]{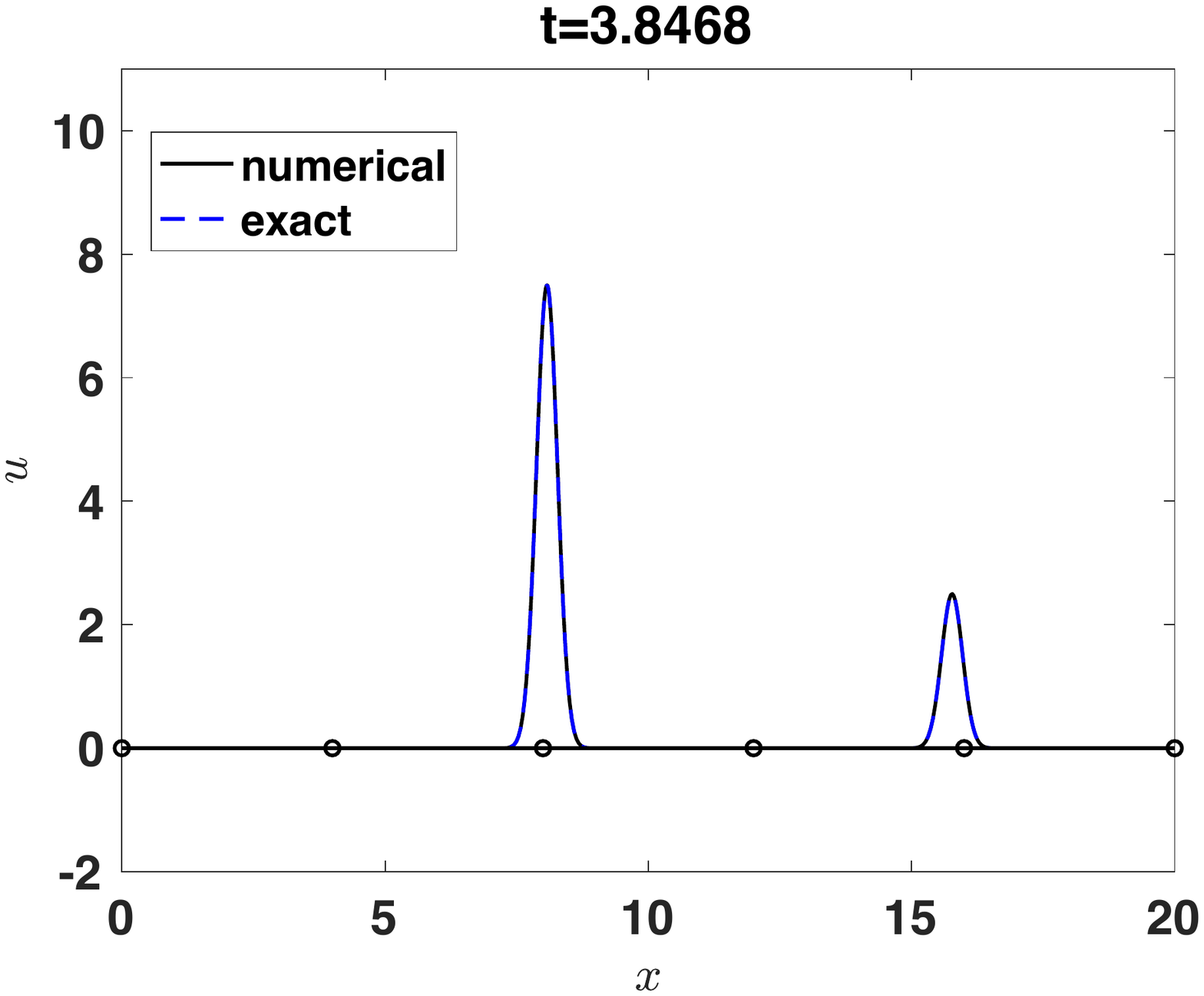}}
{\includegraphics[width=0.245\textwidth]{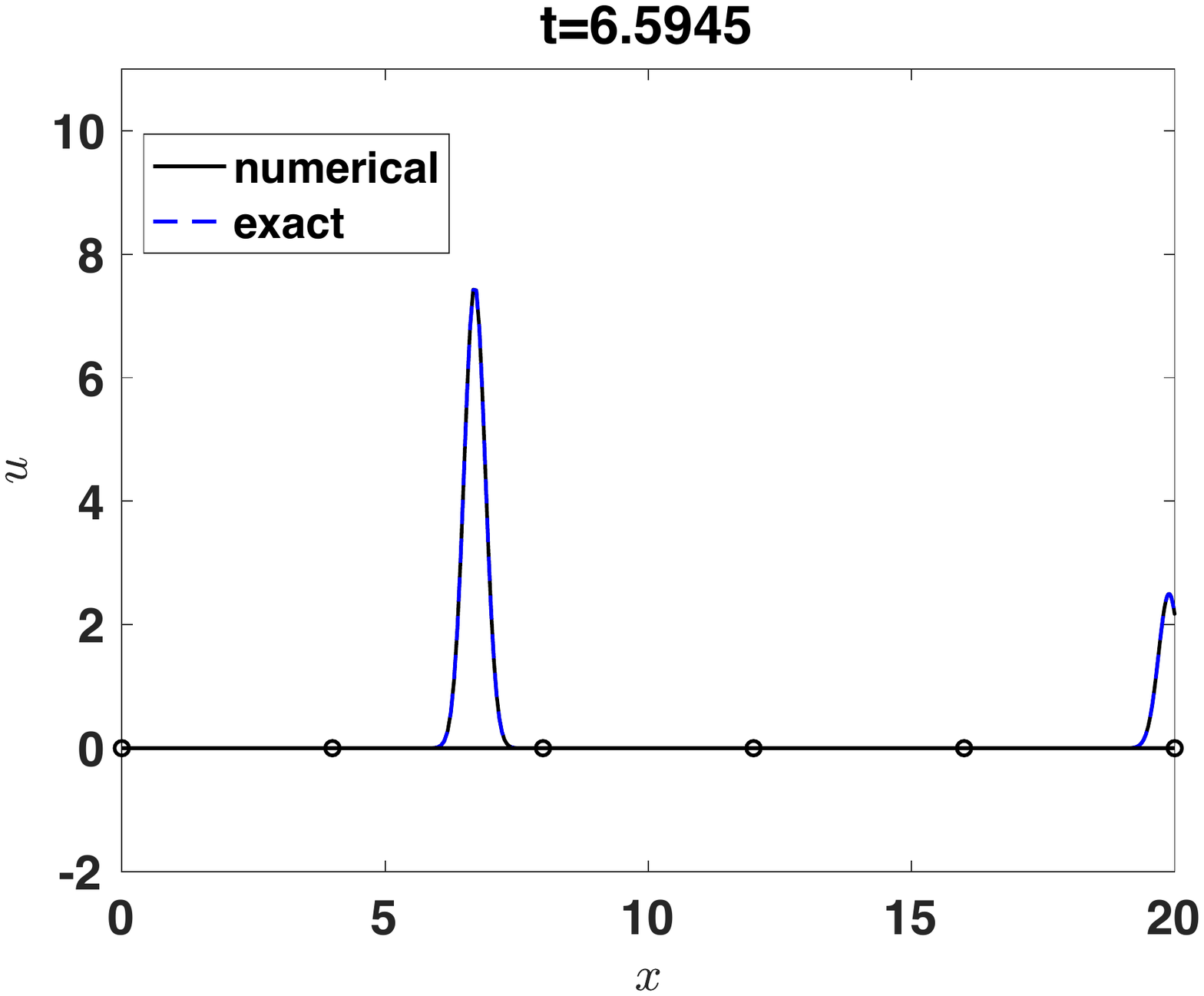}}
{\includegraphics[width=0.245\textwidth]{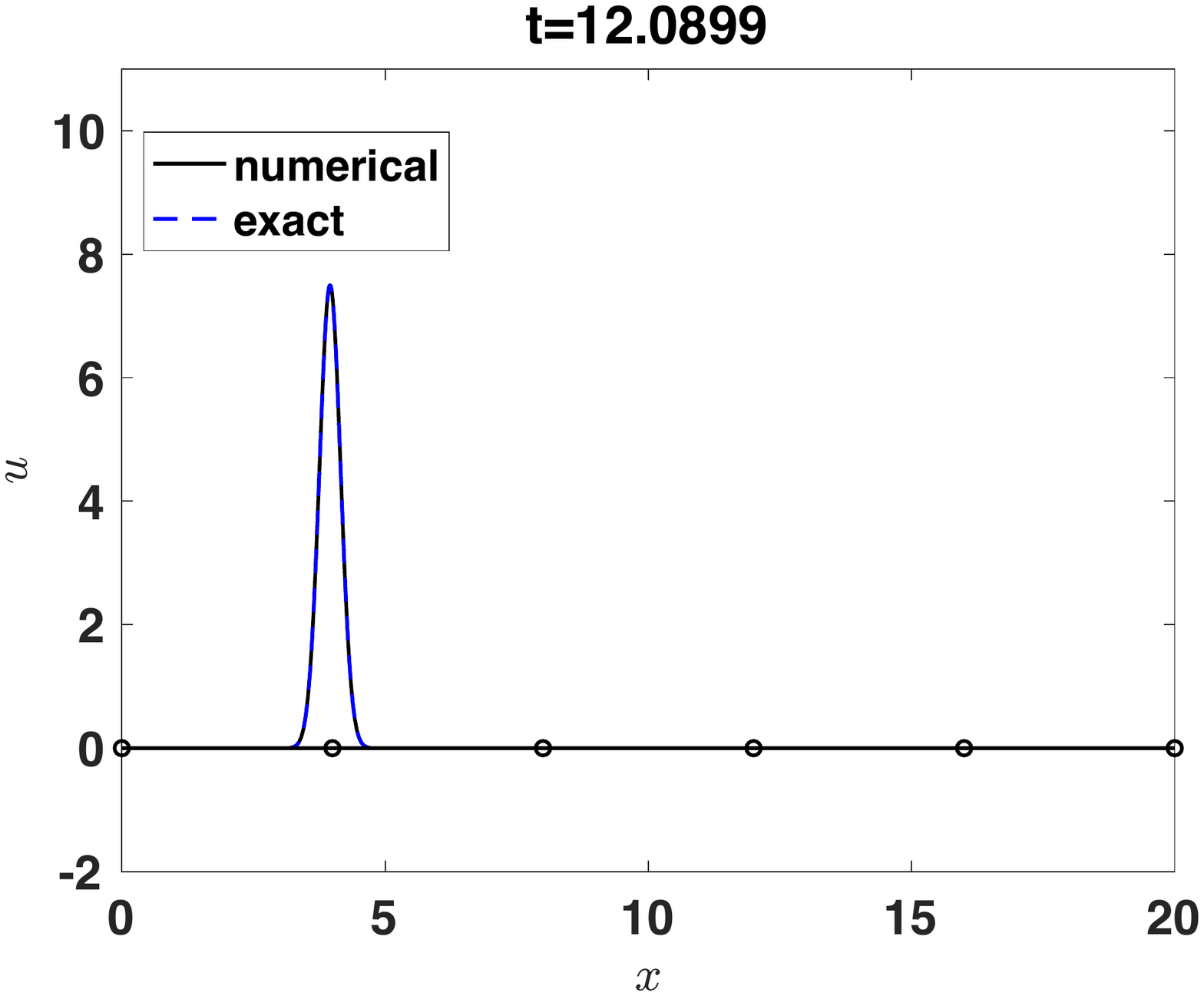}}
    \caption{Snapshots of the numerical and exact solution $u(x,t)$ for a subsonic regime for $a= -0.5, b = 1$.}
    \label{fig:subsonic_u_1D}
\end{figure}

\begin{figure}[H]
\begin{subfigure}{.49\textwidth}
\centering
\includegraphics[width=\linewidth]{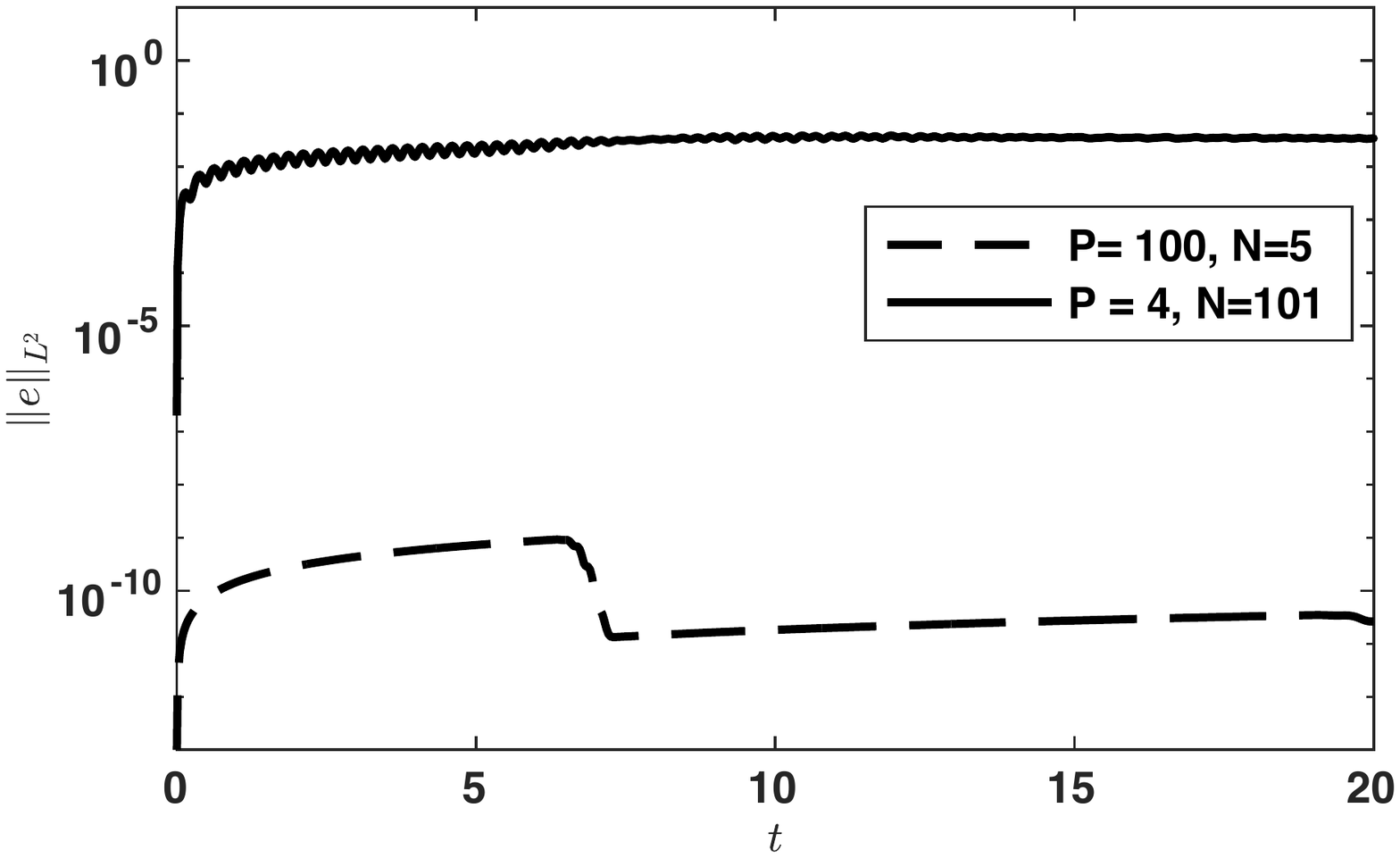}
\caption{Subsonic}
\end{subfigure}
\hfill
\begin{subfigure}{.49\textwidth}
\centering
\includegraphics[width=\linewidth]{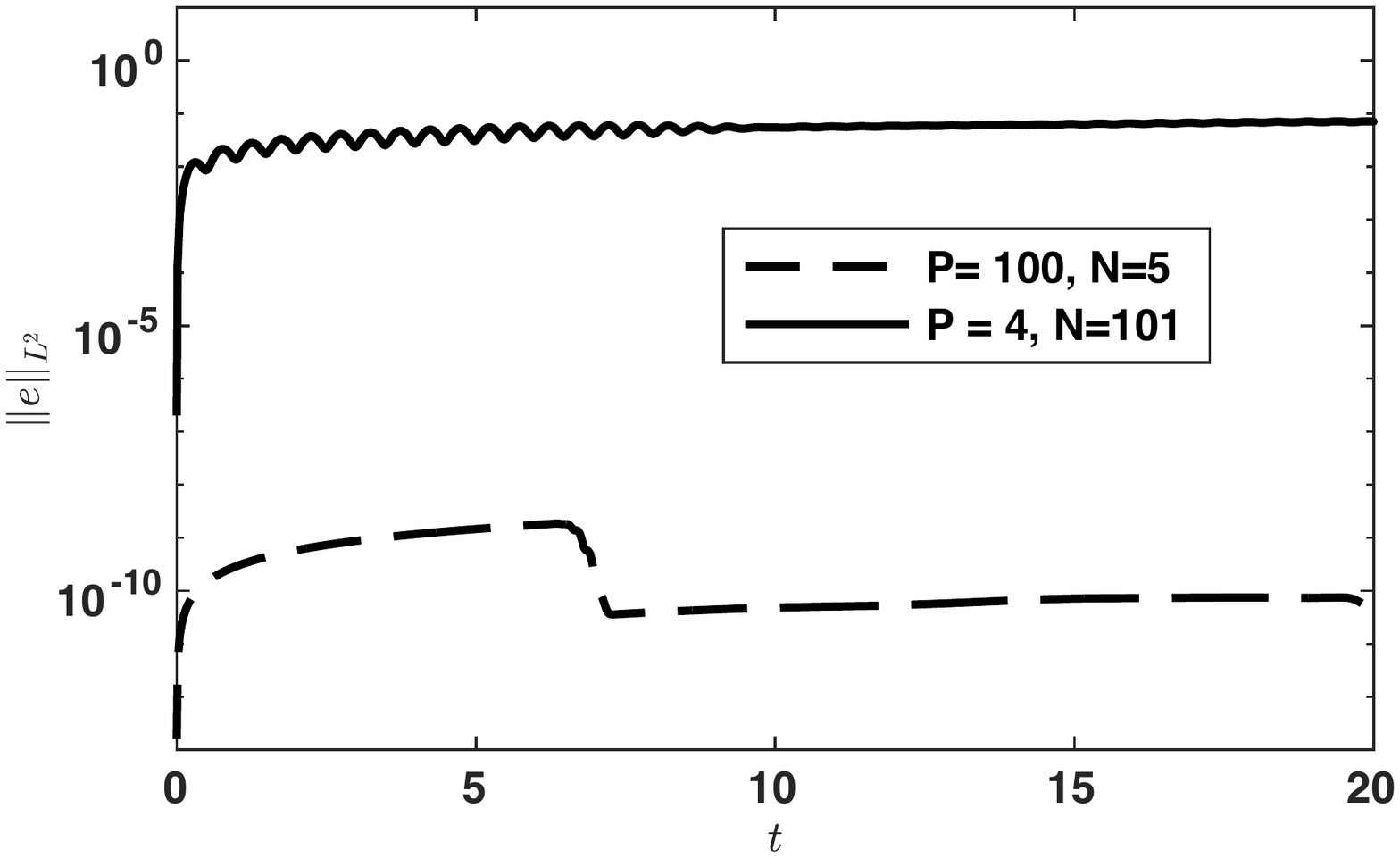}
\caption{Supersonic}
\end{subfigure}
\caption{Evolution of the numerical errors for  $P = 4, 100$ degrees polynomial approximation and $\mathrm{DoF} = 505$.} 
\label{conservationplot_00}
\end{figure}

\begin{figure}[H]
\begin{subfigure}{.49\textwidth}
\centering
\includegraphics[width=\linewidth]{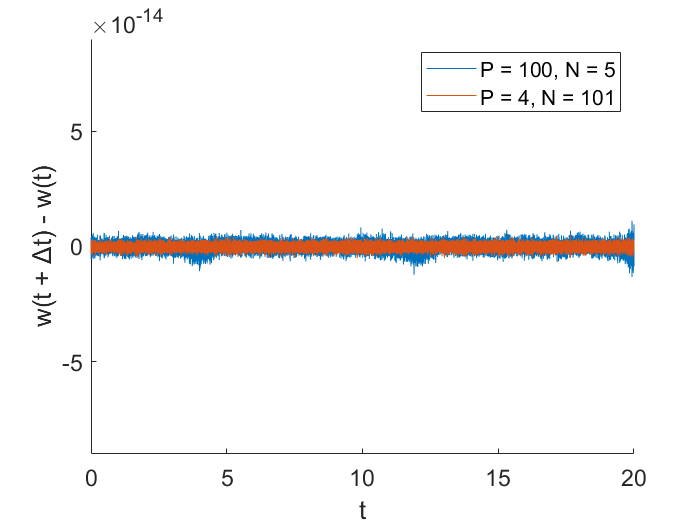}
\caption{Subsonic}
\end{subfigure}
\hfill
\begin{subfigure}{.49\textwidth}
\centering
\includegraphics[width=\linewidth]{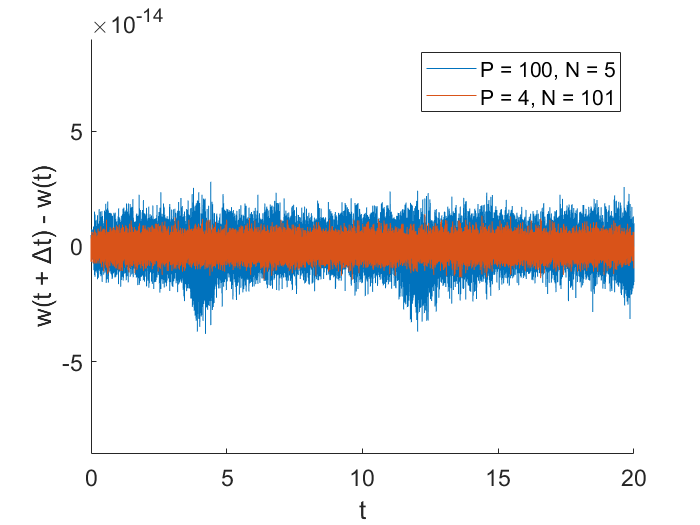}
\caption{Supersonic}
\end{subfigure}
\caption{Discrete conservation principle.} \label{conservationplot}
\end{figure}

\paragraph{Accuracy and convergence of numerical errors}
We investigate the convergence rate of our method by solving a problem similar to the model problem in Sec. 4.1 in \cite{Zhang2019}. Specifically, we solve \eqref{1deqn} in the unit interval $x\in [0,1]$ and choose the exact solution
\begin{equation}\label{exact_sol}
u(x,t) = \cos(2\sqrt{b}\pi t)\sin(2\pi(x+at)), \quad t \geq 0,
\end{equation}
where $a,b$ are the material constants. We consider three sets of the material constants $a=-0.5$,  $b=1$; $a = -1$, $b=0.25$; $a=-0.5$, $b = 0.25$,  corresponding to $c=b-a^2$ positive, negative and equal to zero, respectively. Similar to the previous numerical experiment, we refer to these three cases as subsonic $(c > 0)$, supersonic $(c <0)$ and sonic $(c = 0)$ regimes. Depending on the type of the regime, we have the IBVP \eqref{1deqn} with either \eqref{eq:BC_case1} or \eqref{eq:BC_case3} as its boundary condition with nonzero boundary data, $g_1(t), g_2(t)$, generated by the exact solution \eqref{exact_sol}. We also consider the case where we impose the periodic boundary condition $u(0,t) = u(1,t)$. These two types of boundary conditions will be considered separately in our numerical experiments.

We discretise the unit interval domain into uniform elements of size $\Delta{x}$, where $\Delta{x} = 0.2, 0.1$, $0.05, 0.025$, and we construct the spectral difference operators for polynomial degree $P = 1,2,3,4,5,6$. The time-step is computed using
\begin{equation} \label{timestepeq}
\Delta t = \frac{\mathrm{CFL}}{(2P + 1)(1 + \sqrt{2})\max(|\lambda_1|,|\lambda_2|)} \Delta{x}, \quad \mathrm{CFL} = 0.1,
\end{equation}
 and we compute the $L_2$ error $ \|e\|_{L^2}$ at the final time $T = 0.4$. The convergence plots are depicted in Figures \ref{convplot1}-\ref{convplot2}, and the convergence rates are illustrated in Tables \ref{convtable}-\ref{convtable2}.  It can be seen that the proposed method is $P$th order accurate when $P$ is odd and $(P+1)$th order accurate when $P$ is even for the periodic boundary condition case. Similar convergence rates are observed for the IBVP model, except in the supersonic regime. In this regime, the IBVP model convergence rates are $P$ and $P-1$ for even and odd $P$, respectively.

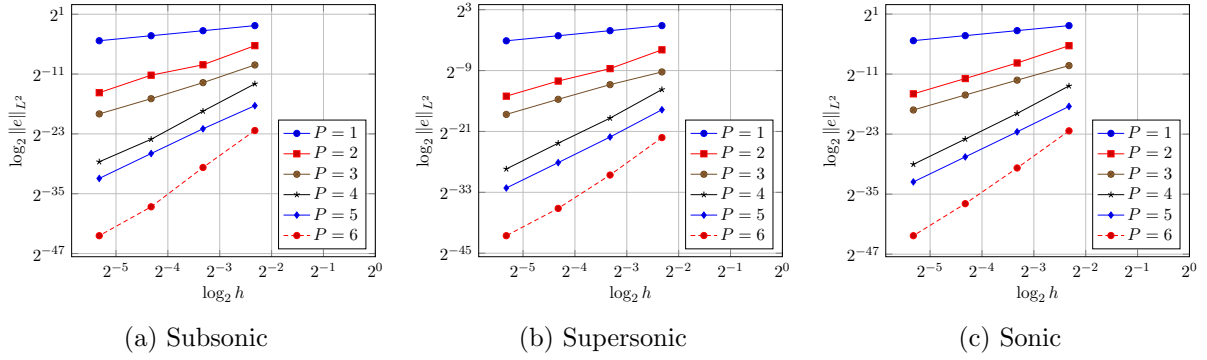
\begin{figure}[H]
     \centering
     \begin{subfigure}[b]{0.32\textwidth}
         \centering
\resizebox{\textwidth}{!}{
\begin{tikzpicture}
\begin{loglogaxis}[
xlabel={$\log_2 h$},
ylabel={$\log_2 \|e\|_{L^2}$},
  log basis y = {2},
  log basis x = {2},
  legend entries = {$P = 1$, $P = 2$, $P = 3$, $P = 4$, $P  = 5$, $P = 6$},
  legend pos=south east,
  grid=major,
  xmax = 1,
]
\addplot table[x index=0,y index=1,col sep=comma] {data1.csv};
\addplot table[x index=0,y index=2,col sep=comma] {data1.csv};
\addplot table[x index=0,y index=3,col sep=comma] {data1.csv};
\addplot table[x index=0,y index=4,col sep=comma] {data1.csv};
\addplot table[x index=0,y index=5,col sep=comma] {data1.csv};
\addplot table[x index=0,y index=6,col sep=comma] {data1.csv};
\end{loglogaxis}
\end{tikzpicture}
}
\caption{Subsonic}
     \end{subfigure}
     \begin{subfigure}[b]{0.32\textwidth}
         \centering
\resizebox{\textwidth}{!}{
\begin{tikzpicture}
\begin{loglogaxis}[
xlabel={$\log_2 h$},
ylabel={$\log_2 \|e\|_{L^2}$},
  log basis y = {2},
  log basis x = {2},
  legend entries = {$P = 1$, $P = 2$, $P = 3$, $P = 4$, $P  = 5$, $P = 6$},
  legend pos=south east,
  grid=major,
  xmax = 1,
]
\addplot table[x index=0,y index=1,col sep=comma] {data2.csv};
\addplot table[x index=0,y index=2,col sep=comma] {data2.csv};
\addplot table[x index=0,y index=3,col sep=comma] {data2.csv};
\addplot table[x index=0,y index=4,col sep=comma] {data2.csv};
\addplot table[x index=0,y index=5,col sep=comma] {data2.csv};
\addplot table[x index=0,y index=6,col sep=comma] {data2.csv};
\end{loglogaxis}
\end{tikzpicture}
}
\caption{Supersonic}
     \end{subfigure}
     \begin{subfigure}[b]{0.32\textwidth}
         \centering
\resizebox{\textwidth}{!}{
\begin{tikzpicture}
\begin{loglogaxis}[
xlabel={$\log_2 h$},
ylabel={$\log_2 \|e\|_{L^2}$},
  log basis y = {2},
  log basis x = {2},
  legend entries = {$P = 1$, $P = 2$, $P = 3$, $P = 4$, $P  = 5$, $P = 6$},
  legend pos=south east,
  grid=major,
  xmax = 1,
]
\addplot table[x index=0,y index=1,col sep=comma] {data3.csv};
\addplot table[x index=0,y index=2,col sep=comma] {data3.csv};
\addplot table[x index=0,y index=3,col sep=comma] {data3.csv};
\addplot table[x index=0,y index=4,col sep=comma] {data3.csv};
\addplot table[x index=0,y index=5,col sep=comma] {data3.csv};
\addplot table[x index=0,y index=6,col sep=comma] {data3.csv};
\end{loglogaxis}
\end{tikzpicture}
}
\caption{Sonic}
     \end{subfigure}
        \caption{Convergence plots for periodic boundary conditions}
        \label{convplot1}
\end{figure}

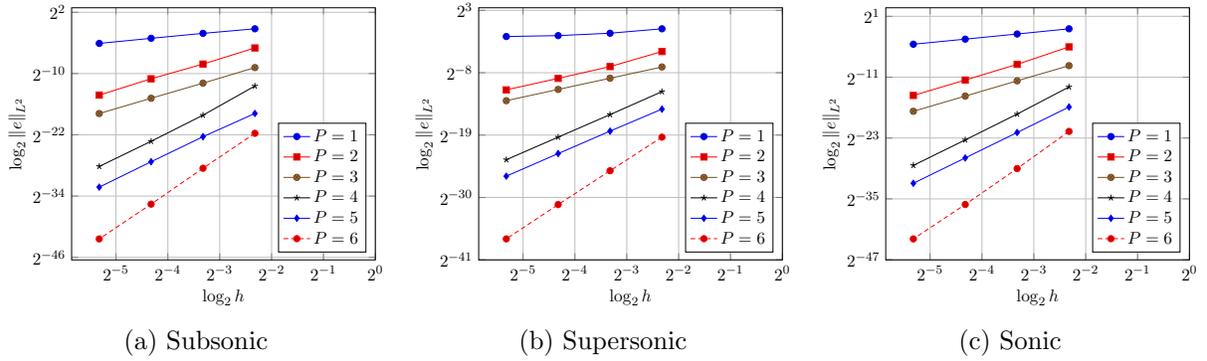
\begin{figure}[H]
     \centering
     \begin{subfigure}[b]{0.32\textwidth}
         \centering
\resizebox{\textwidth}{!}{
\begin{tikzpicture}
\begin{loglogaxis}[
xlabel={$\log_2 h$},
ylabel={$\log_2 \|e\|_{L^2}$},
  log basis y = {2},
  log basis x = {2},
  legend entries = {$P = 1$, $P = 2$, $P = 3$, $P = 4$, $P  = 5$, $P = 6$},
  legend pos=south east,
  grid=major,
  xmax = 1,
]
\addplot table[x index=0,y index=1,col sep=comma, row sep=newline] {data4.csv};
\addplot table[x index=0,y index=2,col sep=comma, row sep=newline] {data4.csv};
\addplot table[x index=0,y index=3,col sep=comma, row sep=newline] {data4.csv};
\addplot table[x index=0,y index=4,col sep=comma, row sep=newline] {data4.csv};
\addplot table[x index=0,y index=5,col sep=comma, row sep=newline] {data4.csv};
\addplot table[x index=0,y index=6,col sep=comma, row sep=newline] {data4.csv};
\end{loglogaxis}
\end{tikzpicture}
}
\caption{Subsonic}
     \end{subfigure}
     \begin{subfigure}[b]{0.32\textwidth}
         \centering
\resizebox{\textwidth}{!}{
\begin{tikzpicture}
\begin{loglogaxis}[
xlabel={$\log_2 h$},
ylabel={$\log_2 \|e\|_{L^2}$},
  log basis y = {2},
  log basis x = {2},
  legend entries = {$P = 1$, $P = 2$, $P = 3$, $P = 4$, $P  = 5$, $P = 6$},
  legend pos=south east,
  grid=major,
  xmax = 1,
]
\addplot table[x index=0,y index=1,col sep=comma, row sep=newline] {data5.csv};
\addplot table[x index=0,y index=2,col sep=comma, row sep=newline] {data5.csv};
\addplot table[x index=0,y index=3,col sep=comma, row sep=newline] {data5.csv};
\addplot table[x index=0,y index=4,col sep=comma, row sep=newline] {data5.csv};
\addplot table[x index=0,y index=5,col sep=comma, row sep=newline] {data5.csv};
\addplot table[x index=0,y index=6,col sep=comma, row sep=newline] {data5.csv};
\end{loglogaxis}
\end{tikzpicture}
}
\caption{Supersonic}
     \end{subfigure}
     \begin{subfigure}[b]{0.32\textwidth}
         \centering
\resizebox{\textwidth}{!}{
\begin{tikzpicture}
\begin{loglogaxis}[
xlabel={$\log_2 h$},
ylabel={$\log_2 \|e\|_{L^2}$},
  log basis y = {2},
  log basis x = {2},
  legend entries = {$P = 1$, $P = 2$, $P = 3$, $P = 4$, $P  = 5$, $P = 6$},
  legend pos=south east,
  grid=major,
  xmax = 1,
]
\addplot table[x index=0,y index=1,col sep=comma, row sep=newline] {data6.csv};
\addplot table[x index=0,y index=2,col sep=comma, row sep=newline] {data6.csv};
\addplot table[x index=0,y index=3,col sep=comma, row sep=newline] {data6.csv};
\addplot table[x index=0,y index=4,col sep=comma, row sep=newline] {data6.csv};
\addplot table[x index=0,y index=5,col sep=comma, row sep=newline] {data6.csv};
\addplot table[x index=0,y index=6,col sep=comma, row sep=newline] {data6.csv};
\end{loglogaxis}
\end{tikzpicture}
}
\caption{Sonic}
     \end{subfigure}
        \caption{Convergence plots for the initial boundary value problem}
        \label{convplot2}
\end{figure}

  \begin{table}[H]
    \begin{subtable}[h]{0.45\textwidth}
        \centering
       \begin{tabular}{|c|c|c|c|}
\hline
$P$ & Subsonic & Supersonic & Sonic \\ \hline
1   & 1.0094         & 0.9981          & 1.0020          \\ \hline
2   & 3.0389         & 2.9842          & 3.1983          \\ \hline
3   & 3.2662         & 2.7999          & 2.9590          \\ \hline
4   & 5.2559         & 5.1906          & 5.2144          \\ \hline
5   & 4.8735         & 5.1331          & 5.0301          \\ \hline
6   & 7.1174         & 6.4600           & 7.0015          \\ \hline
\end{tabular}
\caption{Periodic boundary condition} \label{convtable}
    \end{subtable}
    \hfill
    \begin{subtable}[h]{0.45\textwidth}
        \centering
\begin{tabular}{|c|c|c|c|}
\hline
$P$ & Subsonic & Supersonic & Sonic \\ \hline
1   & 0.9554         & 0.4510          & 1.0139          \\ \hline
2   & 3.0587         & 2.2361          & 3.1827          \\ \hline
3   & 2.9950         & 1.9813          & 3.0007          \\ \hline
4   & 5.2353         & 4.0003          & 5.1496          \\ \hline
5   & 4.8206         & 3.9429          & 5.0130          \\ \hline
6   & 6.9049         & 5.9784           & 7.0642          \\ \hline
\end{tabular}
\caption{Initial boundary value problem} \label{convtable2}
     \end{subtable}
     \caption{Convergence rates for the 1D model problem}
     \label{convtable1d}
\end{table}

\subsection{Numerical examples in two space dimensions}
We consider the 2D shifted wave equation 
\begin{align}\label{eq:2deqn}
\frac{\partial}{\partial t}\left(\frac{\partial u}{\partial t}- \mathbf{a}\cdot\grad u\right)-\div\left(\mathbf{a}\left(\frac{\partial u}{\partial t}- \mathbf{a}\cdot\grad u\right) + {b}\grad u\right) = f, \quad (x, y) \in \Omega = [0, L_x]\times [0, L_y], \quad t \ge 0,
\end{align}
with
$$
\mathbf{a} = \left(a_x, a_y\right)^T, \quad 
{b} = c_p^2 >0,
$$
where $a_x, a_y \in \mathbb{R}$ are the components of the background velocity and $c_p$ parametrises the speed of sound. 
We set the smooth initial condition
\begin{align}\label{eq:2deqn_init_condition}
 \frac{\partial }{\partial t}u(x, y,0) = v_0(x, y), \quad u(x, y, 0) = u_0(x,y).
\end{align}
 The boundary conditions will be determined by the background velocity $\mathbf{a}$ and the sound speed $c_p$. We will consider specifically a medium with the background velocity $\mathbf{a}$ parallel to the vertical axis,  that is $\mathbf{a}=\left(0, a_y\right)^T$.  As before the boundary condition will be determined by $c= b- a_y^2$.  As above we will consider the supersonic $c < 0$ and subsonic $c > 0$ regimes separately.
\paragraph{Supersonic regime}
We consider the parameters
\begin{align}
a_x = 0, \quad a_y = -1, \quad  b = 0.25,
\end{align}
thus having $c =b-a_y^2 = -0.75 <0$.
In the $x$-direction we set hard wall boundary conditions, that is the normal derivatives vanish  at the boundaries $x = 0, L_x$.
Note that $y = 0$ is an inflow boundary and $y = L_y$ is an outflow boundary. And since $c < 0$, there are two boundary conditions at $y =0$ and  no boundary conditions at $y = L_y$.
We summarise the boundary conditions below
\begin{align}\label{eq:BC_case2_2D_supersonic}
&\frac{\partial u}{\partial x} = g_{x_1}(y,t), \quad x = 0, \qquad  \frac{\partial u}{\partial x} =  g_{x_2}(y,t), \quad x = L_x,\\
&u = f_{y_1}(x,t), \quad \frac{\partial u}{\partial t} - a_y\frac{\partial u}{\partial y} = g_{y_1}(x,t), \quad y = 0.
\end{align}
\begin{figure}[h!]
{\includegraphics[width=0.23\textwidth,angle=90]{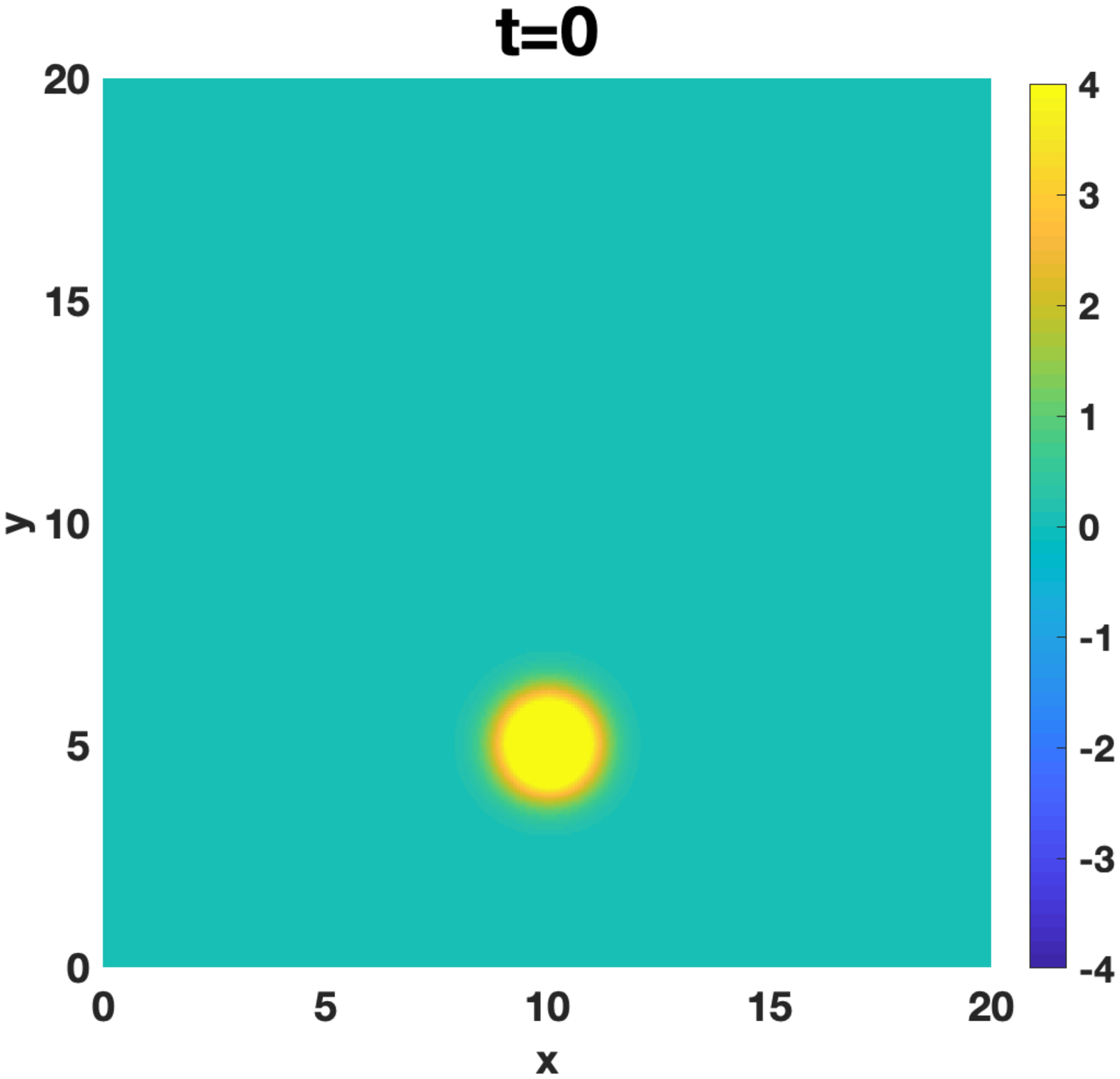}}
{\includegraphics[width=0.23\textwidth,angle=90]{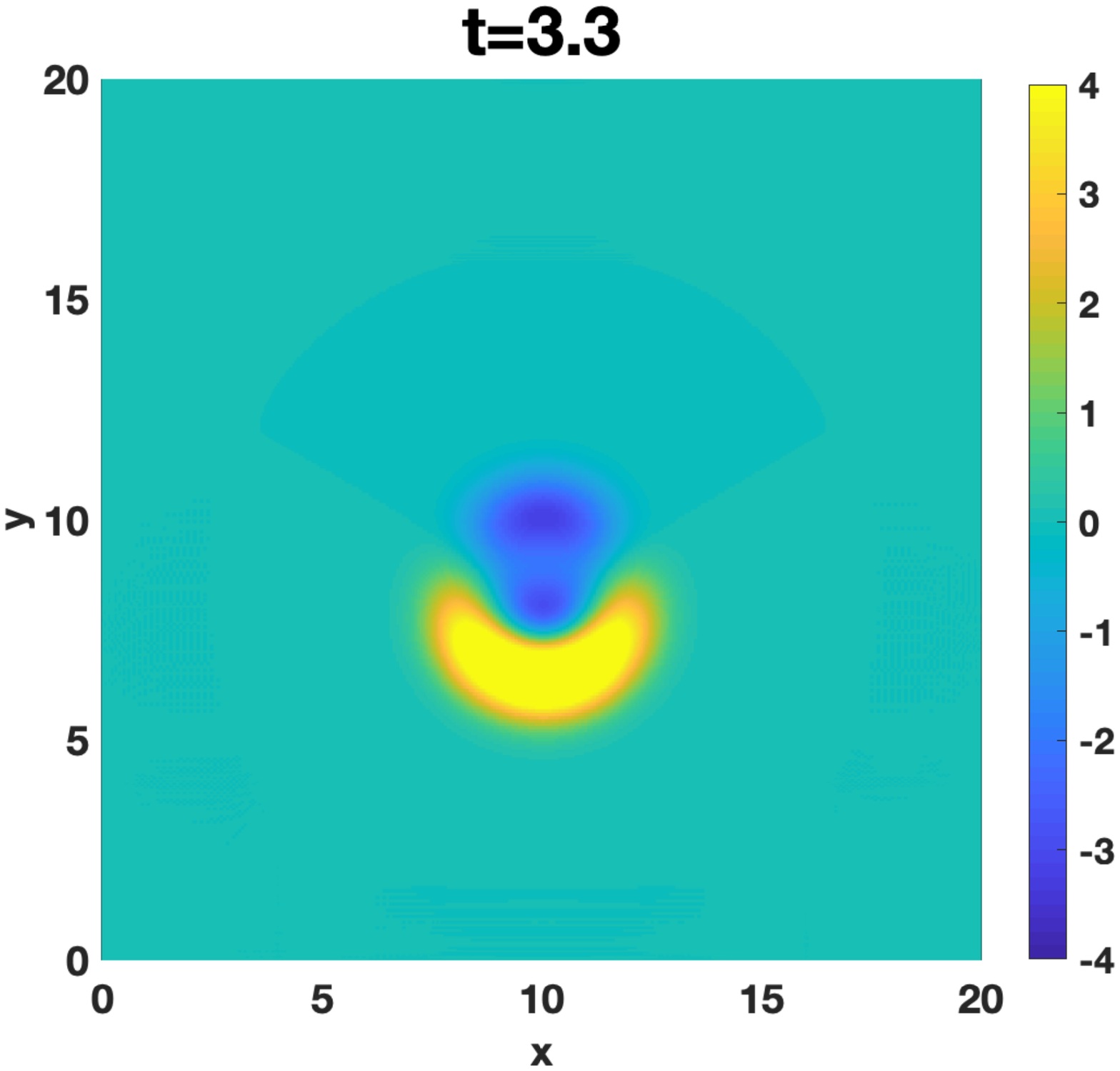}}
{\includegraphics[width=0.23\textwidth,angle=90]{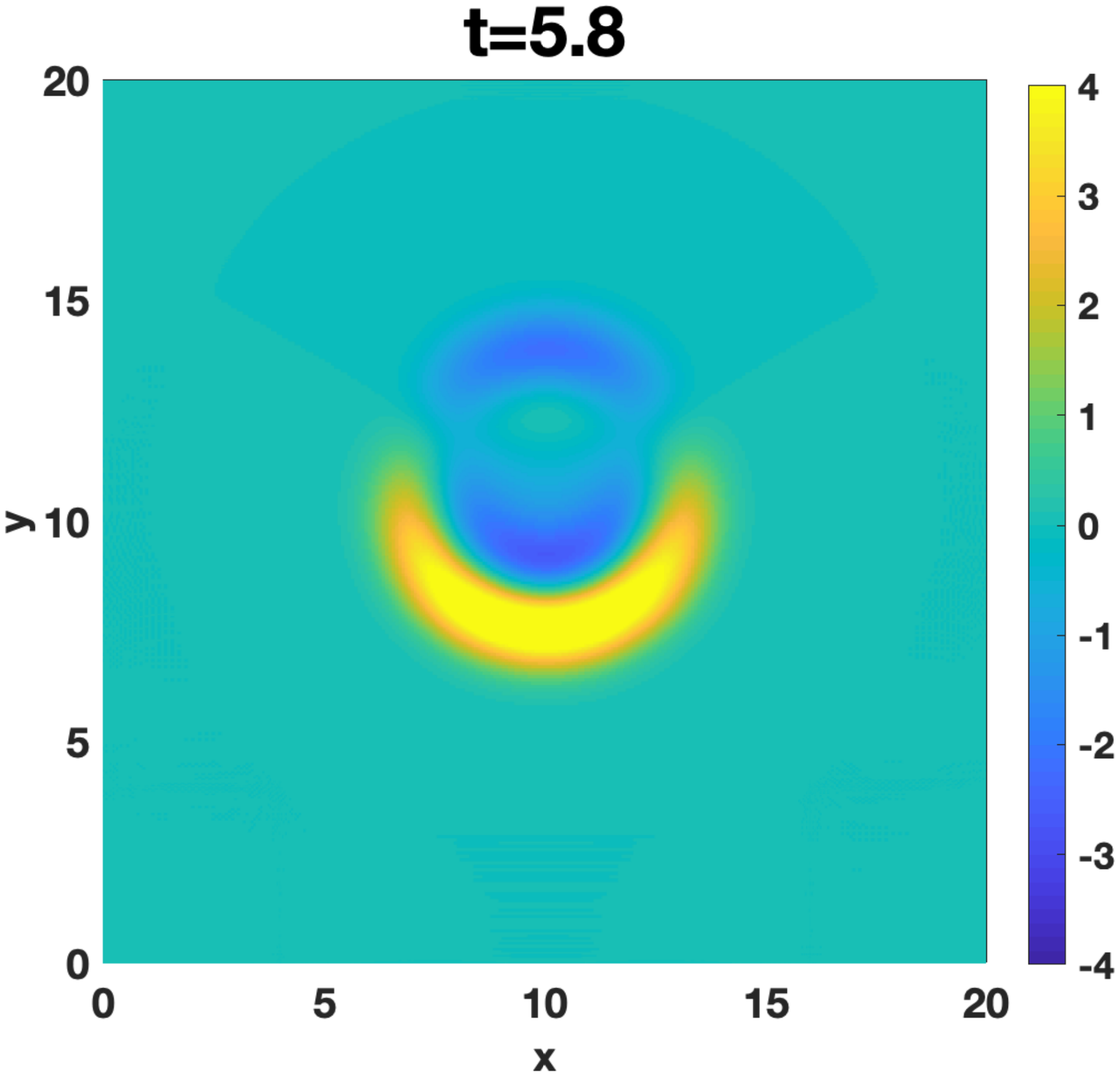}}
{\includegraphics[width=0.23\textwidth,angle=90]{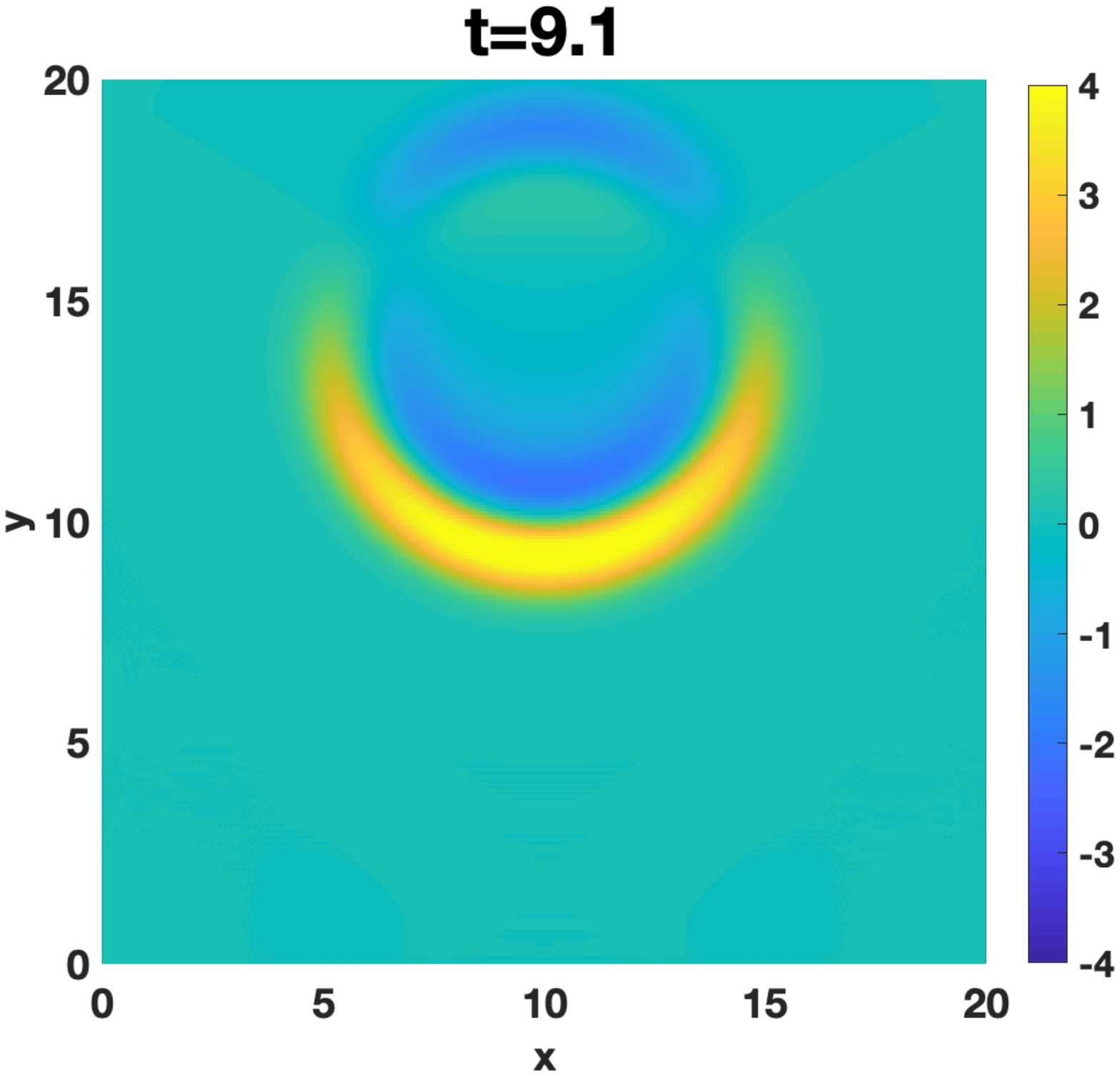}}
{\includegraphics[width=0.23\textwidth,angle=90]{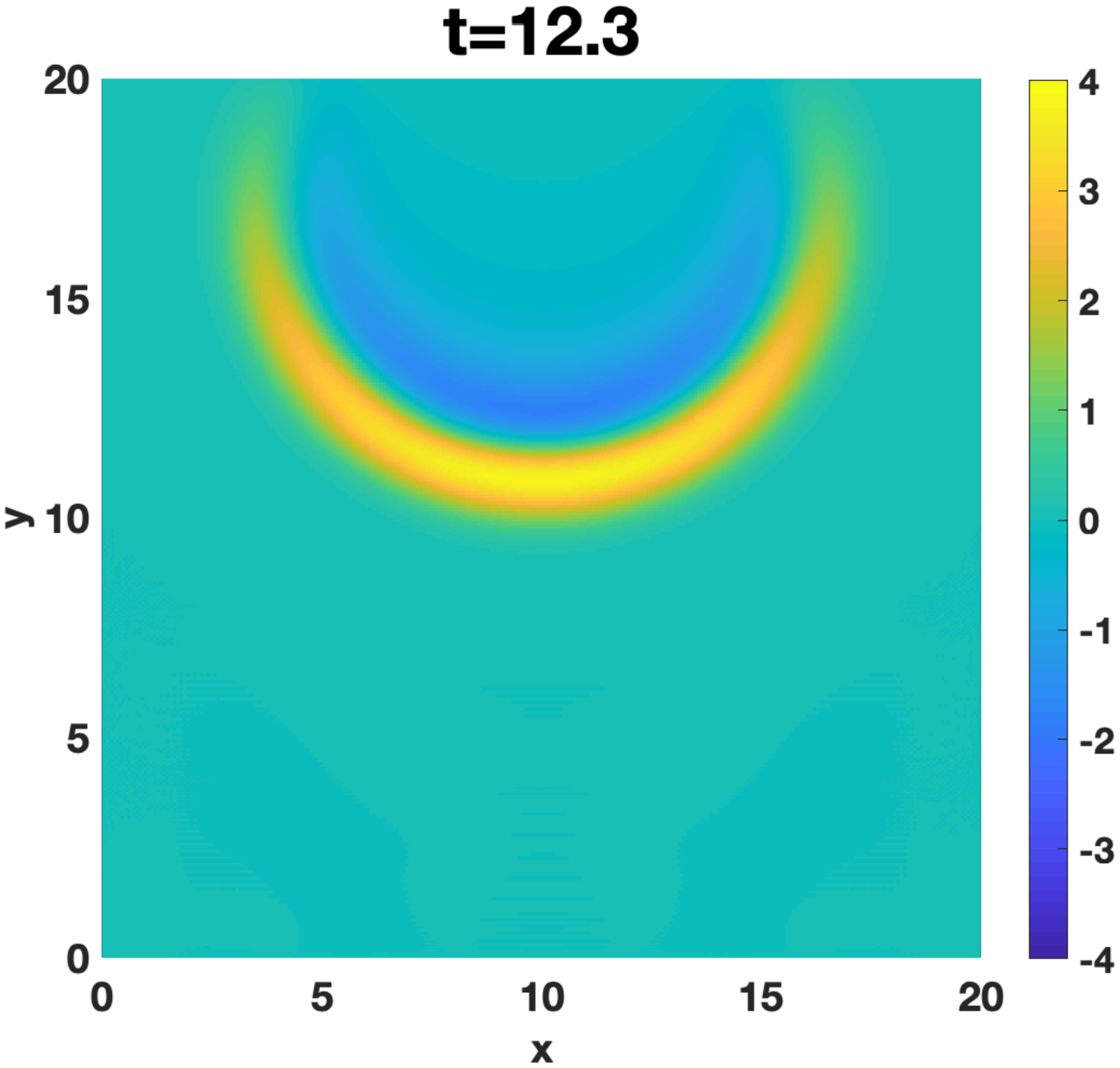}}
{\includegraphics[width=0.23\textwidth,angle=90]{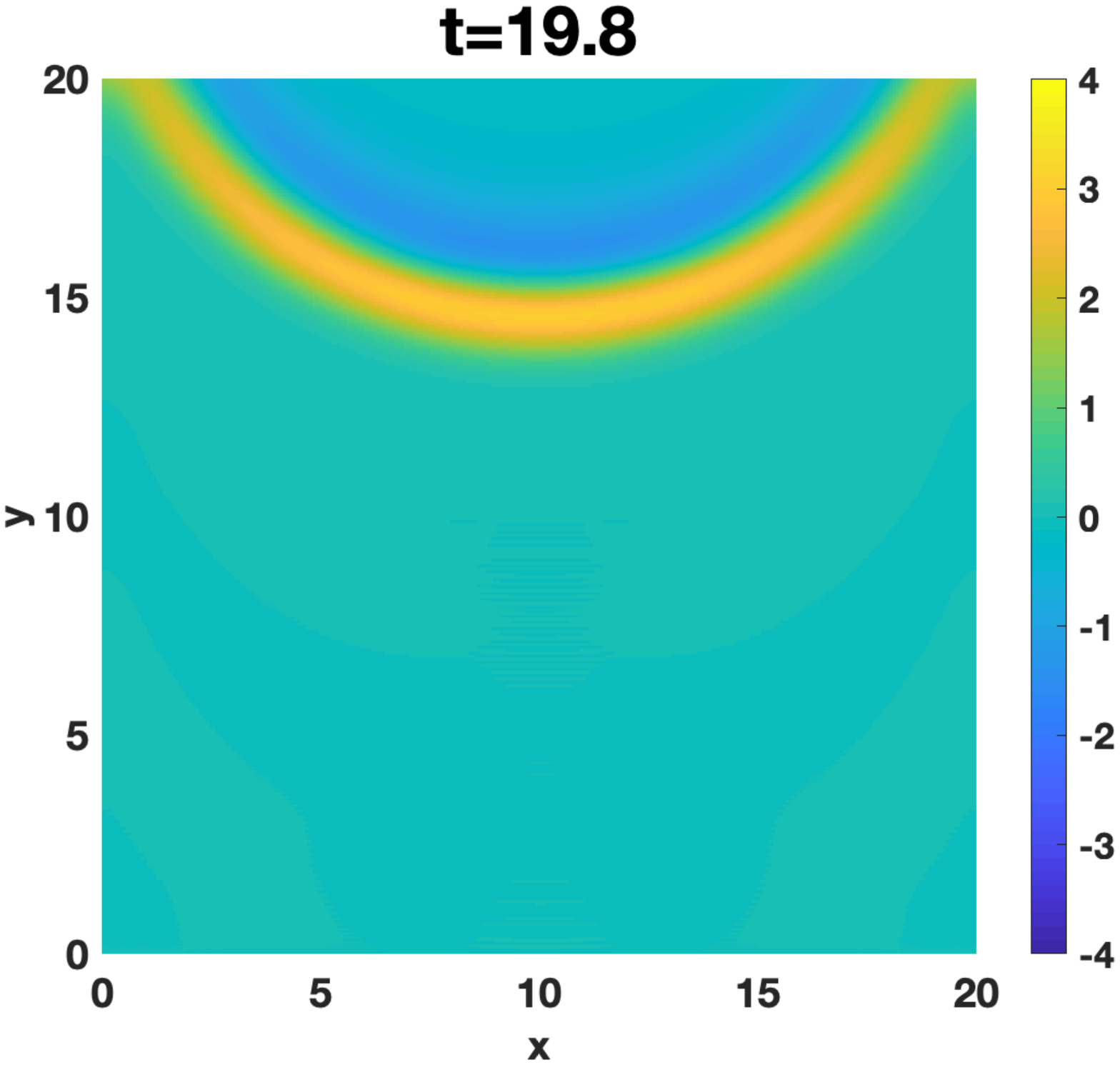}}
    \caption{Snapshots of the wave field $u(x,y,t)$ for a supersonic regime for $a_x = 0, a_y = -1, b = 0.25$, $c = -0.75 <0$.}
    \label{fig:supersonic_u}
\end{figure}
\paragraph{Subsonic regime}
For the subsonic regime, we consider the parameters
\begin{align}
a_x = 0, \quad a_y = -0.5, \quad  b = 1,
\end{align}
thus having $c =b-a_y^2 = 0.75 >0$.
As before, in the $x$-direction we set hard wall boundary conditions, that is the normal derivatives vanish  at the boundaries $x = 0, L_x$.
Note again that $y = 0$ is an inflow boundary and $y = L_y$ is an outflow boundary. However, since $c > 0$, there is one boundary condition at $y =0$ and  one boundary condition at $y = L_y$. We summarise the boundary conditions for the subsonic regime below
\begin{align}\label{eq:BC_case2_2D_subsonic}
&\frac{\partial u}{\partial x} = g_{x_1}(y,t), \quad x = 0, \qquad  \frac{\partial u}{\partial x} =  g_{x_2}(y,t), \quad x = L_x,\\
&\frac{\partial u}{\partial t} - \lambda_1\frac{\partial u}{\partial y}  = g_{y_1}(x,t), \quad y = 0, \qquad \frac{\partial u}{\partial t} - \lambda_2\frac{\partial u}{\partial y}  = g_{y_2}(x,t), \quad y = L_y,
\end{align}
where $\lambda_1 = a_y + \sqrt{b}$ and $\lambda_2 = a_y - \sqrt{b}$.

We discretise the domain with tensor product elements of width $\Delta{x} = L_x/M_x$ and $\Delta{y} = L_y/M_y$, where $M_x$, $M_y$ are number of elements in $x$- and $y$-coordinates, and approximate the solutions  using degree $P$ Lagrange polynomials. The time-step is determined by \eqref{timestepeq}.

To begin, we consider a medium that is initially at rest with the initial Gaussian profile centred at $(x_0, y_0)=(10, 5)$, that is
\begin{align}\label{eq:2deqn_init_condition_2}
v_0(x,y) =0, \quad u_0(x, y) = 10 e^{-\left((x-x_0)^2 + (y-y_0)^2\right)},
\end{align}
and we set homogeneous boundary data in \eqref{eq:BC_case2_2D_supersonic} and \eqref{eq:BC_case2_2D_subsonic}. We discretise the domain uniformly into 4 elements, with 2 elements in each spatial coordinate, $M_x = M_y = 2$, and consider a tensor product of degree $P = 80$ polynomial approximation. The final time is $T = 20$.
The snapshots of the solutions are shown in Figure \ref{fig:supersonic_u} for the supersonic regime and in Figure \ref{fig:subsonic_u} for the subsonic regime.
\begin{figure}[h!]
{\includegraphics[width=0.23\textwidth,angle=90]{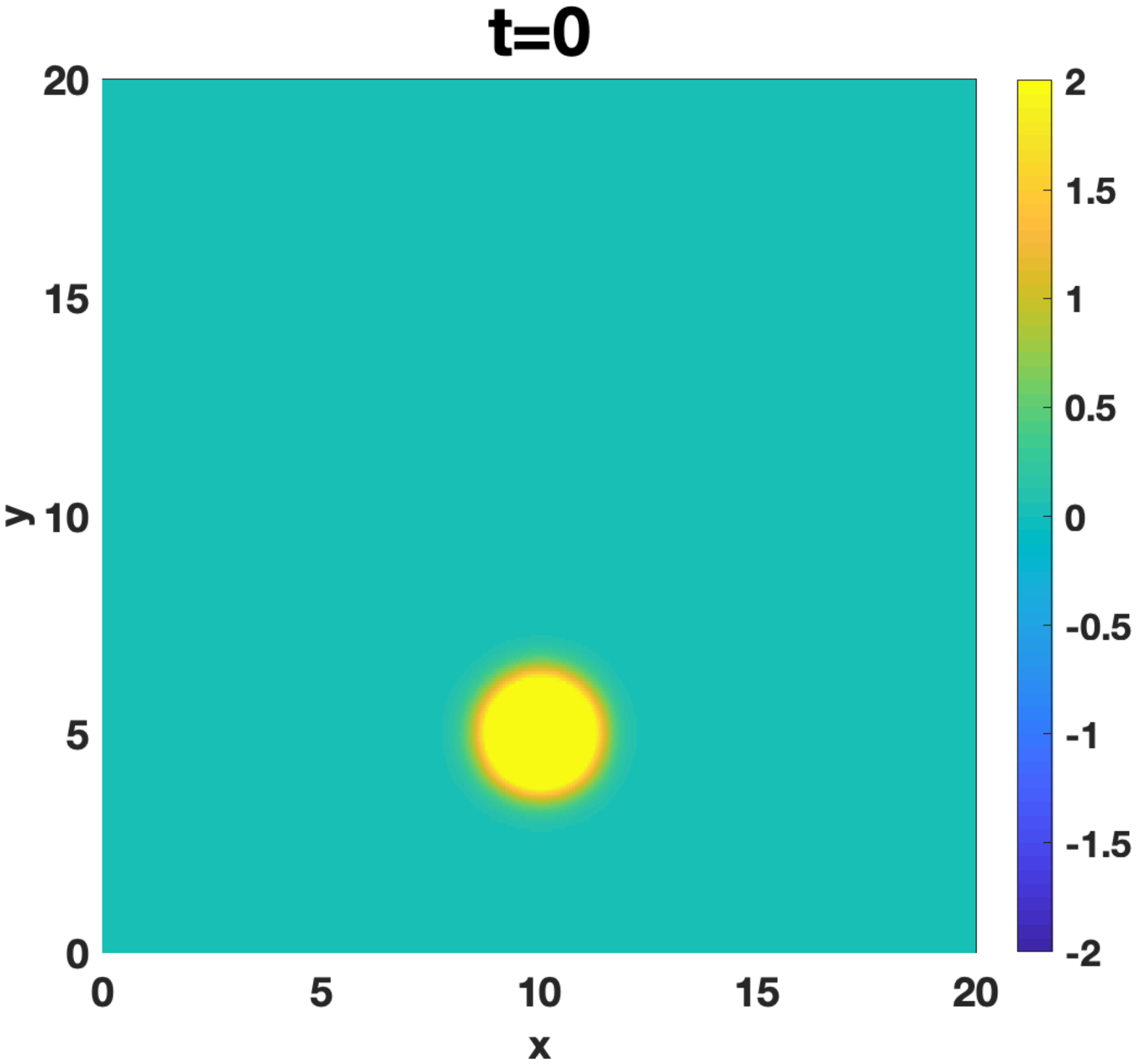}}
{\includegraphics[width=0.23\textwidth,angle=90]{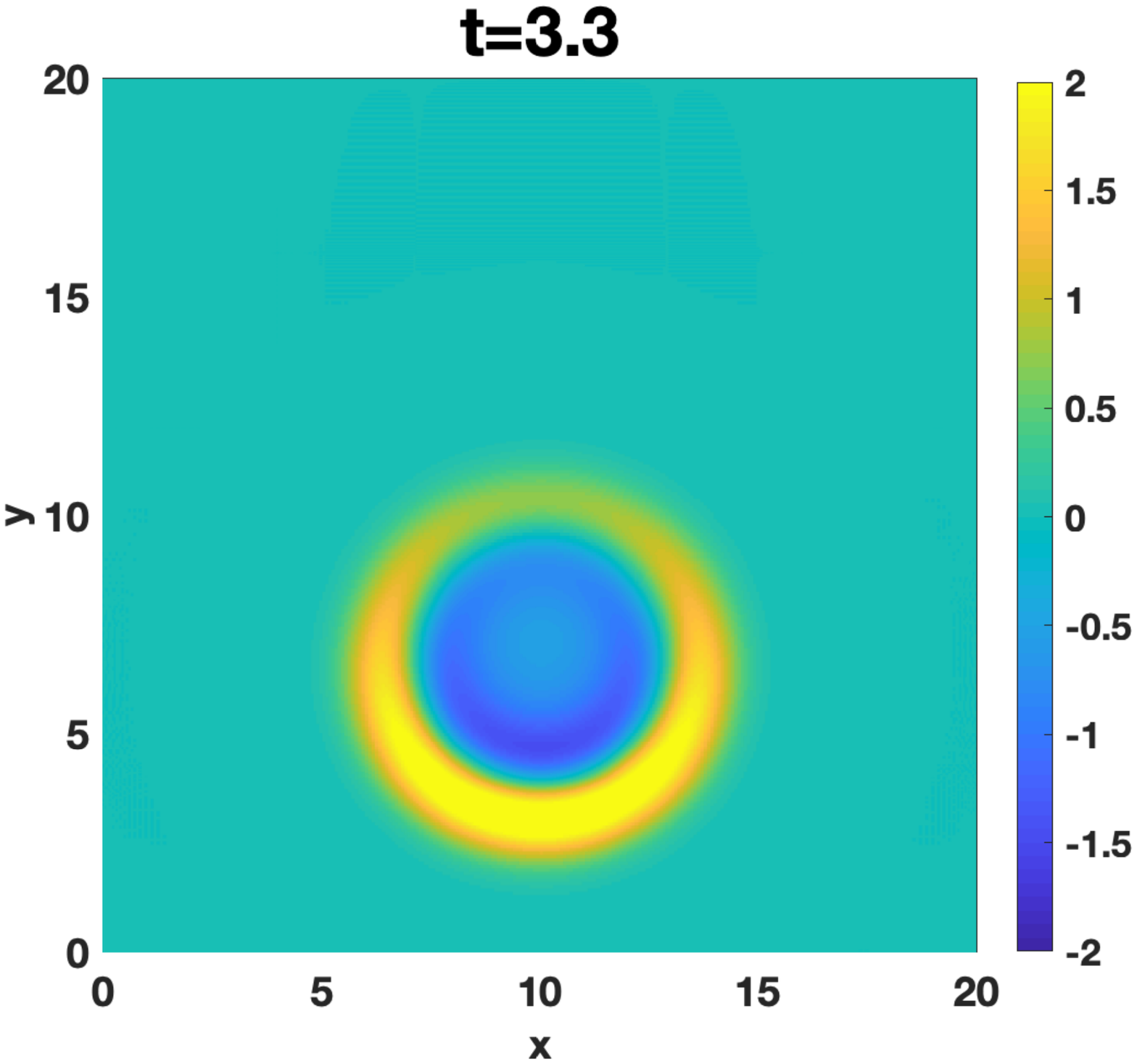}}
{\includegraphics[width=0.23\textwidth,angle=90]{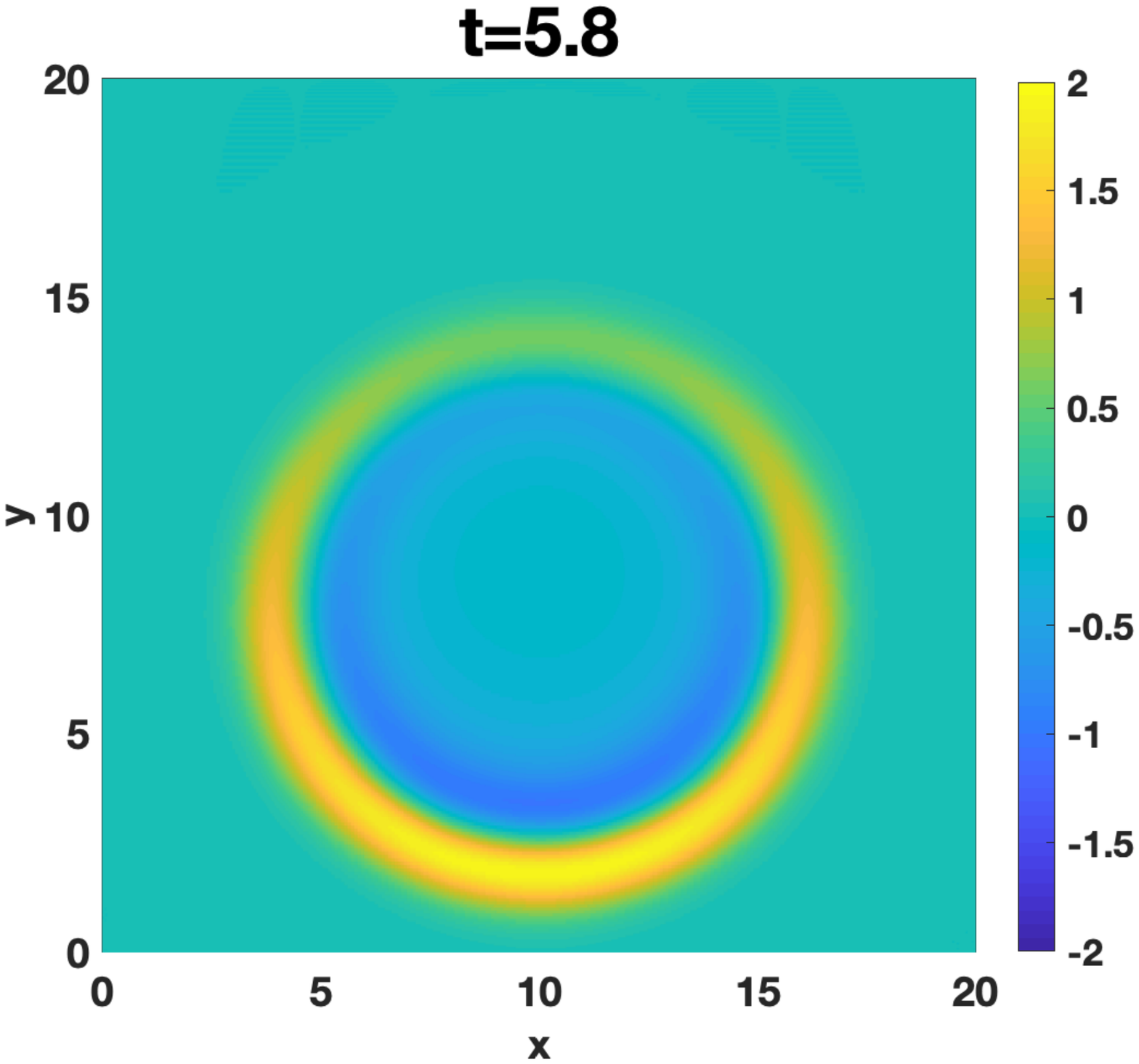}}
{\includegraphics[width=0.23\textwidth,angle=90]{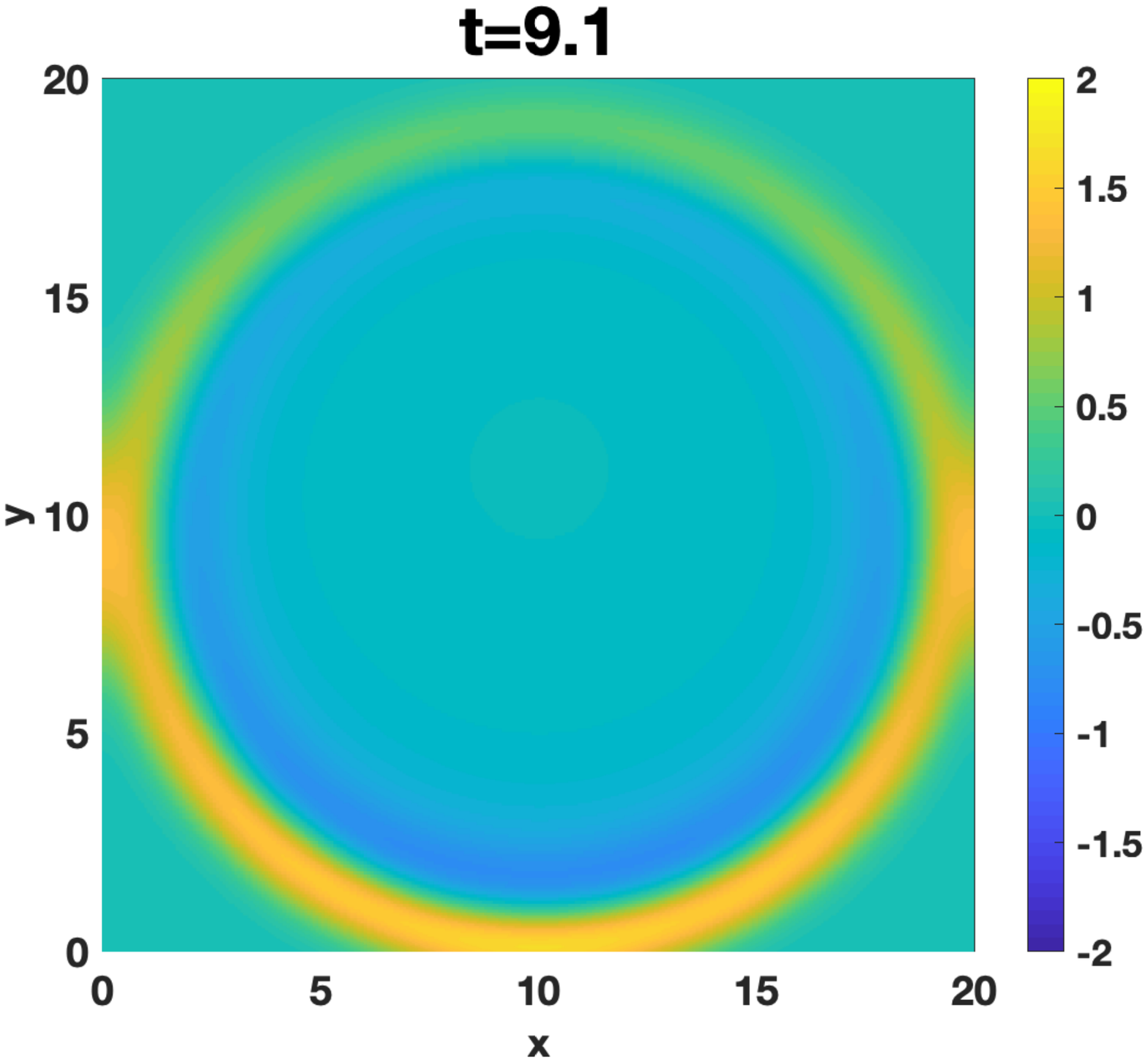}}
{\includegraphics[width=0.23\textwidth,angle=90]{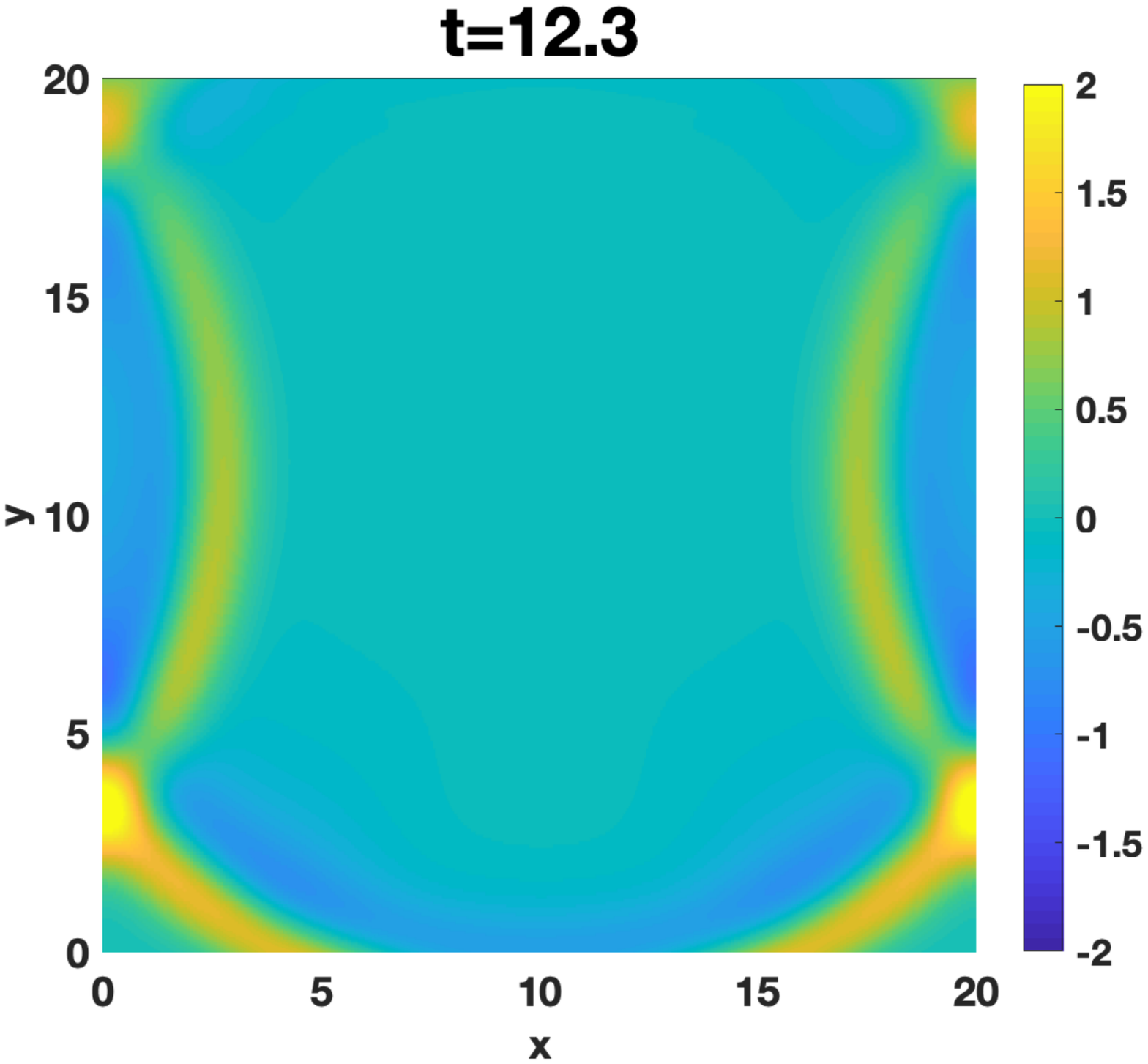}}
{\includegraphics[width=0.23\textwidth,angle=90]{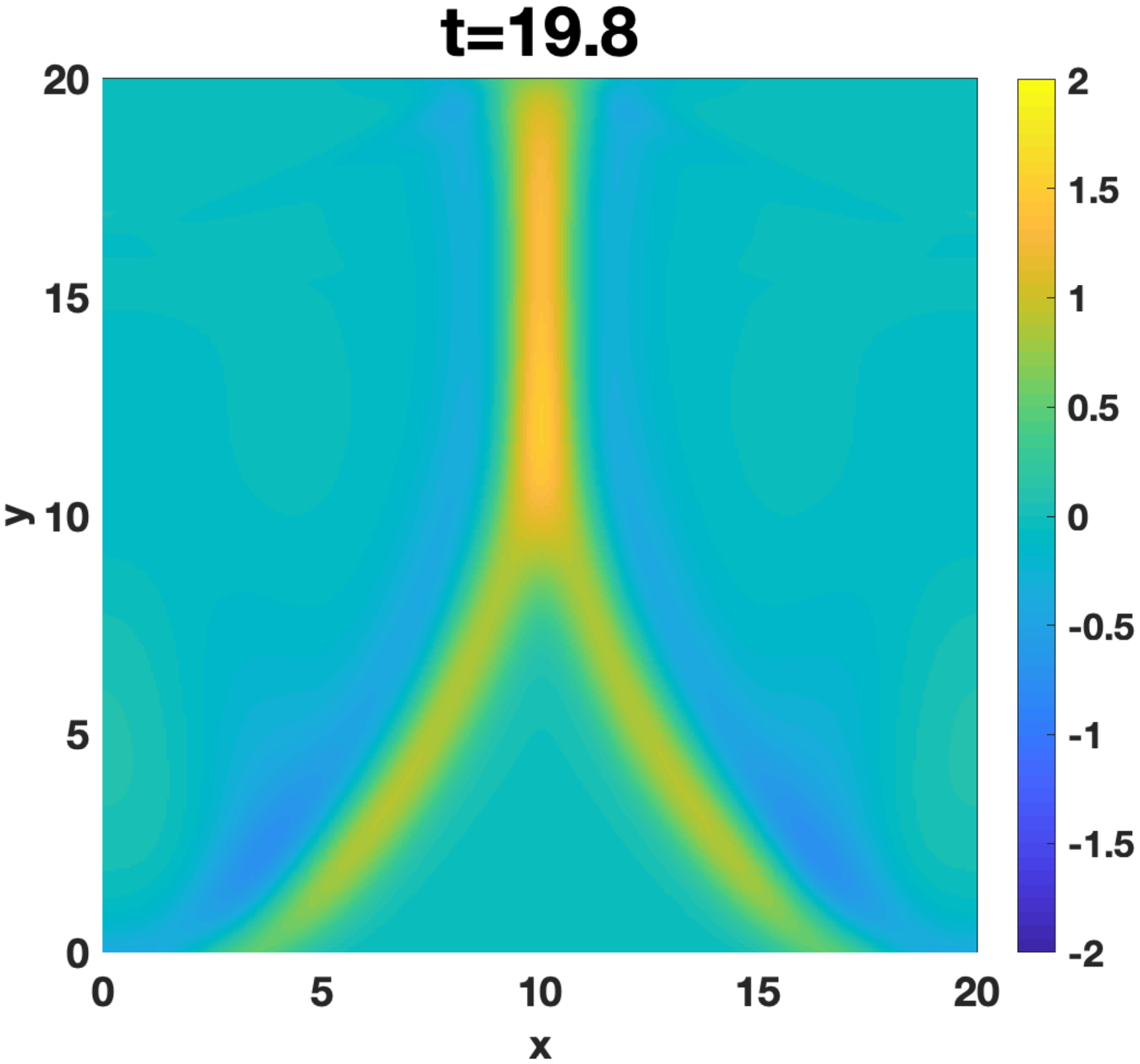}}
    \caption{Snapshots of the wave field $u(x,y,t)$ for a supersonic regime for $a_x = 0, a_y = -0.5, b = 1$, $c = 0.75 >0$.}
    \label{fig:subsonic_u}
\end{figure}
\paragraph{Accuracy and convergence}
We will now investigate numerical accuracy and convergence for the 2D model problem. We will consider the supersonic,  subsonic  and sonic flow regimes separately.
In the sonic regime we have $a_x = 0, a_y = -0.5, b = 0.25$, so that $c = 0$.
As  in Sec. 3.1 in \cite{Zhang2019}, we consider the unit square $(x,y) \in \Omega = [0,1]\times[0,1]$ and the smooth exact solution,
\begin{align}\label{eq:2deqn_exact_solution}
u(x,y,t) =  \sin(2 \sqrt{b}\pi t)\Bigl( 
\sin\bigl( 
2\pi (x -a_xt)
\bigr) 
+ \sin
\bigl( 
2\pi (y -a_yt)
\bigr) \Bigr) 
, \quad t \geq  0.
\end{align}
Note that the exact solution \eqref{eq:2deqn_exact_solution} is periodic in $(x,y) \in [0,1]\times[0,1]$, that is 
$u(0,y,t)~=~u(1,y,t)$ and $u(x,0,t)~=~u(x,1,t)$.
The problem can be reformulated as an IBVP \eqref{eq:2deqn} with \eqref{eq:BC_case2_2D_supersonic} or \eqref{eq:BC_case2_2D_subsonic} where \eqref{eq:2deqn_exact_solution} satisfy the inhomogeneous boundary data. As in the 1D case above, we will consider the periodic boundary condition and the IBVP separately.

We discretise the unit square using a sequence of uniform meshes $\Delta{x} = \Delta{y} = 1/M$, for $M = 5, 10, 20, 40, 80$, and consider degree $P = 1, 2, 3, 4, 5, 6$ polynomial approximations. We evolve the solution until the final time $T = 0.4$, and compute the $L_2$ error.
The errors are plotted in Figure \ref{convplot2d_periodic_bc},  for the periodic boundary conditions, and  in Figure \ref{convplot2d_IBVP}  for the IBVP. The convergence rates are displayed in Table \ref{convtable_2D_periodic}, for the periodic boundary conditions, and in Table \ref{convtable_2D_IBVP}, for the IBVP. Note that the errors converge to zero in all settings. For the periodic boundary conditions the convergence is $P+1$, for degree $P$ polynomial approximation in all flow regimes,  subsonic,  sonic and supersonic regimes. For the IBVP model convergence rate is $P$ in the subsonic and sonic regimes, and $P-1$ in the supersonic regime. These are in good agreement with the 1D results obtained above, in the last subsection.
 
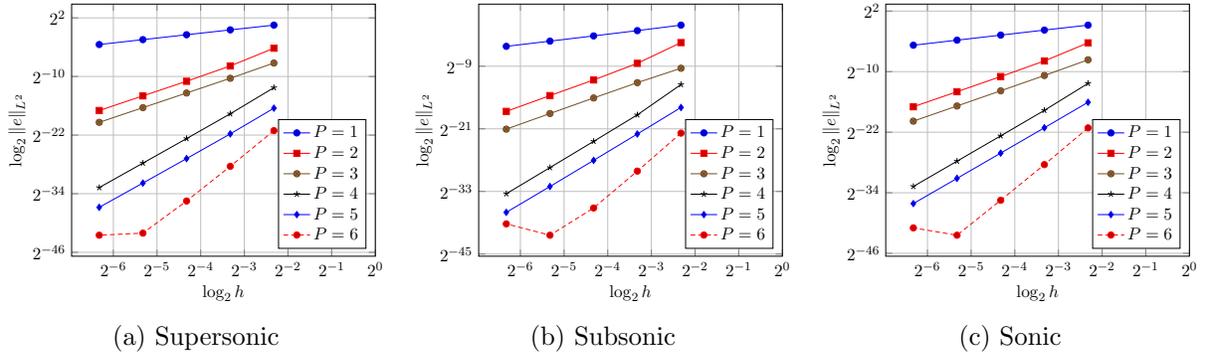
\begin{figure}[H]
     \centering
     \begin{subfigure}[b]{0.32\textwidth}
         \centering
\resizebox{\textwidth}{!}{
\begin{tikzpicture}
\begin{loglogaxis}[
xlabel={$\log_2 h$},
ylabel={$\log_2 \|e\|_{L^2}$},
  log basis y = {2},
  log basis x = {2},
  legend entries = {$P = 1$, $P = 2$, $P = 3$, $P = 4$, $P  = 5$, $P = 6$},
  legend pos=south east,
  grid=major,
  xmax = 1,
]
\addplot table[x index=0,y index=1,col sep=comma, row sep=newline] {supersonic_periodic_2d.csv};
\addplot table[x index=0,y index=2,col sep=comma, row sep=newline] {supersonic_periodic_2d.csv};
\addplot table[x index=0,y index=3,col sep=comma, row sep=newline] {supersonic_periodic_2d.csv};
\addplot table[x index=0,y index=4,col sep=comma, row sep=newline] {supersonic_periodic_2d.csv};
\addplot table[x index=0,y index=5,col sep=comma, row sep=newline] {supersonic_periodic_2d.csv};
\addplot table[x index=0,y index=6,col sep=comma, row sep=newline] {supersonic_periodic_2d.csv};
\end{loglogaxis}
\end{tikzpicture}
}
\caption{Supersonic}
     \end{subfigure}
     \begin{subfigure}[b]{0.32\textwidth}
         \centering
\resizebox{\textwidth}{!}{
\begin{tikzpicture}
\begin{loglogaxis}[
xlabel={$\log_2 h$},
ylabel={$\log_2 \|e\|_{L^2}$},
  log basis y = {2},
  log basis x = {2},
  legend entries = {$P = 1$, $P = 2$, $P = 3$, $P = 4$, $P  = 5$, $P = 6$},
  legend pos=south east,
  grid=major,
  xmax = 1,
]
\addplot table[x index=0,y index=1,col sep=comma, row sep=newline] {subsonic_periodic_2d.csv};
\addplot table[x index=0,y index=2,col sep=comma, row sep=newline] {subsonic_periodic_2d.csv};
\addplot table[x index=0,y index=3,col sep=comma, row sep=newline] {subsonic_periodic_2d.csv};
\addplot table[x index=0,y index=4,col sep=comma, row sep=newline] {subsonic_periodic_2d.csv};
\addplot table[x index=0,y index=5,col sep=comma, row sep=newline] {subsonic_periodic_2d.csv};
\addplot table[x index=0,y index=6,col sep=comma, row sep=newline] {subsonic_periodic_2d.csv};
\end{loglogaxis}
\end{tikzpicture}
}
\caption{Subsonic}
     \end{subfigure}
     \begin{subfigure}[b]{0.32\textwidth}
         \centering
\resizebox{\textwidth}{!}{
\begin{tikzpicture}
\begin{loglogaxis}[
xlabel={$\log_2 h$},
ylabel={$\log_2 \|e\|_{L^2}$},
  log basis y = {2},
  log basis x = {2},
  legend entries = {$P = 1$, $P = 2$, $P = 3$, $P = 4$, $P  = 5$, $P = 6$},
  legend pos=south east,
  grid=major,
  xmax = 1,
]
\addplot table[x index=0,y index=1,col sep=comma, row sep=newline] {sonic_periodic_2d.csv};
\addplot table[x index=0,y index=2,col sep=comma, row sep=newline] {sonic_periodic_2d.csv};
\addplot table[x index=0,y index=3,col sep=comma, row sep=newline] {sonic_periodic_2d.csv};
\addplot table[x index=0,y index=4,col sep=comma, row sep=newline] {sonic_periodic_2d.csv};
\addplot table[x index=0,y index=5,col sep=comma, row sep=newline] {sonic_periodic_2d.csv};
\addplot table[x index=0,y index=6,col sep=comma, row sep=newline] {sonic_periodic_2d.csv};
\end{loglogaxis}
\end{tikzpicture}
}
\caption{Sonic}
     \end{subfigure}
        \caption{Convergence plots for periodic boundary conditions}
        \label{convplot2d_periodic_bc}
\end{figure}

\begin{figure}[H]
     \centering
     \begin{subfigure}[b]{0.32\textwidth}
         \centering
\resizebox{\textwidth}{!}{
\begin{tikzpicture}
\begin{loglogaxis}[
xlabel={$\log_2 h$},
ylabel={$\log_2 \|e\|_{L^2}$},
  log basis y = {2},
  log basis x = {2},
  legend entries = {$P = 1$, $P = 2$, $P = 3$, $P = 4$, $P  = 5$, $P = 6$},
  legend pos=south east,
  grid=major,
  xmax = 1,
]
\addplot table[x index=0,y index=1,col sep=comma, row sep=newline] {supersonic_2d.csv};
\addplot table[x index=0,y index=2,col sep=comma, row sep=newline] {supersonic_2d.csv};
\addplot table[x index=0,y index=3,col sep=comma, row sep=newline] {supersonic_2d.csv};
\addplot table[x index=0,y index=4,col sep=comma, row sep=newline] {supersonic_2d.csv};
\addplot table[x index=0,y index=5,col sep=comma, row sep=newline] {supersonic_2d.csv};
\addplot table[x index=0,y index=6,col sep=comma, row sep=newline] {supersonic_2d.csv};
\end{loglogaxis}
\end{tikzpicture}
}
\caption{Supersonic}
     \end{subfigure}
     \begin{subfigure}[b]{0.32\textwidth}
         \centering
\resizebox{\textwidth}{!}{
\begin{tikzpicture}
\begin{loglogaxis}[
xlabel={$\log_2 h$},
ylabel={$\log_2 \|e\|_{L^2}$},
  log basis y = {2},
  log basis x = {2},
  legend entries = {$P = 1$, $P = 2$, $P = 3$, $P = 4$, $P  = 5$, $P = 6$},
  legend pos=south east,
  grid=major,
  xmax = 1,
]
\addplot table[x index=0,y index=1,col sep=comma, row sep=newline] {subsonic_2d.csv};
\addplot table[x index=0,y index=2,col sep=comma, row sep=newline] {subsonic_2d.csv};
\addplot table[x index=0,y index=3,col sep=comma, row sep=newline] {subsonic_2d.csv};
\addplot table[x index=0,y index=4,col sep=comma, row sep=newline] {subsonic_2d.csv};
\addplot table[x index=0,y index=5,col sep=comma, row sep=newline] {subsonic_2d.csv};
\addplot table[x index=0,y index=6,col sep=comma, row sep=newline] {subsonic_2d.csv};
\end{loglogaxis}
\end{tikzpicture}
}
\caption{Subsonic}
     \end{subfigure}
     \begin{subfigure}[b]{0.32\textwidth}
         \centering
\resizebox{\textwidth}{!}{
\begin{tikzpicture}
\begin{loglogaxis}[
xlabel={$\log_2 h$},
ylabel={$\log_2 \|e\|_{L^2}$},
  log basis y = {2},
  log basis x = {2},
  legend entries = {$P = 1$, $P = 2$, $P = 3$, $P = 4$, $P  = 5$, $P = 6$},
  legend pos=south east,
  grid=major,
  xmax = 1,
]
\addplot table[x index=0,y index=1,col sep=comma, row sep=newline] {sonic_2d.csv};
\addplot table[x index=0,y index=2,col sep=comma, row sep=newline] {sonic_2d.csv};
\addplot table[x index=0,y index=3,col sep=comma, row sep=newline] {sonic_2d.csv};
\addplot table[x index=0,y index=4,col sep=comma, row sep=newline] {sonic_2d.csv};
\addplot table[x index=0,y index=5,col sep=comma, row sep=newline] {sonic_2d.csv};
\addplot table[x index=0,y index=6,col sep=comma, row sep=newline] {sonic_2d.csv};
\end{loglogaxis}
\end{tikzpicture}
}
\caption{Sonic}
     \end{subfigure}
        \caption{Convergence plots for the initial boundary value problem}
        \label{convplot2d_IBVP}
\end{figure}
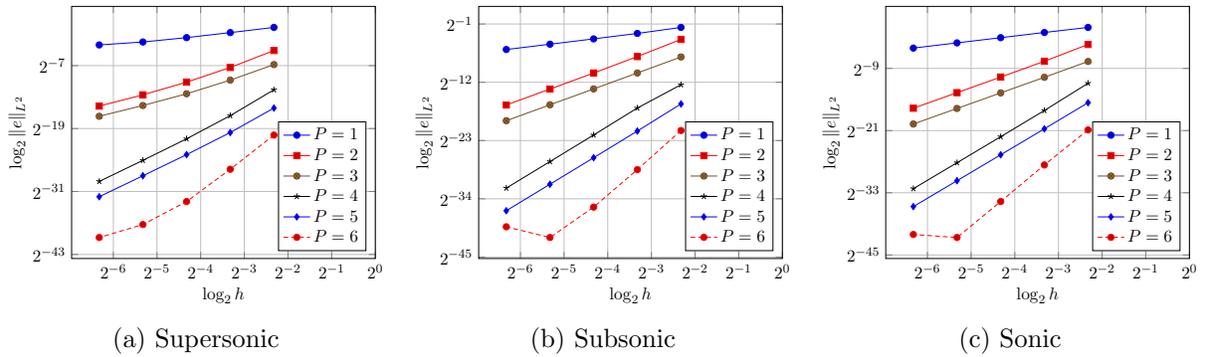

  \begin{table}[h]
    \begin{subtable}[h]{0.45\textwidth}
        \centering
       \begin{tabular}{|c|c|c|c|}
\hline
$P$ & Supersonic & Subsonic &Sonic \\ \hline
1   & 1.0010         & 1.0003  & 1.0014                  \\ \hline
2   & 2.9952         & 3.0256  & 3.0013                 \\ \hline
3   & 3.0010         & 2.9957  & 3.0018                  \\ \hline
4   & 5.0122         & 5.0289  & 5.0095                   \\ \hline
5   & 5.0114         & 5.0127  & 5.0058                   \\ \hline
6   & 7.0995         & 7.0688  & 7.0652                    \\ \hline
\end{tabular}
\caption{ Periodic boundary conditions}
 \label{convtable_2D_periodic}
    \end{subtable}
    \hfill
    \begin{subtable}[h]{0.45\textwidth}
        \centering
       \begin{tabular}{|c|c|c|c|}
\hline
$P$ & Supersonic & Subsonic  & Sonic  \\ \hline
1   & 0.5591         & 1.0014  & 1.0015   \\ \hline
2   & 2.1443         & 3.0097  & 2.9964   \\ \hline
3   & 2.0551         & 3.0003  & 3.0011   \\ \hline
4   & 4.0168         & 5.0059  & 5.0065   \\ \hline
5   & 3.9935         & 5.0073  & 5.0024    \\ \hline
6   & 6.1579         & 7.0714  & 7.0617    \\ \hline
\end{tabular}
\caption{Initial boundary value problem}
\label{convtable_2D_IBVP}
     \end{subtable}
     \caption{Convergence rates for the 2D model problem}
     \label{tab:temps}
\end{table}

\section{Conclusion}
We have developed and analysed a DSEM for the shifted wave equation in second order form. The discretisation is based on spectral difference operators for the first and second derivatives. The operators satisfy the SBP properties and are ultra-compatible, which is important for the conservation and stability for problems with mixed spatial and temporal derivatives. Similar to the DG methods, a spectral difference operator for a single element can be represented by a full matrix, resulting to a block-diagonal structure for the full discretisation operator. The mass matrix in the proposed method is always diagonal thus avoiding any matrix inversion. To couple adjacent elements, we have constructed numerical fluxes for all well-posed material parameters in the subsonic, sonic and supersonic regimes. In each case, an energy estimate is derived to guarantee stability without adding any artificial dissipation. In addition, a priori error estimates are derived in the energy norm. We have presented numerical experiments in both (1+1)- and (2+1)-dimensions to verify the theoretical analysis of conservation, stability and accuracy. With the classical fourth order Runge-Kutta method as the time integrator, the full discretisation is stable under a time step restriction that is proportional to the element size and inversely proportional to the order of local polynomials. The proposed method combines the advantages and central ideas of SBP FD methods, spectral methods and DG methods. 

For the subsonic and sonic regimes, our analysis is valid for all material parameters for well-posed problems. However, in the supersonic regime, the analysis is limited to problems with constant coefficient because the penalised difference operator \eqref{eq:penalise_differential} is not anti-symmetric in the parameter-weighted discrete inner product. This difficulty can be overcome by the flux splitting technique that is commonly used in hyperbolic conservation laws, and will be considered in our forthcoming work. Furthermore, an improved accuracy analysis is needed to understand the observation that the convergence rate in the $L_2$ norm in the supersonic regime is  one order lower than the corresponding sonic and subsonic cases for problems with non-periodic boundary conditions, but this behaviour is not observed for periodic problems.

The method derived in this paper will find  immediate applications in the simulations of aero-acoustic problems and Einsteins' equations modelling gravitational  waves.
In a forthcoming paper, we will apply the numerical method derived in this paper to Einstein's equations of general relativity, to simulate the so-called black hole excision problem \cite{SzilagyiKreissWinicour2005}.

Finally, in this work we have focused on the derivation of accurate and stable numerical approximation in space and used the classical fourth order accurate Runge-Kutta method for time integration. Another possible direction for future work will be the derivation of efficient high order accurate time-stepping schemes for the second order form to match the accuracy of the spatial approximation.
\bibliography{manuscript}
\bibliographystyle{plain} 
\end{document}